\documentclass[a4paper]{article}
\usepackage{latexsym,amsfonts,palatino,eulervm,amsthm,graphicx,verbatim, float, mathdots}
\usepackage[margin=2cm]{geometry}
\usepackage{amsmath}
\usepackage{algorithm}
\usepackage[noend]{algpseudocode}
\usepackage{amsmath,amsthm}
\usepackage{amsfonts}
\usepackage{multicol}
\usepackage{multirow,bigdelim}
\usepackage{fancyvrb}
\usepackage{tkz-graph}
\usepackage[affil-it]{authblk}

\usepackage{mathtools}
\newcommand\SmallMatrix[1]{{%
  \tiny \arraycolsep=0.08\arraycolsep\ensuremath{\begin{pmatrix}#1\end{pmatrix}}}}

\newcommand\Matrix[1]{{%
  \small \arraycolsep=0.8\arraycolsep\ensuremath{\begin{pmatrix}#1\end{pmatrix}}}}

\makeatletter
\def\BState{\State\hskip-\ALG@thistlm}
\makeatother

\algnewcommand\And{\textbf{and} }

\newtheorem{theorem}{Theorem}[section]
\newtheorem{lemma}[theorem]{Lemma}

\newtheorem{problem}[theorem]{Problem}

\theoremstyle{definition}
\newtheorem{definition}[theorem]{Definition}
\newtheorem{example}[theorem]{Example}

\title{Alternating Sign Hypermatrix Decompositions of Latin-like Squares}
\date{}
\author{Cian O'Brien}
\affil{National University of Ireland, Galway}

\begin{document}
\maketitle

\setlength{\parindent}{0pt}

\vspace{-0.7cm}\begin{abstract}
\noindent To any $n \times n$ Latin square $L$, we may associate a unique sequence of mutually orthogonal permutation matrices $P = P_1, P_2, ..., P_n$ such that $L = L(P) = \sum kP_k$. Brualdi and Dahl (2018) described a generalisation of a Latin square, called an \emph{alternating sign hypermatrix Latin-like square (ASHL)}, by replacing $P$ with an \emph{alternating sign hypermatrix (ASHM)}. An ASHM is an $n \times n \times n$ (0,1,-1)-hypermatrix in which the non-zero elements in each row, column, and vertical line alternate in sign, beginning and ending with $1$. Since every sequence of $n$ mutually orthogonal permutation matrices forms the planes of a unique $n \times n \times n$ ASHM, this generalisation of Latin squares follows very naturally, with an ASHM $A$ having corresponding ASHL $L = L(A) =\sum kA_k$, where $A_k$ is the $k^{\text{th}}$ plane of $A$. This paper addresses some open problems posed in Brualdi and Dahl's article, firstly by characterising how pairs of ASHMs with the same corresponding ASHL relate to one another and providing a tight lower bound on $n$ for which two $n \times n \times n$ ASHMs can correspond to the same ASHL, and secondly by exploring the maximum number of times a particular integer may occur as an entry of an $n \times n$ ASHL. A general construction is given for an $n \times n$ ASHL with the same entry occurring $\lfloor\frac{n^2 + 4n -19}{2}\rfloor$ times, improving considerably on the previous best construction, which achieved the same entry occuring $2n$ times.
\end{abstract}

\section{Introduction}

\begin{definition}
An \emph{alternating sign matrix (ASM)}, is a ${(0,1,-1)}$-matrix in which the non-zero elements in each row and column alternate in sign, beginning and ending with $1$.
\end{definition}

\begin{definition}
The $n \times n$ \emph{diamond ASM} is an ASM with the maximum number of non-zero entries for given $n$. For odd $n$, there is a unique diamond ASM $D_n$. For even $n$, there are two diamond ASMs $D_n$ and $D_n'$.
\end{definition}

\begin{example}
The following are the two $4 \times 4$ diamond ASMs $D_4$ and $D_4'$, and the $5 \times 5$ diamond ASM $D_5$.

\includegraphics[width=\textwidth]{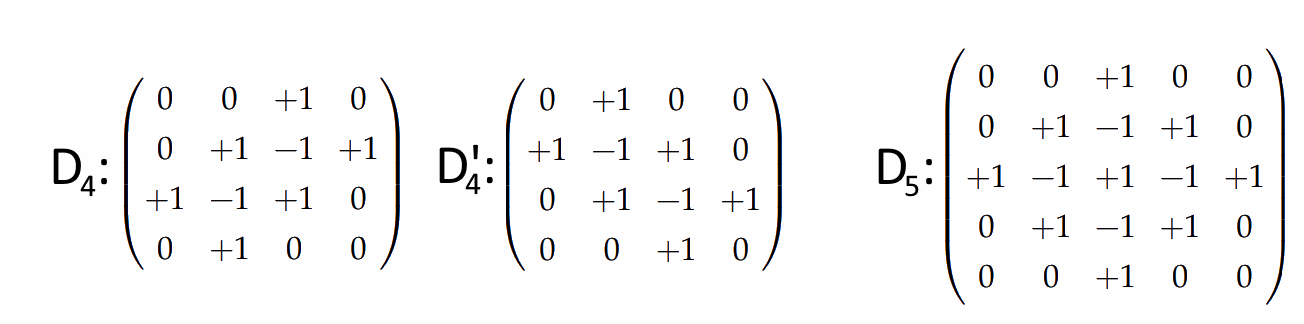}
\end{example}

It is easily observed that permutation matrices are examples of ASMs, which arise naturally as the unique smallest lattice containing the permutation matrices of the \emph{Bruhat order} \cite{lattice}. 

\begin{definition}
A \emph{Latin square} of order $n$ is an $n \times n$ array containing $n$ symbols such that each symbol occurs exactly once in each row and column.
\end{definition}

Any $n \times n$ Latin square $L$ with symbols $1, 2, \dots, n$ can be decomposed into a unique sequence $P$ of $n \times n$ mutually orthogonal permutation matrices $P = P_1, P_2, \dots, P_n$ by the following relation.

\[L = L(P) = \sum_k kP_k\]

For example, consider the following latin square.

\[\Matrix{
    1 & 2 & 3 \\
    2 & 3 & 1 \\
    3 & 1 & 2} 
=
1\Matrix{
      1 & 0 & 0 \\
      0 & 0 & 1 \\
      0 & 1 & 0}
+2\Matrix{
      0 & 1 & 0 \\
      1 & 0 & 0 \\
      0 & 0 & 1}
+3\Matrix{
      0 & 0 & 1 \\
      0 & 1 & 0 \\
      1 & 0 & 0}\]

This decomposition leads to a natural extension of the concept of a Latin square, first introduced by Brualdi and Dahl \cite{ashmbib}, by replacing the sequence of permutation matrices with planes of an \emph{alternating sign hypermatrix (ASHM)}. Before we define these objects, we must first define some features of a hypermatrix.

\*

An $n \times n \times n$ hypermatrix $A = [a_{ijk}]$ has has $n^2$ lines of each of the $3$ following types. Each line has $n$ entries.
\begin{itemize}
\item \emph{Row lines} $A_{*jk} = [a_{ijk} : i = 1, \dots, n]$, for given $1 \leq j,k \leq n$;
\item \emph{Column lines} $A_{i*k} = [a_{ijk} : j = 1, \dots, n]$, for given $1 \leq i,k \leq n$;
\item \emph{Vertical lines} $A_{ij*} = [a_{ijk} : k = 1, \dots, n]$, for given $1 \leq i,j \leq n$.
\end{itemize}
In this paper, we refer to a \emph{plane} $P_k(A)$ of $A$ to be the horizontal plane $A_{**k} = [a_{ijk}:i,j=1,\dots,n]$, for given $1 \leq k \leq n$.

\begin{definition} An \emph{alternating sign hypermatrix (ASHM)} is a $(0,\pm 1)$-hypermatrix for which the non-zero entries in each row, column, and vertical line of the hypermatrix alternate in sign, starting and ending with $+1$. \end{definition}

For example, the following is a $3 \times 3$ ASHM.

\[A = \Matrix{
      0 & 0 & 1 \\
      0 & 1 & 0 \\
      1 & 0 & 0}
\hspace{-0.1cm}\nearrow\hspace{-0.1cm}\Matrix{
      0 & 1 & 0 \\
      1 & -1 & 1 \\
      0 & 1 & 0}
\hspace{-0.1cm}\nearrow\hspace{-0.1cm}\Matrix{
      1 & 0 & 0 \\
      0 & 1 & 0 \\
      0 & 0 & 1}\]

The north-east arrow $P_k(A) \nearrow P_{k+1}(A)$ is used to denote that $P_k(A)$ is below $P_{k+1}$. For simplicity, for the remainder of this paper, we will omit the zero entries of all ASHMs, and represent all $\pm 1$ entries with $+$ or $-$.

\begin{definition} An \emph{ASHL} is an $n \times n$ matrix $L$ constructed from an $n \times n \times n$ ASHM $A$ by 
\[L = L(A) = \sum_k kP_k(A)\text{.}\] \end{definition}

From the previous example, we then have the following ASHL.
\[L(A) = 1\Matrix{
       &  & + \\
       & + &  \\
      + &  & }
+2\Matrix{
       & + &  \\
      + & - & + \\
       & + & }
+3\Matrix{
      + &  &  \\
       & + &  \\
       &  & +}
=\Matrix{
    3 & 2 & 1 \\
    2 & 2 & 2 \\
    1 & 2 & 3}\]

Brualdi and Dahl \cite{ashmbib} posed a number problems about ASHLs, the following of which are addressed in this paper.
\begin{itemize}
\item Given an $n \times n$ ASHL $L$, let $A_n(L)$ be the set of all $n \times n \times n$ ASHMs $A$ such that $L(A) = L$. Investigate $A_n(L)$.
\item What is the maximum number of times an integer can occur as an entry of an $n \times n$ ASHL?
\end{itemize}
In this paper, the relationship between two ASHMs that generate the same ASHL is described, and a general construction is given for building an $n \times n$ ASHL containing $\lfloor\frac{n^2 + 4n -19}{2}\rfloor$ copies of one symbol. This improves on Brualdi and Dahl's lower bound of $2n$ for the maximum number of times that one symbol can appear in an $n \times n$ ASHL.


\section{ASHLs With Multiple ASHM-Decompositions}

In Brualdi and Dahl's paper \cite{ashmbib}, it was proven that for a Latin square $L$, if $L = L(A)$ for some ASHM $A$, then $A$ must be a permutation hypermatrix. The following question was then posed.

\begin{problem}\label{problem1} Given an $n \times n$ ASHL $L$, let $A_n(L)$ be the set of all $n \times n \times n$ ASHMs $A$ such that $L(A) = L$. Investigate $A_n(L)$. \end{problem}

This is presented in \cite{ashmbib} as a completely open problem, with no examples of distinct ASHMs with the same corresponding ASHL given. This problem is motivated by the observation that in the case of a Latin square $L$, $A_n(L)$ contains exactly one ASHM whose planes form a set of mutually orthogonal permutation matrices. Before we discuss how two ASHMs in $A_n(L)$ relate to one another, it is useful to introduce the following definition, and more generally examine how any pair of $n \times n \times n$ ASHMs relate to one another.

\begin{definition} A \emph{T-block} $T_{i_1,j_1,k_1:\,i_2,j_2,k_2}$ is a hypermatrix $[t_{ijk}]$ such that
\[t_{ijk} =  \begin{cases} 
      \color{white}-\color{black}1  & (i,j,k) = (i_1,j_1,k_1), (i_2,j_2,k_1), (i_2,j_1,k_2), \text{or} (i_1,j_2,k_2) \\
      -1 & (i,j,k) = (i_2,j_1,k_1), (i_1,j_2,k_1), (i_1,j_1,k_2), \text{or} (i_2,j_2,k_2) \\
      \color{white}-\color{black}0  & \text{otherwise} 
   \end{cases}
\] where $i_1 < i_2, j_1 < j_2$, and $k_1 < k_2$.\end{definition}

A T-block can be most usefully visualised as an $n \times n \times n$ matrix containing the subhypermatrix
\[\pm \left[\Matrix{
    + & -\\
    - & +}
\nearrow
\Matrix{
    - & +\\
    + & -}\right]\]
such that these are the only non-zero entries. This is a 3-d extension of a concept defined by Brualdi, Kiernan, Meyer, and Schroeder \cite{brualdibib}.

\*

The following lemma will be needed to investigate $A_n(L)$.

\begin{lemma}\label{cube_decomp_lemma} Let $A$ and $B$ be two $n \times n \times n$ ASHMs. Then $A-B$ can be expressed as a sum of T-blocks.\end{lemma}

\begin{proof}
If $A = B$, this is trivially true.

Assume $A \not = B$. In $A$ and $B$, the sum of the entries in any row, column, or vertical line is $1$. Therefore the sum of any row, column, or vertical line of $D = A-B$ is $0$. Now iterate the following step.
\begin{itemize}
\item Let $k_1$ be the least integer for which the plane $P_{k_1}(D)$ contains non-zero entries. For some positive entry $d_{i_1j_1k_1}$ of $D$, we can find entries $d_{i_2j_1k_1}$, $d_{i_1j_2k_1}$ and $d_{i_1j_1k_2}$ with negative sign, where $k_2 > k_1$. Let $D \leftarrow D - T_{i_1,j_1,k_1:\,i_2,j_2,k_2}$ and repeat this step if $D$ is not a $0$-hypermatrix.
\end{itemize}
Note that the sum of the absolute values of the entries of $P_{k_1}(D)$ is at least $2$ less than the previous step, and that all line sums of $D$ are $0$ after each step. We can therefore run this iterative process repeatedly, resulting in $P_{k_1}(D)$ becoming a 0-matrix for the current value of $k_1$, and eventually resulting in $D$ becoming a 0-hypermatrix. Therefore $A-B$ can be expressed as a sum of T-blocks.
\end{proof}

\begin{definition} Let $T$ be a T-block. The \emph{depth} of $T$, $d(T)$, is defined as follows.
\[ d(T) =  \begin{cases} 
      k_2 - k_1\text{,} & T = T_{i_1,j_1,k_1:\,i_2,j_2,k_2} \\
      k_1 - k_2\text{,} & T = -T_{i_1,j_1,k_1:\,i_2,j_2,k_2}
   \end{cases}
\]
Two T-blocks, $T_1$ and $T_2$, have \emph{opposite depth} if $d(T_1) = -d(T_2)$
\end{definition}

\begin{theorem}\label{ASHL_difference} Two ASHMs $A$ and $B$ satisfy $L(A) = L(B)$ if and only if any expression of $A-B$ as a sum of T-blocks satisfies that, in any vertical line $V$ of $A-B$,
\[\sum_{T \in T_V} d(T) = 0\text{,}\]
where $T_V$ is the subset of these T-blocks with non-zero entries in $V$.
\end{theorem}

\begin{proof}
From Lemma \ref{cube_decomp_lemma}, we know that $A - B$ can be expressed as a sum of T-blocks. The entries of any vertical line $V$ in $A-B$ can be decomposed into pairs $(t_{k_1},t_{k_2})$, where $k_1 < k_2$, such that pairs are non-zero entries in the same T-block. Note that $t_{k_1} = -t_{k_2}$. Therefore 
\[\sum kV_k = \sum_{T \in T_V} k_1t_{k_1} + k_2t_{k_2} = \sum_{T \in T_V} \pm (k_1 - k_2) = \sum_{T \in T_V} d(T)\text{.}\]
\begin{itemize}
\item Suppose that, for any vertical line $V$ in $A - B$, 
\[\sum_{T \in T_V} d(T) = 0\text{.}\]
This means that $\sum kV_k = 0$, which means that $L(A-B) = 0$. Therefore $L(A) = L(B)$.

\item Now suppose that $L(A) = L(B)$. Then $L(A-B) = 0$, so for any vertical line $V$ in $A-B$, we have $\sum kV_k = 0$. Therefore \[\sum_{T \in T_V} d(T) = 0\text{.}\]
\end{itemize}
\end{proof}

This provides some progress on Problem \ref{problem1}, as an ASHM $B$ is contained in $A_n(L)$ if and only $A$ and $B$ satisfy Theorem \ref{ASHL_difference} for all $A \in A_n(L)$. In particular, Theorem \ref{ASHL_difference} provides a strategy for constructing an ASHM $B$ for which $L(B) = L(A)$ for some given ASHM $A$, by adding T-blocks to $A$ in such a way that satisfies $\sum_{T \in T_V} d(T) = 0$. This is by far our most successful method for generating pairs of ASHMs with the same corresponding ASHL. The relationship between elements of $A_n(L)$ can be characterised further, as shown in the following two examples.

\begin{example}\label{uneven_distance} Here, $A$ and $B$ are two ASHMs for which $D=A-B$ has a very natural decomposition as the sum of three T-blocks with depths $1, 1, -2$, respectively, occupying the same vertical lines.
\[ A = 
\SmallMatrix{
      &  & \textbf{+} &  &  &  &  &  \\
     \textbf{+} &  &  &  &  &  &  &  \\
      & \textbf{+} &  &  &  &  &  &  \\
      &  &  & \textbf{+} &  &  &  &  \\
      &  &  &  & \textbf{+} &  &  &  \\
      &  &  &  &  &  & \textbf{+} &  \\
      &  &  &  &  &  &  & \textbf{+} \\
      &  &  &  &  & \textbf{+} &  & }
\hspace{-0.2cm}\nearrow\hspace{-0.2cm}\SmallMatrix{
      & \textbf{+} &  &  &  &  &  &  \\
      &  & \textbf{+} &  &  &  &  &  \\
     \textbf{+} &  &  &  &  &  &  &  \\
      &  &  &  & \textbf{+} &  &  &  \\
      &  &  & \textbf{+} &  &  &  &  \\
      &  &  &  &  &  &  & \textbf{+} \\
      &  &  &  &  & \textbf{+} &  &  \\
      &  &  &  &  &  & \textbf{+} & }
\hspace{-0.2cm}\nearrow\hspace{-0.2cm}\SmallMatrix{
      &  &  &  &  &  &  & \textbf{+} \\
      &  &  &  & \textbf{+} &  &  &  \\
      &  &  & \textbf{+} &  &  &  &  \\
      &  & \textbf{+} &  & \textbf{-} &  & \textbf{+} &  \\
      & \textbf{+} &  & \textbf{-} &  & \textbf{+} &  &  \\
      &  &  &  & \textbf{+} &  &  &  \\
      &  &  & \textbf{+} &  &  &  &  \\
     \textbf{+} &  &  &  &  &  &  & }
\hspace{-0.2cm}\nearrow\hspace{-0.2cm}\SmallMatrix{
      &  &  &  &  & \textbf{+} &  &  \\
      &  &  &  &  &  &  & \textbf{+} \\
      &  &  &  &  &  & \textbf{+} &  \\
      &  &  &  & \textbf{+} &  &  &  \\
      &  &  & \textbf{+} &  &  &  &  \\
     \textbf{+} &  &  &  &  &  &  &  \\
      &  & \textbf{+} &  &  &  &  &  \\
      & \textbf{+} &  &  &  &  &  & }
\hspace{-0.2cm}\nearrow\hspace{-0.2cm}\SmallMatrix{
      &  &  & \textbf{+} &  &  &  &  \\
      & \textbf{+} &  &  &  &  &  &  \\
      &  & \textbf{+} &  &  &  &  &  \\
     \textbf{+} &  &  &  &  &  &  &  \\
      &  &  &  &  &  &  & \textbf{+} \\
      &  &  &  &  & \textbf{+} &  &  \\
      &  &  &  &  &  & \textbf{+} &  \\
      &  &  &  & \textbf{+} &  &  & }
\hspace{-0.2cm}\nearrow\hspace{-0.2cm}\SmallMatrix{
     \textbf{+} &  &  &  &  &  &  &  \\
      &  &  & \textbf{+} &  &  &  &  \\
      &  &  &  & \textbf{+} &  &  &  \\
      & \textbf{+} &  & \textbf{-} &  & \textbf{+} &  &  \\
      &  & \textbf{+} &  & \textbf{-} &  & \textbf{+} &  \\
      &  &  & \textbf{+} &  &  &  &  \\
      &  &  &  & \textbf{+} &  &  &  \\
      &  &  &  &  &  &  & \textbf{+}}
\hspace{-0.2cm}\nearrow\hspace{-0.2cm}\SmallMatrix{
      &  &  &  & \textbf{+} &  &  &  \\
      &  &  &  &  &  & \textbf{+} &  \\
      &  &  &  &  & \textbf{+} &  &  \\
      &  &  &  &  &  &  & \textbf{+} \\
     \textbf{+} &  &  &  &  &  &  &  \\
      &  & \textbf{+} &  &  &  &  &  \\
      & \textbf{+} &  &  &  &  &  &  \\
      &  &  & \textbf{+} &  &  &  & }
\hspace{-0.2cm}\nearrow\hspace{-0.2cm}\SmallMatrix{
      &  &  &  &  &  & \textbf{+} &  \\
      &  &  &  &  & \textbf{+} &  &  \\
      &  &  &  &  &  &  & \textbf{+} \\
      &  &  & \textbf{+} &  &  &  &  \\
      &  &  &  & \textbf{+} &  &  &  \\
      & \textbf{+} &  &  &  &  &  &  \\
     \textbf{+} &  &  &  &  &  &  &  \\
      &  & \textbf{+} &  &  &  &  & }\]


\[ B = 
\SmallMatrix{
      &  & \textbf{+} &  &  &  &  &  \\
     \textbf{+} &  &  &  &  &  &  &  \\
      & \textbf{+} &  &  &  &  &  &  \\
      &  &  &  & \textbf{+} &  &  &  \\
      &  &  & \textbf{+} &  &  &  &  \\
      &  &  &  &  &  & \textbf{+} &  \\
      &  &  &  &  &  &  & \textbf{+} \\
      &  &  &  &  & \textbf{+} &  & }
\hspace{-0.2cm}\nearrow\hspace{-0.2cm}\SmallMatrix{
      & \textbf{+} &  &  &  &  &  &  \\
      &  & \textbf{+} &  &  &  &  &  \\
     \textbf{+} &  &  &  &  &  &  &  \\
      &  &  & \textbf{+} &  &  &  &  \\
      &  &  &  & \textbf{+} &  &  &  \\
      &  &  &  &  &  &  & \textbf{+} \\
      &  &  &  &  & \textbf{+} &  &  \\
      &  &  &  &  &  & \textbf{+} & }
\hspace{-0.2cm}\nearrow\hspace{-0.2cm}\SmallMatrix{
      &  &  &  &  &  &  & \textbf{+} \\
      &  &  &  & \textbf{+} &  &  &  \\
      &  &  & \textbf{+} &  &  &  &  \\
      &  & \textbf{+} & \textbf{-} &  &  & \textbf{+} &  \\
      & \textbf{+} &  &  & \textbf{-} & \textbf{+} &  &  \\
      &  &  &  & \textbf{+} &  &  &  \\
      &  &  & \textbf{+} &  &  &  &  \\
     \textbf{+} &  &  &  &  &  &  & }
\hspace{-0.2cm}\nearrow\hspace{-0.2cm}\SmallMatrix{
      &  &  &  &  & \textbf{+} &  &  \\
      &  &  &  &  &  &  & \textbf{+} \\
      &  &  &  &  &  & \textbf{+} &  \\
      &  &  & \textbf{+} &  &  &  &  \\
      &  &  &  & \textbf{+} &  &  &  \\
     \textbf{+} &  &  &  &  &  &  &  \\
      &  & \textbf{+} &  &  &  &  &  \\
      & \textbf{+} &  &  &  &  &  & }
\hspace{-0.2cm}\nearrow\hspace{-0.2cm}\SmallMatrix{
      &  &  & \textbf{+} &  &  &  &  \\
      & \textbf{+} &  &  &  &  &  &  \\
      &  & \textbf{+} &  &  &  &  &  \\
     \textbf{+} &  &  &  &  &  &  &  \\
      &  &  &  &  &  &  & \textbf{+} \\
      &  &  &  &  & \textbf{+} &  &  \\
      &  &  &  &  &  & \textbf{+} &  \\
      &  &  &  & \textbf{+} &  &  & }
\hspace{-0.2cm}\nearrow\hspace{-0.2cm}\SmallMatrix{
     \textbf{+} &  &  &  &  &  &  &  \\
      &  &  & \textbf{+} &  &  &  &  \\
      &  &  &  & \textbf{+} &  &  &  \\
      & \textbf{+} &  &  & \textbf{-} & \textbf{+} &  &  \\
      &  & \textbf{+} & \textbf{-} &  &  & \textbf{+} &  \\
      &  &  & \textbf{+} &  &  &  &  \\
      &  &  &  & \textbf{+} &  &  &  \\
      &  &  &  &  &  &  & \textbf{+}}
\hspace{-0.2cm}\nearrow\hspace{-0.2cm}\SmallMatrix{
      &  &  &  & \textbf{+} &  &  &  \\
      &  &  &  &  &  & \textbf{+} &  \\
      &  &  &  &  & \textbf{+} &  &  \\
      &  &  &  &  &  &  & \textbf{+} \\
     \textbf{+} &  &  &  &  &  &  &  \\
      &  & \textbf{+} &  &  &  &  &  \\
      & \textbf{+} &  &  &  &  &  &  \\
      &  &  & \textbf{+} &  &  &  & }
\hspace{-0.2cm}\nearrow\hspace{-0.2cm}\SmallMatrix{
      &  &  &  &  &  & \textbf{+} &  \\
      &  &  &  &  & \textbf{+} &  &  \\
      &  &  &  &  &  &  & \textbf{+} \\
      &  &  &  & \textbf{+} &  &  &  \\
      &  &  & \textbf{+} &  &  &  &  \\
      & \textbf{+} &  &  &  &  &  &  \\
     \textbf{+} &  &  &  &  &  &  &  \\
      &  & \textbf{+} &  &  &  &  & }\]


\[ D = 
\SmallMatrix{
      \color{white}\textbf{+} &  &  &  &  &  &  &  \\
      & \color{white}\textbf{+} &  &  &  &  &  &  \\
      &  & \color{white}\textbf{+} &  &  &  &  &  \\
      \color{black!70}\cdot& \color{black!70}\cdot & \color{black!70}\cdot & \textbf{+} & \textbf{-} & \color{black!70}\cdot & \color{black!70}\cdot & \color{black!70}\cdot \\
      \color{black!70}\cdot& \color{black!70}\cdot & \color{black!70}\cdot & \textbf{-} & \textbf{+} & \color{black!70}\cdot & \color{black!70}\cdot & \color{black!70}\cdot \\
      &  &  &  &  & \color{white}\textbf{+} &  &  \\
      &  &  &  &  &  & \color{white}\textbf{+} &  \\
      &  &  &  &  &  &  &\color{white}\textbf{+} }
\hspace{-0.2cm}\nearrow\hspace{-0.2cm}\SmallMatrix{
     \color{white}\textbf{+} &  &  &  &  &  &  &  \\
      & \color{white}\textbf{+} &  &  &  &  &  &  \\
      &  & \color{white}\textbf{+} &  &  &  &  &  \\
      \color{black!70}\cdot& \color{black!70}\cdot &\color{black!70}\cdot  & \textbf{-} & \textbf{+} & \color{black!70}\cdot & \color{black!70}\cdot & \color{black!70}\cdot \\
      \color{black!70}\cdot& \color{black!70}\cdot & \color{black!70}\cdot & \textbf{+} & \textbf{-} & \color{black!70}\cdot & \color{black!70}\cdot & \color{black!70}\cdot \\
      &  &  &  &  &\color{white}\textbf{+}  &  &  \\
      &  &  &  &  &  &\color{white}\textbf{+} &  \\
      &  &  &  &  &  &  & \color{white}\textbf{+}}
\hspace{-0.2cm}\nearrow\hspace{-0.2cm}\SmallMatrix{
     \color{white}\textbf{+} &  &  &  &  &  &  &  \\
      & \color{white}\textbf{+} &  &  &  &  &  &  \\
      &  & \color{white}\textbf{+} &  &  &  &  &  \\
      \color{black!70}\cdot& \color{black!70}\cdot & \color{black!70}\cdot & \textbf{+} & \textbf{-} & \color{black!70}\cdot & \color{black!70}\cdot & \color{black!70}\cdot \\
      \color{black!70}\cdot& \color{black!70}\cdot & \color{black!70}\cdot & \textbf{-} & \textbf{+} & \color{black!70}\cdot & \color{black!70}\cdot & \color{black!70}\cdot \\
      &  &  &  &  & \color{white}\textbf{+} &  &  \\
      &  &  &  &  &  & \color{white}\textbf{+} &  \\
      &  &  &  &  &  &  &\color{white}\textbf{+} }
\hspace{-0.2cm}\nearrow\hspace{-0.2cm}\SmallMatrix{
     \color{white}\textbf{+} &  &  &  &  &  &  &  \\
      & \color{white}\textbf{+} &  &  &  &  &  &  \\
      &  & \color{white}\textbf{+} &  &  &  &  &  \\
      &  &  & \textbf{-} & \textbf{+} &  &  &  \\
      &  &  & \textbf{+} & \textbf{-} &  &  &  \\
      &  &  &  &  & \color{white}\textbf{+} &  &  \\
      &  &  &  &  &  &\color{white}\textbf{+}  &  \\
      &  &  &  &  &  &  & \color{white}\textbf{+}}
\hspace{-0.2cm}\nearrow\hspace{-0.2cm}\SmallMatrix{
      \color{white}\textbf{+}&  &  &  &  &  &  &  \\
      & \color{white}\textbf{+} &  &  &  &  &  &  \\
      &  &\color{white}\textbf{+}  &  &  &  &  &  \\
      &  &  & \color{white}\textbf{+} &  &  &  &  \\
      &  &  &  &\color{white}\textbf{+} &  &  &  \\
      &  &  &  &  & \color{white}\textbf{+}&  &  \\
      &  &  &  &  &  &\color{white}\textbf{+}  &  \\
      &  &  &  &  &  &  & \color{white}\textbf{+}}
\hspace{-0.2cm}\nearrow\hspace{-0.2cm}\SmallMatrix{
      \color{white}\textbf{+}&  &  &  &  &  &  &  \\
      & \color{white}\textbf{+} &  &  &  &  &  &  \\
      &  &\color{white}\textbf{+}  &  &  &  &  &  \\
      \color{black!70}\cdot& \color{black!70}\cdot &\color{black!70}\cdot  & \textbf{-} & \textbf{+} & \color{black!70}\cdot & \color{black!70}\cdot & \color{black!70}\cdot \\
      \color{black!70}\cdot& \color{black!70}\cdot & \color{black!70}\cdot & \textbf{+} & \textbf{-} & \color{black!70}\cdot & \color{black!70}\cdot & \color{black!70}\cdot \\
      &  &  &  &  & \color{white}\textbf{+} &  &  \\
      &  &  &  &  &  &\color{white}\textbf{+}  &  \\
      &  &  &  &  &  &  &\color{white}\textbf{+} }
\hspace{-0.2cm}\nearrow\hspace{-0.2cm}\SmallMatrix{
     \color{white}\textbf{+} &  &  &  &  &  &  &  \\
      & \color{white}\textbf{+} &  &  &  &  &  &  \\
      &  &\color{white}\textbf{+}  &  &  &  &  &  \\
      &  &  & \color{white}\textbf{+} &  &  &  &  \\
      &  &  &  & \color{white}\textbf{+} &  &  &  \\
      &  &  &  &  & \color{white}\textbf{+} &  &  \\
      &  &  &  &  &  & \color{white}\textbf{+} &  \\
      &  &  &  &  &  &  &\color{white}\textbf{+} }
\hspace{-0.2cm}\nearrow\hspace{-0.2cm}\SmallMatrix{
     \color{white}\textbf{+} &  &  &  &  &  &  &  \\
      &  \color{white}\textbf{+}&  &  &  &  &  &  \\
      &  & \color{white}\textbf{+} &  &  &  &  &  \\
      \color{black!70}\cdot& \color{black!70}\cdot & \color{black!70}\cdot & \textbf{+} & \textbf{-} & \color{black!70}\cdot & \color{black!70}\cdot & \color{black!70}\cdot \\
      \color{black!70}\cdot& \color{black!70}\cdot & \color{black!70}\cdot & \textbf{-} & \textbf{+} & \color{black!70}\cdot & \color{black!70}\cdot & \color{black!70}\cdot \\
      &  &  &  &  & \color{white}\textbf{+} &  &  \\
      &  &  &  &  &  & \color{white}\textbf{+} &  \\
      &  &  &  &  &  &  &\color{white}\textbf{+} }\]

\[D = T_{4,4,1:\,5,5,2} + T_{4,4,3:\,5,5,4} - T_{4,4,6:\,5,5,8}\]

\[L(A) = L(B) = 
\Matrix{
    6 & 2 & 1 & 5 & 7 & 4 & 8 & 3 \\
    1 & 5 & 2 & 6 & 3 & 8 & 7 & 4 \\
    2 & 1 & 5 & 3 & 6 & 7 & 4 & 8 \\
    5 & 6 & 3 & 3 & 3 & 6 & 3 & 7 \\
    7 & 3 & 6 & 3 & 3 & 3 & 6 & 5 \\
    4 & 8 & 7 & 6 & 3 & 5 & 1 & 2 \\
    8 & 7 & 4 & 3 & 6 & 2 & 5 & 1 \\
    3 & 4 & 8 & 7 & 5 & 1 & 2 & 6}\]

$D$ can also be expressed as the sum of pairs of T-blocks with opposite depth occupying the same vertical lines.

\[D = (T_{4,4,1:\,5,5,2} - T_{4,4,7:\,5,5,8}) + (T_{4,4,3:\,5,5,4} - T_{4,4,6:\,5,5,7})\]

\end{example}


\begin{example}\label{interlaced}
Here, $A$ and $B$ are two ASHMs for which $D=A-B$ has a very natural decomposition as the sum of three T-blocks such that each pair occupy exactly two of the same vertical lines.
\[\hspace{-0.4cm}A = 
\SmallMatrix{
     \textbf{+} &  &  &  &  &  &  &  &  \\
      & \textbf{+} &  &  &  &  &  &  &  \\
      &  & \textbf{+} &  &  &  &  &  &  \\
      &  &  &  & \textbf{+} &  &  &  &  \\
      &  &  & \textbf{+} &  &  &  &  &  \\
      &  &  &  &  & \textbf{+} &  &  &  \\
      &  &  &  &  &  & \textbf{+} &  &  \\
      &  &  &  &  &  &  & \textbf{+} &  \\
      &  &  &  &  &  &  &  & \textbf{+}}
\hspace{-0.2cm}\nearrow\hspace{-0.2cm}\SmallMatrix{
      & \textbf{+} &  &  &  &  &  &  &  \\
      &  & \textbf{+} &  &  &  &  &  &  \\
     \textbf{+} &  &  &  &  &  &  &  &  \\
      &  &  &  &  & \textbf{+} &  &  &  \\
      &  &  &  & \textbf{+} &  &  &  &  \\
      &  &  & \textbf{+} &  &  &  &  &  \\
      &  &  &  &  &  &  & \textbf{+} &  \\
      &  &  &  &  &  &  &  & \textbf{+} \\
      &  &  &  &  &  & \textbf{+} &  & }
\hspace{-0.2cm}\nearrow\hspace{-0.2cm}\SmallMatrix{
      &  & \textbf{+} &  &  &  &  &  &  \\
     \textbf{+} &  &  &  &  &  &  &  &  \\
      & \textbf{+} &  &  &  &  &  &  &  \\
      &  &  & \textbf{+} &  &  &  &  &  \\
      &  &  &  &  & \textbf{+} &  &  &  \\
      &  &  &  & \textbf{+} &  &  &  &  \\
      &  &  &  &  &  &  &  & \textbf{+} \\
      &  &  &  &  &  & \textbf{+} &  &  \\
      &  &  &  &  &  &  & \textbf{+} & }
\hspace{-0.2cm}\nearrow\hspace{-0.2cm}\SmallMatrix{
      &  &  & \textbf{+} &  &  &  &  &  \\
      &  &  &  & \textbf{+} &  &  &  &  \\
      &  &  &  &  & \textbf{+} &  &  &  \\
     \textbf{+} &  &  &  & \textbf{-} &  & \textbf{+} &  &  \\
      & \textbf{+} &  & \textbf{-} &  &  &  & \textbf{+} &  \\
      &  & \textbf{+} &  &  & \textbf{-} &  &  & \textbf{+} \\
      &  &  & \textbf{+} &  &  &  &  &  \\
      &  &  &  & \textbf{+} &  &  &  &  \\
      &  &  &  &  & \textbf{+} &  &  & }
\hspace{-0.2cm}\nearrow\hspace{-0.2cm}\SmallMatrix{
      &  &  &  & \textbf{+} &  &  &  &  \\
      &  &  &  &  & \textbf{+} &  &  &  \\
      &  &  & \textbf{+} &  &  &  &  &  \\
      & \textbf{+} &  &  &  & \textbf{-} &  & \textbf{+} &  \\
      &  & \textbf{+} &  & \textbf{-} &  &  &  & \textbf{+} \\
     \textbf{+} &  &  & \textbf{-} &  &  & \textbf{+} &  &  \\
      &  &  &  & \textbf{+} &  &  &  &  \\
      &  &  &  &  & \textbf{+} &  &  &  \\
      &  &  & \textbf{+} &  &  &  &  & }
\hspace{-0.2cm}\nearrow\hspace{-0.2cm}\SmallMatrix{
      &  &  &  &  & \textbf{+} &  &  &  \\
      &  &  & \textbf{+} &  &  &  &  &  \\
      &  &  &  & \textbf{+} &  &  &  &  \\
      &  & \textbf{+} & \textbf{-} &  &  &  &  & \textbf{+} \\
     \textbf{+} &  &  &  &  & \textbf{-} & \textbf{+} &  &  \\
      & \textbf{+} &  &  & \textbf{-} &  &  & \textbf{+} &  \\
      &  &  &  &  & \textbf{+} &  &  &  \\
      &  &  & \textbf{+} &  &  &  &  &  \\
      &  &  &  & \textbf{+} &  &  &  & }
\hspace{-0.2cm}\nearrow\hspace{-0.2cm}\SmallMatrix{
      &  &  &  &  &  & \textbf{+} &  &  \\
      &  &  &  &  &  &  & \textbf{+} &  \\
      &  &  &  &  &  &  &  & \textbf{+} \\
      &  &  &  &  & \textbf{+} &  &  &  \\
      &  &  & \textbf{+} &  &  &  &  &  \\
      &  &  &  & \textbf{+} &  &  &  &  \\
     \textbf{+} &  &  &  &  &  &  &  &  \\
      & \textbf{+} &  &  &  &  &  &  &  \\
      &  & \textbf{+} &  &  &  &  &  & }
\hspace{-0.2cm}\nearrow\hspace{-0.2cm}\SmallMatrix{
      &  &  &  &  &  &  & \textbf{+} &  \\
      &  &  &  &  &  &  &  & \textbf{+} \\
      &  &  &  &  &  & \textbf{+} &  &  \\
      &  &  &  & \textbf{+} &  &  &  &  \\
      &  &  &  &  & \textbf{+} &  &  &  \\
      &  &  & \textbf{+} &  &  &  &  &  \\
      & \textbf{+} &  &  &  &  &  &  &  \\
      &  & \textbf{+} &  &  &  &  &  &  \\
     \textbf{+} &  &  &  &  &  &  &  & }
\hspace{-0.2cm}\nearrow\hspace{-0.2cm}\SmallMatrix{
      &  &  &  &  &  &  &  & \textbf{+} \\
      &  &  &  &  &  & \textbf{+} &  &  \\
      &  &  &  &  &  &  & \textbf{+} &  \\
      &  &  & \textbf{+} &  &  &  &  &  \\
      &  &  &  & \textbf{+} &  &  &  &  \\
      &  &  &  &  & \textbf{+} &  &  &  \\
      &  & \textbf{+} &  &  &  &  &  &  \\
     \textbf{+} &  &  &  &  &  &  &  &  \\
      & \textbf{+} &  &  &  &  &  &  & }\]


\[\hspace{-0.4cm}B = 
\SmallMatrix{
     \textbf{+} &  &  &  &  &  &  &  &  \\
      & \textbf{+} &  &  &  &  &  &  &  \\
      &  & \textbf{+} &  &  &  &  &  &  \\
      &  &  & \textbf{+} &  &  &  &  &  \\
      &  &  &  & \textbf{+} &  &  &  &  \\
      &  &  &  &  & \textbf{+} &  &  &  \\
      &  &  &  &  &  & \textbf{+} &  &  \\
      &  &  &  &  &  &  & \textbf{+} &  \\
      &  &  &  &  &  &  &  & \textbf{+}}
\hspace{-0.2cm}\nearrow\hspace{-0.2cm}\SmallMatrix{
      & \textbf{+} &  &  &  &  &  &  &  \\
      &  & \textbf{+} &  &  &  &  &  &  \\
     \textbf{+} &  &  &  &  &  &  &  &  \\
      &  &  &  & \textbf{+} &  &  &  &  \\
      &  &  &  &  & \textbf{+} &  &  &  \\
      &  &  & \textbf{+} &  &  &  &  &  \\
      &  &  &  &  &  &  & \textbf{+} &  \\
      &  &  &  &  &  &  &  & \textbf{+} \\
      &  &  &  &  &  & \textbf{+} &  & }
\hspace{-0.2cm}\nearrow\hspace{-0.2cm}\SmallMatrix{
      &  & \textbf{+} &  &  &  &  &  &  \\
     \textbf{+} &  &  &  &  &  &  &  &  \\
      & \textbf{+} &  &  &  &  &  &  &  \\
      &  &  &  &  & \textbf{+} &  &  &  \\
      &  &  & \textbf{+} &  &  &  &  &  \\
      &  &  &  & \textbf{+} &  &  &  &  \\
      &  &  &  &  &  &  &  & \textbf{+} \\
      &  &  &  &  &  & \textbf{+} &  &  \\
      &  &  &  &  &  &  & \textbf{+} & }
\hspace{-0.2cm}\nearrow\hspace{-0.2cm}\SmallMatrix{
      &  &  & \textbf{+} &  &  &  &  &  \\
      &  &  &  & \textbf{+} &  &  &  &  \\
      &  &  &  &  & \textbf{+} &  &  &  \\
     \textbf{+} &  &  & \textbf{-} &  &  & \textbf{+} &  &  \\
      & \textbf{+} &  &  & \textbf{-} &  &  & \textbf{+} &  \\
      &  & \textbf{+} &  &  & \textbf{-} &  &  & \textbf{+} \\
      &  &  & \textbf{+} &  &  &  &  &  \\
      &  &  &  & \textbf{+} &  &  &  &  \\
      &  &  &  &  & \textbf{+} &  &  & }
\hspace{-0.2cm}\nearrow\hspace{-0.2cm}\SmallMatrix{
      &  &  &  & \textbf{+} &  &  &  &  \\
      &  &  &  &  & \textbf{+} &  &  &  \\
      &  &  & \textbf{+} &  &  &  &  &  \\
      & \textbf{+} &  &  & \textbf{-} &  &  & \textbf{+} &  \\
      &  & \textbf{+} &  &  & \textbf{-} &  &  & \textbf{+} \\
     \textbf{+} &  &  & \textbf{-} &  &  & \textbf{+} &  &  \\
      &  &  &  & \textbf{+} &  &  &  &  \\
      &  &  &  &  & \textbf{+} &  &  &  \\
      &  &  & \textbf{+} &  &  &  &  & }
\hspace{-0.2cm}\nearrow\hspace{-0.2cm}\SmallMatrix{
      &  &  &  &  & \textbf{+} &  &  &  \\
      &  &  & \textbf{+} &  &  &  &  &  \\
      &  &  &  & \textbf{+} &  &  &  &  \\
      &  & \textbf{+} &  &  & \textbf{-} &  &  & \textbf{+} \\
     \textbf{+} &  &  & \textbf{-} &  &  & \textbf{+} &  &  \\
      & \textbf{+} &  &  & \textbf{-} &  &  & \textbf{+} &  \\
      &  &  &  &  & \textbf{+} &  &  &  \\
      &  &  & \textbf{+} &  &  &  &  &  \\
      &  &  &  & \textbf{+} &  &  &  & }
\hspace{-0.2cm}\nearrow\hspace{-0.2cm}\SmallMatrix{
      &  &  &  &  &  & \textbf{+} &  &  \\
      &  &  &  &  &  &  & \textbf{+} &  \\
      &  &  &  &  &  &  &  & \textbf{+} \\
      &  &  &  &  & \textbf{+} &  &  &  \\
      &  &  & \textbf{+} &  &  &  &  &  \\
      &  &  &  & \textbf{+} &  &  &  &  \\
     \textbf{+} &  &  &  &  &  &  &  &  \\
      & \textbf{+} &  &  &  &  &  &  &  \\
      &  & \textbf{+} &  &  &  &  &  & }
\hspace{-0.2cm}\nearrow\hspace{-0.2cm}\SmallMatrix{
      &  &  &  &  &  &  & \textbf{+} &  \\
      &  &  &  &  &  &  &  & \textbf{+} \\
      &  &  &  &  &  & \textbf{+} &  &  \\
      &  &  &  & \textbf{+} &  &  &  &  \\
      &  &  &  &  & \textbf{+} &  &  &  \\
      &  &  & \textbf{+} &  &  &  &  &  \\
      & \textbf{+} &  &  &  &  &  &  &  \\
      &  & \textbf{+} &  &  &  &  &  &  \\
     \textbf{+} &  &  &  &  &  &  &  & }
\hspace{-0.2cm}\nearrow\hspace{-0.2cm}\SmallMatrix{
      &  &  &  &  &  &  &  & \textbf{+} \\
      &  &  &  &  &  & \textbf{+} &  &  \\
      &  &  &  &  &  &  & \textbf{+} &  \\
      &  &  & \textbf{+} &  &  &  &  &  \\
      &  &  &  & \textbf{+} &  &  &  &  \\
      &  &  &  &  & \textbf{+} &  &  &  \\
      &  & \textbf{+} &  &  &  &  &  &  \\
     \textbf{+} &  &  &  &  &  &  &  &  \\
      & \textbf{+} &  &  &  &  &  &  & }\]


\[\hspace{-0.4cm}D = 
\SmallMatrix{
     \color{white}\textbf{+} &  &  &  &  &  &  &  &  \\
      & \color{white}\textbf{+} &  &  &  &  &  &  &  \\
      &  &\color{white}\textbf{+}  &  &  &  &  &  &  \\
     \color{black!70}\cdot & \color{black!70}\cdot & \color{black!70}\cdot & \textbf{-} & \textbf{+} & \color{black!70}\cdot & \color{black!70}\cdot & \color{black!70}\cdot & \color{black!70}\cdot \\
     \color{black!70}\cdot & \color{black!70}\cdot &\color{black!70}\cdot  & \textbf{+} & \textbf{-} & \color{black!70}\cdot & \color{black!70}\cdot & \color{black!70}\cdot & \color{black!70}\cdot \\
      &  &  &  &  & \color{white}\textbf{+} &  &  &  \\
      &  &  &  &  &  & \color{white}\textbf{+} &  &  \\
      &  &  &  &  &  &  &\color{white}\textbf{+}  &  \\
      &  &  &  &  &  &  &  &\color{white}\textbf{+} }
\hspace{-0.2cm}\nearrow\hspace{-0.2cm}\SmallMatrix{
     \color{white}\textbf{+} &  &  &  &  &  &  &  &  \\
      &\color{white}\textbf{+} &  &  &  &  &  &  &  \\
      &  & \color{white}\textbf{+} &  &  &  &  &  &  \\
      \color{black!70}\cdot&  \color{black!70}\cdot& \color{black!70}\cdot & \color{black!70}\cdot & \textbf{-} & \textbf{+} & \color{black!70}\cdot &\color{black!70}\cdot  &  \color{black!70}\cdot\\
     \color{black!70}\cdot & \color{black!70}\cdot & \color{black!70}\cdot &\color{black!70}\cdot  & \textbf{+} & \textbf{-} &\color{black!70}\cdot &\color{black!70}\cdot  & \color{black!70}\cdot \\
      &  &  &  &  & \color{white}\textbf{+} &  &  &  \\
      &  &  &  &  &  & \color{white}\textbf{+} &  &  \\
      &  &  &  &  &  &  & \color{white}\textbf{+} &  \\
      &  &  &  &  &  &  &  & \color{white}\textbf{+}}
\hspace{-0.2cm}\nearrow\hspace{-0.2cm}\SmallMatrix{
     \color{white}\textbf{+} &  &  &  &  &  &  &  &  \\
      & \color{white}\textbf{+} &  &  &  &  &  &  &  \\
      &  & \color{white}\textbf{+} &  &  &  &  &  &  \\
     \color{black!70}\cdot & \color{black!70}\cdot &\color{black!70}\cdot  & \textbf{+} & \color{black!70}\cdot & \textbf{-} & \color{black!70}\cdot & \color{black!70}\cdot & \color{black!70}\cdot \\
    \color{black!70}\cdot  & \color{black!70}\cdot & \color{black!70}\cdot & \textbf{-} & \color{black!70}\cdot & \textbf{+} &\color{black!70}\cdot  & \color{black!70}\cdot &  \color{black!70}\cdot\\
      &  &  &  &  & \color{white}\textbf{+} &  &  &  \\
      &  &  &  &  &  &\color{white}\textbf{+}  &  &  \\
      &  &  &  &  &  &  &\color{white}\textbf{+}  &  \\
      &  &  &  &  &  &  &  &\color{white}\textbf{+} }
\hspace{-0.2cm}\nearrow\hspace{-0.2cm}\SmallMatrix{
     \color{white}\textbf{+} &  &  &  &  &  &  &  &  \\
      & \color{white}\textbf{+} &  &  &  &  &  &  &  \\
      &  &\color{white}\textbf{+}  &  &  &  &  &  &  \\
     \color{black!70}\cdot & \color{black!70}\cdot & \color{black!70}\cdot & \textbf{+} & \textbf{-} & \color{black!70}\cdot & \color{black!70}\cdot & \color{black!70}\cdot & \color{black!70}\cdot \\
     \color{black!70}\cdot & \color{black!70}\cdot &\color{black!70}\cdot  & \textbf{-} & \textbf{+} & \color{black!70}\cdot & \color{black!70}\cdot & \color{black!70}\cdot & \color{black!70}\cdot \\
      &  &  &  &  & \color{white}\textbf{+} &  &  &  \\
      &  &  &  &  &  & \color{white}\textbf{+} &  &  \\
      &  &  &  &  &  &  &\color{white}\textbf{+}  &  \\
      &  &  &  &  &  &  &  &\color{white}\textbf{+} }
\hspace{-0.2cm}\nearrow\hspace{-0.2cm}\SmallMatrix{
     \color{white}\textbf{+} &  &  &  &  &  &  &  &  \\
      &\color{white}\textbf{+} &  &  &  &  &  &  &  \\
      &  & \color{white}\textbf{+} &  &  &  &  &  &  \\
      \color{black!70}\cdot&  \color{black!70}\cdot& \color{black!70}\cdot & \color{black!70}\cdot & \textbf{+} & \textbf{-} & \color{black!70}\cdot &\color{black!70}\cdot  &  \color{black!70}\cdot\\
     \color{black!70}\cdot & \color{black!70}\cdot & \color{black!70}\cdot &\color{black!70}\cdot  & \textbf{-} & \textbf{+} &\color{black!70}\cdot &\color{black!70}\cdot  & \color{black!70}\cdot \\
      &  &  &  &  & \color{white}\textbf{+} &  &  &  \\
      &  &  &  &  &  & \color{white}\textbf{+} &  &  \\
      &  &  &  &  &  &  & \color{white}\textbf{+} &  \\
      &  &  &  &  &  &  &  & \color{white}\textbf{+}}
\hspace{-0.2cm}\nearrow\hspace{-0.2cm}\SmallMatrix{
     \color{white}\textbf{+} &  &  &  &  &  &  &  &  \\
      & \color{white}\textbf{+} &  &  &  &  &  &  &  \\
      &  & \color{white}\textbf{+} &  &  &  &  &  &  \\
     \color{black!70}\cdot & \color{black!70}\cdot &\color{black!70}\cdot  & \textbf{-} & \color{black!70}\cdot & \textbf{+} & \color{black!70}\cdot & \color{black!70}\cdot & \color{black!70}\cdot \\
    \color{black!70}\cdot  & \color{black!70}\cdot & \color{black!70}\cdot & \textbf{+} & \color{black!70}\cdot & \textbf{-} &\color{black!70}\cdot  & \color{black!70}\cdot &  \color{black!70}\cdot\\
      &  &  &  &  & \color{white}\textbf{+} &  &  &  \\
      &  &  &  &  &  &\color{white}\textbf{+}  &  &  \\
      &  &  &  &  &  &  &\color{white}\textbf{+}  &  \\
      &  &  &  &  &  &  &  &\color{white}\textbf{+} }
\hspace{-0.2cm}\nearrow\hspace{-0.2cm}\SmallMatrix{
     \color{white}\textbf{+} &  &  &  &  &  &  &  &  \\
      & \color{white}\textbf{+} &  &  &  &  &  &  &  \\
      &  & \color{white}\textbf{+} &  &  &  &  &  &  \\
      &  &  & \color{white}\textbf{+} &  &  &  &  &  \\
      &  &  &  & \color{white}\textbf{+} &  &  &  &  \\
      &  &  &  &  & \color{white}\textbf{+} &  &  &  \\
      &  &  &  &  &  &\color{white}\textbf{+}  &  &  \\
      &  &  &  &  &  &  &  \color{white}\textbf{+}&  \\
      &  &  &  &  &  &  &  &\color{white}\textbf{+} }
\hspace{-0.2cm}\nearrow\hspace{-0.2cm}\SmallMatrix{
     \color{white}\textbf{+} &  &  &  &  &  &  &  &  \\
      & \color{white}\textbf{+} &  &  &  &  &  &  &  \\
      &  & \color{white}\textbf{+} &  &  &  &  &  &  \\
      &  &  & \color{white}\textbf{+} &  &  &  &  &  \\
      &  &  &  & \color{white}\textbf{+} &  &  &  &  \\
      &  &  &  &  & \color{white}\textbf{+} &  &  &  \\
      &  &  &  &  &  &\color{white}\textbf{+}  &  &  \\
      &  &  &  &  &  &  &  \color{white}\textbf{+}&  \\
      &  &  &  &  &  &  &  &\color{white}\textbf{+} }
\hspace{-0.2cm}\nearrow\hspace{-0.2cm}\SmallMatrix{
     \color{white}\textbf{+} &  &  &  &  &  &  &  &  \\
      & \color{white}\textbf{+} &  &  &  &  &  &  &  \\
      &  & \color{white}\textbf{+} &  &  &  &  &  &  \\
      &  &  & \color{white}\textbf{+} &  &  &  &  &  \\
      &  &  &  & \color{white}\textbf{+} &  &  &  &  \\
      &  &  &  &  & \color{white}\textbf{+} &  &  &  \\
      &  &  &  &  &  &\color{white}\textbf{+}  &  &  \\
      &  &  &  &  &  &  &  \color{white}\textbf{+}&  \\
      &  &  &  &  &  &  &  &\color{white}\textbf{+} }\]

\[D = -T_{4,4,1:\,5,5,4} - T_{4,5,2:\,5,6,5} + T_{4,4,3:\,5,6,6}\]

\[L(A) = L(B) = 
\Matrix{
    1 & 2 & 3 & 4 & 5 & 6 & 7 & 8 & 9 \\
    3 & 1 & 2 & 6 & 4 & 5 & 9 & 7 & 8 \\
    2 & 3 & 1 & 5 & 6 & 4 & 8 & 9 & 7 \\
    4 & 5 & 6 & 6 & 5 & 4 & 4 & 5 & 6 \\
    6 & 4 & 5 & 4 & 6 & 5 & 6 & 4 & 5 \\
    5 & 6 & 4 & 5 & 4 & 6 & 5 & 6 & 4 \\
    7 & 8 & 9 & 4 & 5 & 6 & 1 & 2 & 3 \\
    9 & 7 & 8 & 6 & 4 & 5 & 3 & 1 & 2 \\
    8 & 9 & 7 & 5 & 6 & 4 & 2 & 3 & 1}\]

$D$ can also be expressed as the sum of pairs of T-blocks occupying the same vertical lines with opposite depth.

\[D = (T_{4,4,3:\,5,6,6} - T_{4,4,1:\,5,6,4}) + (T_{4,5,1:\,5,6,4} - T_{4,5,2:\,5,6,5})\]

\end{example}

These examples demonstrate an alternative characterisation of two ASHMs with the same corresponding ASHL.

\begin{theorem}\label{T-block_pairs} Two ASHMs $A$ and $B$ satisfy $L(A) = L(B)$ if and only if $A-B$ can be expressed as a sum of pairs of T-blocks with opposite depth occupying the same vertical lines. \end{theorem}

\begin{proof}
First assume that two ASHMs $A$ and $B$ satisfy $L(A) = L(B)$. From Theorem \ref{ASHL_difference}, we know that $D = A-B$ can be decomposed into T-blocks such that in any vertical line $V$ of $D$,
\[\sum_{T \in T_V} d(T) = 0\text{.}\]

Run the following iterative step.

\begin{itemize}
\item Let $k_1$ be the least integer for which the plane $P_{k_1}(D)$ contains non-zero entries. For some positive entry $d_{i_1j_1k_1}$ of $D$, we can find entries $d_{i_2j_1k_1}$, $d_{i_1j_2k_1}$ and $d_{i_1j_1k_2}$ with negative sign, where $k_2 > k_1$. Choose $k_2$ to be the least integer satisfying this condition.

As $\sum_{T \in T_{V_{i_1j_1}}} d(T) = 0$, there must also be another positive entry $d_{i_1j_1k_3}$ in $V_{i_1j_1}$. Choose $k_3$ to be the largest integer satisfying this condition. Let $k_4 = k_3 - (k_2-k_1)$, and let $D' = D - T_{i_1,j_1,k_1:\,i_2,j_2,k_2} + T_{i_1,j_1,k_4:\,i_2,j_2,k_3}$. Now repeat this step for $D'$.
\end{itemize}

Note that the sum of the absolute value of the entries of $P_{k_1}(D')$ is at least $2$ less than that of $P_{k_1}(D)$, and that all line sums of $D'$ are $0$. Note also that $k_4 > k_1$, because the sum of the absolute value of the negative entries in $V$ must equal the sum of the positive entries in $V$ and the weighted sums of each must also equal. If $k_4 \leq k_1$, this implies that all the negative entries of $V$ are positioned above all positive entries in $V$, which means that their weighted sum is negative. Therefore $k_4 > k_1$, which means that $k_1$ remains the lowest integer for which $P_{k_1}(D)$ contains non-zero entries. We can therefore run this iterative process repeatedly, resulting in $P_{k_1}(D')$ becoming a 0-matrix and on the next iteration, $k_1$ will increase to the new lowest integer for which plane $P_{k_1}(D')$ contains non-zero entries until $D'$ is a 0-hypermatrix. Therefore $A-B$ can be expressed as a sum of pairs of T-blocks with opposite depth occupying the same vertical lines.

\*

Now, assume that $A-B$ can be expressed as a sum of pairs of T-blocks with opposite depth occupying the same vertical lines. This means that, in any vertical line $V$ of $A-B$,
\[\sum_{T \in T_V} d(T) = \sum_T d(T)-d(T) = 0\text{.}\]
Which, by Theorem \ref{ASHL_difference}, means that $L(A) = L(B)$.
\end{proof}

So, for any ASHM $A$ with ASHL $L = L(A)$, we have that $A_n(L)$ is the set of all ASHMs $B$ for which $A-B$ can be expressed as a sum of pairs of T-blocks with opposite depth occupying the same vertical lines. The following theorem tells us the smallest dimension an ASHL $L$ can have if $A_n(L)$ contains more than one element.

\begin{theorem} \label{size_thm} The minimum $n$ for which two distinct $n \times n \times n$ ASHMs $A$ and $B$ can satisfy $L(A) = L(B)$ is 4. \end{theorem}
\begin{proof}
Two ASHMs $A$ and $B$ can satisfy $L(A) = L(B)$ only if both ASHMs contain at least one negative entry \cite{ashmbib}.

\*

The following are the only two $3 \times 3$ ASHMs containing negative entries, and these do not satisfy $L(A) = L(B)$.

\[A = \Matrix{
       &  & + \\
       & + &  \\
      + &  & }
\hspace{-0.1cm}\nearrow\hspace{-0.1cm}\Matrix{
       & + &  \\
      + & - & + \\
       & + & }
\hspace{-0.1cm}\nearrow\hspace{-0.1cm}\Matrix{
      + &  &  \\
       & + &  \\
       &  & +}
\]
\[
B = \Matrix{
      + &  &  \\
       & + &  \\
       &  & +}
\hspace{-0.1cm}\nearrow\hspace{-0.1cm}\Matrix{
       & + &  \\
      + & - & + \\
       & + & }
\hspace{-0.1cm}\nearrow\hspace{-0.1cm}\Matrix{
       &  & + \\
       & + &  \\
      + &  & }\]

\[L(A) = \Matrix{
    3 & 2 & 1 \\
    2 & 2 & 2 \\
    1 & 2 & 3}
\hspace{1.5cm}
L(B) =\Matrix{
    1 & 2 & 3 \\
    2 & 2 & 2 \\
    3 & 2 & 1}\]

The following example is a pair of $4 \times 4 \times 4$ ASHMs with the same corresponding ASHL.

\[A = 
\Matrix{
       &  &  & + \\
       &  & + &  \\
       & + &  &  \\
      + &  &  & }
\nearrow\Matrix{
       &  & + &  \\
       & + & - & + \\
      + & - & + &  \\
       & + &  & }
\nearrow\Matrix{
       & + &  &  \\
      + &  &  &  \\
       &  &  & + \\
       &  & + & }
\nearrow\Matrix{
      + &  &  &  \\
       &  & + &  \\
       & + &  &  \\
       &  &  & +}\]

\[B = 
\Matrix{
       &  &  & + \\
       & + &  &  \\
       &  & + &  \\
      + &  &  & }
\nearrow\Matrix{
       &  & + &  \\
       &  &  & + \\
      + &  &  &  \\
       & + &  & }
\nearrow\Matrix{
       & + &  &  \\
      + & - & + &  \\
       & + & - & + \\
       &  & + & }
\nearrow\Matrix{
      + &  &  &  \\
       & + &  &  \\
       &  & + &  \\
       &  &  & +}\]

\[L(A) = L(B) = 
\Matrix{
    4 & 3 & 2 & 1 \\
    3 & 2 & 3 & 2 \\
    2 & 3 & 2 & 3 \\
    1 & 2 & 3 & 4}\]

Therefore the minimum dimension for which two ASHMs $A$ and $B$ can satisfy $L(A) = L(B)$ is 4.

\end{proof}

Note that $A_n(L)$ can contain ASHMs with different numbers of non-zero elements. In the following example, the number of non-zero entries in $A$ is 68, while the number of non-zero entries in $B$ is 76.

\begin{example}\label{2ashm}

\[A = 
\SmallMatrix{
     & & & & & \textbf{+} & & \\
     & & & & \textbf{+} & & & \\
     & \textbf{+} & & & & & & \\
     \textbf{+} & & & & & & & \\
     & & & & & & & \textbf{+} \\
     & & & & & & \textbf{+} & \\
     & & & \textbf{+} & & & & \\
     & & \textbf{+} & & & & & }
\hspace{-0.2cm}\nearrow
\hspace{-0.16cm}\SmallMatrix{
     & & & \textbf{+} & & & & \\
     & & \textbf{+} & & & & & \\
     & & & & & & \textbf{+} & \\
     & & & & \textbf{+} & & & \\
     & \textbf{+} & & & & & & \\
     \textbf{+} & & & & & & & \\
     & & & & & & & \textbf{+} \\
     & & & & & \textbf{+} & & }
\hspace{-0.2cm}\nearrow
\hspace{-0.16cm}\SmallMatrix{
     & & & & & & \textbf{+} & \\
     & & & & & & & \textbf{+} \\
     & & & & \textbf{+} & & & \\
     & & & & & \textbf{+} & & \\
     & & & \textbf{+} & & & & \\
     & & \textbf{+} & & & & & \\
     & \textbf{+} & & & & & & \\
     \textbf{+} & & & & & & & }
\hspace{-0.2cm}\nearrow
\hspace{-0.16cm}\SmallMatrix{
     & \textbf{+} & & & & & & \\
     & & & \textbf{+} & & & & \\
     \textbf{+} & & & & & & & \\
     & & \textbf{+} & & & & & \\
     & & & & \textbf{+} & & & \\
     & & & & & \textbf{+} & & \\
     & & & & & & \textbf{+} & \\
     & & & & & & & \textbf{+}}
\hspace{-0.2cm}\nearrow
\hspace{-0.16cm}\SmallMatrix{
     & & & & \textbf{+} & & & \\
     & & & & & \textbf{+} & & \\
     & & & \textbf{+} & & & & \\
     & \textbf{+} & & & \textbf{-} & & & \textbf{+} \\
     \textbf{+} & & & \textbf{-} & & & \textbf{+} & \\
     & & & & \textbf{+} & & & \\
     & & \textbf{+} & & & & & \\
     & & & \textbf{+} & & & & }
\hspace{-0.2cm}\nearrow
\hspace{-0.16cm}\SmallMatrix{
     \textbf{+} & & & & & & & \\
     & \textbf{+} & & & & & & \\
     & & \textbf{+} & & & & & \\
     & & & \textbf{+} & & & & \\
     & & & & & \textbf{+} & & \\
     & & & & & & & \textbf{+} \\
     & & & & \textbf{+} & & & \\
     & & & & & & \textbf{+} & }
\hspace{-0.2cm}\nearrow
\hspace{-0.16cm}\SmallMatrix{
     & & & & & & & \textbf{+} \\
     & & & & & & \textbf{+} & \\
     & & & & & \textbf{+} & & \\
     & & & & \textbf{+} & & & \\
     & & \textbf{+} & & & & & \\
     & & & \textbf{+} & & & & \\
     \textbf{+} & & & & & & & \\
     & \textbf{+} & & & & & & }
\hspace{-0.2cm}\nearrow
\hspace{-0.16cm}\SmallMatrix{
     & & \textbf{+} & & & & & \\
     \textbf{+} & & & & & & & \\
     & & & & & & & \textbf{+} \\
     & & & & & & \textbf{+} & \\
     & & & \textbf{+} & & & & \\
     & \textbf{+} & & & & & & \\
     & & & & & \textbf{+} & & \\
     & & & & \textbf{+} & & & }\]


\[ B = 
\SmallMatrix{
     & & & & & \textbf{+} & & \\
     & & & & \textbf{+} & & & \\
     & \textbf{+} & & & & & & \\
     \textbf{+} & & & & & & & \\
     & & & & & & & \textbf{+} \\
     & & & & & & \textbf{+} & \\
     & & & \textbf{+} & & & & \\
     & & \textbf{+} & & & & & }
\hspace{-0.2cm}\nearrow
\hspace{-0.16cm}\SmallMatrix{
     & & & \textbf{+} & & & & \\
     & & \textbf{+} & & & & & \\
     & & & & & & \textbf{+} & \\
     & & & & \textbf{+} & & & \\
     & \textbf{+} & & & & & & \\
     \textbf{+} & & & & & & & \\
     & & & & & & & \textbf{+} \\
     & & & & & \textbf{+} & & }
\hspace{-0.2cm}\nearrow
\hspace{-0.16cm}\SmallMatrix{
     & & & & & & \textbf{+} & \\
     & & & & & & & \textbf{+} \\
     & & & & \textbf{+} & & & \\
     & & & \textbf{+} & \textbf{-} & \textbf{+} & & \\
     & & & & \textbf{+} & & & \\
     & & \textbf{+} & & & & & \\
     & \textbf{+} & & & & & & \\
     \textbf{+} & & & & & & & }
\hspace{-0.2cm}\nearrow
\hspace{-0.16cm}\SmallMatrix{
     & \textbf{+} & & & & & & \\
     & & & \textbf{+} & & & & \\
     \textbf{+} & & & & & & & \\
     & & \textbf{+} & \textbf{-} & \textbf{+} & & & \\
     & & & \textbf{+} & & & & \\
     & & & & & \textbf{+} & & \\
     & & & & & & \textbf{+} & \\
     & & & & & & & \textbf{+}}
\hspace{-0.2cm}\nearrow
\hspace{-0.16cm}\SmallMatrix{
     & & & & \textbf{+} & & & \\
     & & & & & \textbf{+} & & \\
     & & & \textbf{+} & & & & \\
     & \textbf{+} & & & \textbf{-} & & & \textbf{+} \\
     \textbf{+} & & & \textbf{-} & & & \textbf{+} & \\
     & & & & \textbf{+} & & & \\
     & & \textbf{+} & & & & & \\
     & & & \textbf{+} & & & & }
\hspace{-0.2cm}\nearrow
\hspace{-0.16cm}\SmallMatrix{
     \textbf{+} & & & & & & & \\
     & \textbf{+} & & & & & & \\
     & & \textbf{+} & & & & & \\
     & & & & \textbf{+} & & & \\
     & & & \textbf{+} & \textbf{-} & \textbf{+} & & \\
     & & & & & & & \textbf{+} \\
     & & & & \textbf{+} & & & \\
     & & & & & & \textbf{+} & }
\hspace{-0.2cm}\nearrow
\hspace{-0.16cm}\SmallMatrix{
     & & & & & & & \textbf{+} \\
     & & & & & & \textbf{+} & \\
     & & & & & \textbf{+} & & \\
     & & & \textbf{+} & & & & \\
     & & \textbf{+} & \textbf{-} & \textbf{+} & & & \\
     & & & \textbf{+} & & & & \\
     \textbf{+} & & & & & & & \\
     & \textbf{+} & & & & & & }
\hspace{-0.2cm}\nearrow
\hspace{-0.16cm}\SmallMatrix{
     & & \textbf{+} & & & & & \\
     \textbf{+} & & & & & & & \\
     & & & & & & & \textbf{+} \\
     & & & & & & \textbf{+} & \\
     & & & \textbf{+} & & & & \\
     & \textbf{+} & & & & & & \\
     & & & & & \textbf{+} & & \\
     & & & & \textbf{+} & & & }\]


\[L(A) = L(B) = 
\Matrix{
     6 & 4 & 8 & 2 & 5 & 1 & 3 & 7 \\
     8 & 6 & 2 & 4 & 1 & 5 & 7 & 3 \\
     4 & 1 & 6 & 5 & 3 & 7 & 2 & 8 \\
     1 & 5 & 4 & 6 & 4 & 3 & 8 & 5 \\
     5 & 2 & 7 & 6 & 4 & 6 & 5 & 1 \\
     2 & 8 & 3 & 7 & 5 & 4 & 1 & 6 \\
     7 & 3 & 5 & 1 & 6 & 8 & 4 & 2 \\
     3 & 7 & 1 & 5 & 8 & 2 & 6 & 4}\]
\end{example}

\section{The Maximum Number of Equal Entries of an ASHL}

The following question is also posed in Brualdi and Dahl's paper \cite{ashmbib}.

\begin{problem} What is the maximum number of times an integer can occur as an entry of an $n \times n$ ASHL? \end{problem}

It is shown in their paper that an integer can occur $2n$ times in an $n \times n$ ASHL, and it is asked if the maximum is equal to $2n$. The following example exceeds this bound.

\begin{example}\label{max_entries7} Here, 4 occurs as an entry in this $7 \times 7$ ASHL 29 times.
\[ A = 
\SmallMatrix{
     \color{white} 0 & \color{white} 0 & \textbf{+} & \color{white} 0 & \color{white} 0 & \color{white} 0 & \color{white} 0 \\
     \color{white} 0 & \textbf{+} & \color{white} 0 & \color{white} 0 & \color{white} 0 & \color{white} 0 & \color{white} 0 \\
     \textbf{+} & \color{white} 0 & \color{white} 0 & \color{white} 0 & \color{white} 0 & \color{white} 0 & \color{white} 0 \\
     \color{white} 0 & \color{white} 0 & \color{white} 0 & \textbf{+} & \color{white} 0 & \color{white} 0 & \color{white} 0 \\
     \color{white} 0 & \color{white} 0 & \color{white} 0 & \color{white} 0 & \color{white} 0 & \color{white} 0 & \textbf{+} \\
     \color{white} 0 & \color{white} 0 & \color{white} 0 & \color{white} 0 & \color{white} 0 & \textbf{+} & \color{white} 0 \\
     \color{white} 0 & \color{white} 0 & \color{white} 0 & \color{white} 0 & \textbf{+} & \color{white} 0 & \color{white} 0}
\hspace{-0.2cm}\nearrow\hspace{-0.2cm}\SmallMatrix{
     \color{white} 0 & \color{white} 0 & \color{white} 0 & \color{white} 0 & \color{white} 0 & \color{white} 0 & \textbf{+} \\
     \color{white} 0 & \color{white} 0 & \color{white} 0 & \color{white} 0 & \color{white} 0 & \textbf{+} & \color{white} 0 \\
     \color{white} 0 & \color{white} 0 & \color{white} 0 & \textbf{+} & \color{white} 0 & \color{white} 0 & \color{white} 0 \\
     \color{white} 0 & \color{white} 0 & \textbf{+} & \textbf{-} & \textbf{+} & \color{white} 0 & \color{white} 0 \\
     \color{white} 0 & \color{white} 0 & \color{white} 0 & \textbf{+} & \color{white} 0 & \color{white} 0 & \color{white} 0 \\
     \color{white} 0 & \textbf{+} & \color{white} 0 & \color{white} 0 & \color{white} 0 & \color{white} 0 & \color{white} 0 \\
     \textbf{+} & \color{white} 0 & \color{white} 0 & \color{white} 0 & \color{white} 0 & \color{white} 0 & \color{white} 0}
\hspace{-0.2cm}\nearrow\hspace{-0.2cm}\SmallMatrix{
     \color{white} 0 & \textbf{+} & \color{white} 0 & \color{white} 0 & \color{white} 0 & \color{white} 0 & \color{white} 0 \\
     \textbf{+} & \textbf{-} & \color{white} 0 & \textbf{+} & \color{white} 0 & \color{white} 0 & \color{white} 0 \\
     \color{white} 0 & \color{white} 0 & \textbf{+} & \textbf{-} & \textbf{+} & \color{white} 0 & \color{white} 0 \\
     \color{white} 0 & \textbf{+} & \textbf{-} & \textbf{+} & \textbf{-} & \textbf{+} & \color{white} 0 \\
     \color{white} 0 & \color{white} 0 & \textbf{+} & \textbf{-} & \textbf{+} & \color{white} 0 & \color{white} 0 \\
     \color{white} 0 & \color{white} 0 & \color{white} 0 & \textbf{+} & \color{white} 0 & \textbf{-} & \textbf{+} \\
     \color{white} 0 & \color{white} 0 & \color{white} 0 & \color{white} 0 & \color{white} 0 & \textbf{+} & \color{white} 0}
\hspace{-0.2cm}\nearrow\hspace{-0.2cm}\SmallMatrix{
     \color{white} 0 & \color{white} 0 & \color{white} 0 & \textbf{+} & \color{white} 0 & \color{white} 0 & \color{white} 0 \\
     \color{white} 0 & \color{white} 0 & \textbf{+} & \textbf{-} & \textbf{+} & \color{white} 0 & \color{white} 0 \\
     \color{white} 0 & \textbf{+} & \textbf{-} & \textbf{+} & \textbf{-} & \textbf{+} & \color{white} 0 \\
     \textbf{+} & \textbf{-} & \textbf{+} & \textbf{-} & \textbf{+} & \textbf{-} & \textbf{+} \\
     \color{white} 0 & \textbf{+} & \textbf{-} & \textbf{+} & \textbf{-} & \textbf{+} & \color{white} 0 \\
     \color{white} 0 & \color{white} 0 & \textbf{+} & \textbf{-} & \textbf{+} & \color{white} 0 & \color{white} 0 \\
     \color{white} 0 & \color{white} 0 & \color{white} 0 & \textbf{+} & \color{white} 0 & \color{white} 0 & \color{white} 0}
\hspace{-0.2cm}\nearrow\hspace{-0.2cm}\SmallMatrix{
     \color{white} 0 & \color{white} 0 & \color{white} 0 & \color{white} 0 & \color{white} 0 & \textbf{+} & \color{white} 0 \\
     \color{white} 0 & \color{white} 0 & \color{white} 0 & \textbf{+} & \color{white} 0 & \textbf{-} & \textbf{+} \\
     \color{white} 0 & \color{white} 0 & \textbf{+} & \textbf{-} & \textbf{+} & \color{white} 0 & \color{white} 0 \\
     \color{white} 0 & \textbf{+} & \textbf{-} & \textbf{+} & \textbf{-} & \textbf{+} & \color{white} 0 \\
     \color{white} 0 & \color{white} 0 & \textbf{+} & \textbf{-} & \textbf{+} & \color{white} 0 & \color{white} 0 \\
     \textbf{+} & \textbf{-} & \color{white} 0 & \textbf{+} & \color{white} 0 & \color{white} 0 & \color{white} 0 \\
     \color{white} 0 & \textbf{+} & \color{white} 0 & \color{white} 0 & \color{white} 0 & \color{white} 0 & \color{white} 0}
\hspace{-0.2cm}\nearrow\hspace{-0.2cm}\SmallMatrix{
     \textbf{+} & \color{white} 0 & \color{white} 0 & \color{white} 0 & \color{white} 0 & \color{white} 0 & \color{white} 0 \\
     \color{white} 0 & \textbf{+} & \color{white} 0 & \color{white} 0 & \color{white} 0 & \color{white} 0 & \color{white} 0 \\
     \color{white} 0 & \color{white} 0 & \color{white} 0 & \textbf{+} & \color{white} 0 & \color{white} 0 & \color{white} 0 \\
     \color{white} 0 & \color{white} 0 & \textbf{+} & \textbf{-} & \textbf{+} & \color{white} 0 & \color{white} 0 \\
     \color{white} 0 & \color{white} 0 & \color{white} 0 & \textbf{+} & \color{white} 0 & \color{white} 0 & \color{white} 0 \\
     \color{white} 0 & \color{white} 0 & \color{white} 0 & \color{white} 0 & \color{white} 0 & \textbf{+} & \color{white} 0 \\
     \color{white} 0 & \color{white} 0 & \color{white} 0 & \color{white} 0 & \color{white} 0 & \color{white} 0 & \textbf{+}}
\hspace{-0.2cm}\nearrow\hspace{-0.2cm}\SmallMatrix{
     \color{white} 0 & \color{white} 0 & \color{white} 0 & \color{white} 0 & \textbf{+} & \color{white} 0 & \color{white} 0 \\
     \color{white} 0 & \color{white} 0 & \color{white} 0 & \color{white} 0 & \color{white} 0 & \textbf{+} & \color{white} 0 \\
     \color{white} 0 & \color{white} 0 & \color{white} 0 & \color{white} 0 & \color{white} 0 & \color{white} 0 & \textbf{+} \\
     \color{white} 0 & \color{white} 0 & \color{white} 0 & \textbf{+} & \color{white} 0 & \color{white} 0 & \color{white} 0 \\
     \textbf{+} & \color{white} 0 & \color{white} 0 & \color{white} 0 & \color{white} 0 & \color{white} 0 & \color{white} 0 \\
     \color{white} 0 & \textbf{+} & \color{white} 0 & \color{white} 0 & \color{white} 0 & \color{white} 0 & \color{white} 0 \\
     \color{white} 0 & \color{white} 0 & \textbf{+} & \color{white} 0 & \color{white} 0 & \color{white} 0 & \color{white} 0}\]

\[L(A) =
\Matrix{
    6 & 3 & 1 & 4 & 7 & 5 & 2 \\
    3 & 4 & 4 & 4 & 4 & 4 & 5 \\
    1 & 4 & 4 & 4 & 4 & 4 & 7 \\
    4 & 4 & 4 & 4 & 4 & 4 & 4 \\
    7 & 4 & 4 & 4 & 4 & 4 & 1 \\
    5 & 4 & 4 & 4 & 4 & 4 & 3 \\
    2 & 5 & 7 & 4 & 1 & 3 & 6}\]

This is the highest possible number of times an entry can be repeated in a $7 \times 7$ ASHL, as each number $1, 2, \dots, n$ must appear exactly once in the first and last rows and columns of an $n \times n$ ASHL. This means that the upper bound for such a construction is $(n-2)^2 + 4$, which is $(7-2)^2+4=29$, in the $n=7$ case.

\end{example}

\begin{definition} We define the \emph{diamond positions} of an $n \times n$ ASM $A$ to be the positions of $A$ corresponding to non-zero entries of the diamond ASM $D_n$. \end{definition}

Example \ref{max_entries7} can be generalised in the following way.

\begin{theorem}\label{max_entries} For a given $n$, there exists an $n \times n$ ASHL such that 
\begin{itemize}
\item $\frac{n+1}{2}$ occurs as an entry $\frac{n^2+4n-19}{2}$ times, if $n$ is odd;
\item $\frac{n}{2}$ occurs as an entry $\frac{n^2+4n-20}{2}$ times, if $n$ is even.
\end{itemize}
\end{theorem}

\begin{proof}
Let $p = \lfloor \frac{n+1}{2} \rfloor$, $m = \lceil\frac{n+1}{2}\rceil$, and note that $p=m$ for odd $n$. We construct an ASHM $A$ with the required properties as follows.

\begin{itemize}
\item $P_p(A) = D_n$, and for $k = 1, 2, \dots, p-1$, plane $P_{p \pm k}(A)$ contains the diamond ASM $D_{n-2k}$ such that there is a $+$ entry in every position where there is a $-$ entry in the diamond ASM contained in the plane $P_{p \pm (k-1)}$.
\item The other non-zero entries of $P_{p-1}(A)$ are a diagonal of $+$ entries from $A_{1,m+2,p-1}$ to $A_{p-2,n,p-1}$, a diagonal of $-$ entries from $A_{2,m+2,p-1}$ to $A_{p-2,n-1,p-1}$, a diagonal of $+$ entries from $A_{m+2,1,p-1}$ to $A_{n,p-2,p-1}$, and a diagonal of $-$ entries from $A_{m+2,2,p-1}$ to $A_{n-1,p-2,p-1}$.
\item The other non-zero entries of $P_1(A)$ are a diagonal of $+$ entries from $A_{1,m+1,1}$ to $A_{p-1,n,1}$ and a diagonal of $+$ entries from $A_{m+1,1,1}$ to $A_{n,p-1,1}$.
\item The other non-zero entries of $P_2(A)$ are an anti-diagonal of $+$ entries from $A_{2,p-2,2}$ to $A_{p-2,2,2}$, an anti-diagonal of $+$ entries from $A_{m+2,n-1,2}$ to $A_{n-1,m+2,2}$, and $+$ entries in $A_{1,1,2}$ and $A_{n,n,2}$.
\item For $k=2, \dots, p-3$, the other non-zero entries of $P_{p-k}(A)$ are an anti-diagonal of $+$ entries from $A_{1,k,p-k}$ to $A_{k,1,p-k}$, and an anti-diagonal of $+$ entries from $A_{n-k+1,n,p-k}$ to $A_{n,n-k+1,p-k}$.
\item For $k = 1, 2, \dots, p-1$, the entries of $P_{p+k}(A)$ not containing $D_{n-2k}$ (as outlined in the first step) satisfy $A_{i,j,p+k} = A_{n-i, j, p-k}$.
\item If $n$ is even, the non-zero entries of $P_n(A)$ are an anti-diagonal from $A_{p,1,n}$ to $A_{1,p,n}$ and an anti-diagonal from $A_{n,m,n}$ to $A_{m,n,n}$.
\end{itemize}

\*

We see that $p$ occurs in all diamond positions of $L = L(A)$ because
\[(p-k+1) - (p-k+2) + \dots -(p+k-2) + (p+k-1) = p\text{.}\]

The other occurances of $p$ as entries of $L$ occur along diagonals from $L_{2,m+2}$ to $L_{p-2,n-1}$  and from $L_{m+2,2}$ to $L_{n-1,p-2}$ by
\[1 - (p-1) + (n+p-m-1) = n-m+1 = p\text(,)\]
and along antidiagonals from $L_{p-2,2}$ to $L_{2,p-2}$  and from $L_{n-1, m+2}$ to $L_{m+2,n-1}$ by
\[2 - (p+1) + (n+p-m) = n-m+1 = p\text{.}\]

\*

\underline{The odd case}:
\[
\SmallMatrix{
      \color{white} 0 &  \color{white} 0 &  \color{white} 0 &  \color{white} 0 &  \color{white} 0 &  \color{white} 0 & \textbf{+} &  \color{white} 0 &  \color{white} 0 &  \color{white} 0 &  \color{white} 0 \\
     \vspace{-0.1cm} \color{white} 0 &  \color{white} 0 &  \color{white} 0 &  \color{white} 0 &  \color{white} 0 &  \color{white} 0 &  \color{white} 0 & \textbf{+} &  \color{white} 0 &  \color{white} 0 &  \color{white} 0 \\
      \color{white} 0 &  \color{white} 0 &  \color{white} 0 &  \color{white} 0 &  \color{white} 0 &  \color{white} 0 &  \color{white} 0 &  \color{white} 0 & \ddots &  \color{white} 0 &  \color{white} 0 \\
      \color{white} 0 &  \color{white} 0 &  \color{white} 0 &  \color{white} 0 &  \color{white} 0 &  \color{white} 0 &  \color{white} 0 &  \color{white} 0 &  \color{white} 0 & \textbf{+} &  \color{white} 0 \\
      \color{white} 0 &  \color{white} 0 &  \color{white} 0 &  \color{white} 0 &  \color{white} 0 &  \color{white} 0 &  \color{white} 0 &  \color{white} 0 &  \color{white} 0 &  \color{white} 0 & \textbf{+} \\
      \color{white} 0 &  \color{white} 0 &  \color{white} 0 &  \color{white} 0 &  \color{white} 0 & \textbf{+} &  \color{white} 0 &  \color{white} 0 &  \color{white} 0 &  \color{white} 0 &  \color{white} 0 \\
     \textbf{+} &  \color{white} 0 &  \color{white} 0 &  \color{white} 0 &  \color{white} 0 &  \color{white} 0 &  \color{white} 0 &  \color{white} 0 &  \color{white} 0 &  \color{white} 0 &  \color{white} 0 \\
     \vspace{-0.1cm} \color{white} 0 & \textbf{+} &  \color{white} 0 &  \color{white} 0 &  \color{white} 0 &  \color{white} 0 &  \color{white} 0 &  \color{white} 0 &  \color{white} 0 &  \color{white} 0 &  \color{white} 0 \\
      \color{white} 0 &  \color{white} 0 & \ddots &  \color{white} 0 &  \color{white} 0 &  \color{white} 0 &  \color{white} 0 &  \color{white} 0 &  \color{white} 0 &  \color{white} 0 &  \color{white} 0 \\
      \color{white} 0 &  \color{white} 0 &  \color{white} 0 & \textbf{+} &  \color{white} 0 &  \color{white} 0 &  \color{white} 0 &  \color{white} 0 &  \color{white} 0 &  \color{white} 0 &  \color{white} 0 \\
      \color{white} 0 &  \color{white} 0 &  \color{white} 0 &  \color{white} 0 & \textbf{+} &  \color{white} 0 &  \color{white} 0 &  \color{white} 0 &  \color{white} 0 &  \color{white} 0 &  \color{white} 0}
\hspace{-0.12cm}\nearrow\hspace{-0.12cm}\SmallMatrix{
     \textbf{+} &  \color{white} 0 &  \color{white} 0 &  \color{white} 0 &  \color{white} 0 &  \color{white} 0 &  \color{white} 0 &  \color{white} 0 &  \color{white} 0 &  \color{white} 0 &  \color{white} 0 \\
     \vspace{-0.1cm} \color{white} 0 &  \color{white} 0 &  \color{white} 0 & \textbf{+} &  \color{white} 0 &  \color{white} 0 &  \color{white} 0 &  \color{white} 0 &  \color{white} 0 &  \color{white} 0 &  \color{white} 0 \\
      \color{white} 0 &  \color{white} 0 & \iddots &  \color{white} 0 &  \color{white} 0 &  \color{white} 0 &  \color{white} 0 &  \color{white} 0 &  \color{white} 0 &  \color{white} 0 &  \color{white} 0 \\
      \color{white} 0 & \textbf{+} &  \color{white} 0 &  \color{white} 0 &  \color{white} 0 &  \color{white} 0 &  \color{white} 0 &  \color{white} 0 &  \color{white} 0 &  \color{white} 0 &  \color{white} 0 \\
      \color{white} 0 &  \color{white} 0 &  \color{white} 0 &  \color{white} 0 &  \color{white} 0 & \textbf{+} &  \color{white} 0 &  \color{white} 0 &  \color{white} 0 &  \color{white} 0 &  \color{white} 0 \\
      \color{white} 0 &  \color{white} 0 &  \color{white} 0 &  \color{white} 0 & \textbf{+} & \textbf{-} & \textbf{+} &  \color{white} 0 &  \color{white} 0 &  \color{white} 0 &  \color{white} 0 \\
      \color{white} 0 &  \color{white} 0 &  \color{white} 0 &  \color{white} 0 &  \color{white} 0 & \textbf{+} &  \color{white} 0 &  \color{white} 0 &  \color{white} 0 &  \color{white} 0 &  \color{white} 0 \\
     \vspace{-0.1cm} \color{white} 0 &  \color{white} 0 &  \color{white} 0 &  \color{white} 0 &  \color{white} 0 &  \color{white} 0 &  \color{white} 0 &  \color{white} 0 &  \color{white} 0 & \textbf{+} &  \color{white} 0 \\
      \color{white} 0 &  \color{white} 0 &  \color{white} 0 &  \color{white} 0 &  \color{white} 0 &  \color{white} 0 &  \color{white} 0 &  \color{white} 0 & \iddots &  \color{white} 0 &  \color{white} 0 \\
      \color{white} 0 &  \color{white} 0 &  \color{white} 0 &  \color{white} 0 &  \color{white} 0 &  \color{white} 0 &  \color{white} 0 & \textbf{+} &  \color{white} 0 &  \color{white} 0 &  \color{white} 0 \\
      \color{white} 0 &  \color{white} 0 &  \color{white} 0 &  \color{white} 0 &  \color{white} 0 &  \color{white} 0 &  \color{white} 0 &  \color{white} 0 &  \color{white} 0 &  \color{white} 0 & \textbf{+}}
\hspace{-0.12cm}\nearrow\hspace{-0.12cm}\SmallMatrix{
     \vspace{-0.1cm} \color{white} 0 &  \color{white} 0 & \textbf{+} &  \color{white} 0 &  \color{white} 0 &  \color{white} 0 &  \color{white} 0 &  \color{white} 0 &  \color{white} 0 &  \color{white} 0 &  \color{white} 0 \\
      \color{white} 0 & \iddots &  \color{white} 0 &  \color{white} 0 &  \color{white} 0 &  \color{white} 0 &  \color{white} 0 &  \color{white} 0 &  \color{white} 0 &  \color{white} 0 &  \color{white} 0 \\
     \textbf{+} &  \color{white} 0 &  \color{white} 0 &  \color{white} 0 &  \color{white} 0 &  \color{white} 0 &  \color{white} 0 &  \color{white} 0 &  \color{white} 0 &  \color{white} 0 &  \color{white} 0 \\
      \color{white} 0 &  \color{white} 0 &  \color{white} 0 &  \color{white} 0 &  \color{white} 0 & \textbf{+} &  \color{white} 0 &  \color{white} 0 &  \color{white} 0 &  \color{white} 0 &  \color{white} 0 \\
      \color{white} 0 &  \color{white} 0 &  \color{white} 0 &  \color{white} 0 & \textbf{+} & \textbf{-} & \textbf{+} &  \color{white} 0 &  \color{white} 0 &  \color{white} 0 &  \color{white} 0 \\
      \color{white} 0 &  \color{white} 0 &  \color{white} 0 & \textbf{+} & \textbf{-} & \textbf{+} & \textbf{-} & \textbf{+} &  \color{white} 0 &  \color{white} 0 &  \color{white} 0 \\
      \color{white} 0 &  \color{white} 0 &  \color{white} 0 &  \color{white} 0 & \textbf{+} & \textbf{-} & \textbf{+} &  \color{white} 0 &  \color{white} 0 &  \color{white} 0 &  \color{white} 0 \\
      \color{white} 0 &  \color{white} 0 &  \color{white} 0 &  \color{white} 0 &  \color{white} 0 & \textbf{+} &  \color{white} 0 &  \color{white} 0 &  \color{white} 0 &  \color{white} 0 &  \color{white} 0 \\
     \vspace{-0.1cm} \color{white} 0 &  \color{white} 0 &  \color{white} 0 &  \color{white} 0 &  \color{white} 0 &  \color{white} 0 &  \color{white} 0 &  \color{white} 0 &  \color{white} 0 &  \color{white} 0 & \textbf{+} \\
      \color{white} 0 &  \color{white} 0 &  \color{white} 0 &  \color{white} 0 &  \color{white} 0 &  \color{white} 0 &  \color{white} 0 &  \color{white} 0 &  \color{white} 0 & \iddots &  \color{white} 0 \\
      \color{white} 0 &  \color{white} 0 &  \color{white} 0 &  \color{white} 0 &  \color{white} 0 &  \color{white} 0 &  \color{white} 0 &  \color{white} 0 & \textbf{+} &  \color{white} 0 &  \color{white} 0}
\hspace{-0.12cm}\nearrow
\hdots
\nearrow\hspace{-0.12cm}\SmallMatrix{
      \color{white} 0 & \textbf{+} &  \color{white} 0 &  \color{white} 0 &  \color{white} 0 &  \color{white} 0 &  \color{white} 0 &  \color{white} 0 &  \color{white} 0 &  \color{white} 0 &  \color{white} 0 \\
     \textbf{+} &  \color{white} 0 &  \color{white} 0 &  \color{white} 0 &  \color{white} 0 &  \color{white} 0 &  \color{white} 0 &  \color{white} 0 &  \color{white} 0 &  \color{white} 0 &  \color{white} 0 \\
      \color{white} 0 &  \color{white} 0 &  \color{white} 0 &  \color{white} 0 &  \color{white} 0 & \textbf{+} &  \color{white} 0 &  \color{white} 0 &  \color{white} 0 &  \color{white} 0 &  \color{white} 0 \\
     \vspace{-0.1cm} \color{white} 0 &  \color{white} 0 &  \color{white} 0 &  \color{white} 0 & \textbf{+} & \textbf{-} & \textbf{+} &  \color{white} 0 &  \color{white} 0 &  \color{white} 0 &  \color{white} 0 \\
      \color{white} 0 &  \color{white} 0 &  \color{white} 0 & \iddots &  \textbf{-} &  \textbf{+} &  \textbf{-} & \ddots &  \color{white} 0 &  \color{white} 0 &  \color{white} 0 \\
     \vspace{-0.1cm} \color{white} 0 &  \color{white} 0 & \textbf{+} & \textbf{-} &  \textbf{+} & \hdots &  \textbf{+} & \textbf{-} & \textbf{+} &  \color{white} 0 &  \color{white} 0 \\
      \color{white} 0 &  \color{white} 0 &  \color{white} 0 & \ddots &  \textbf{-} &  \textbf{+} &  \textbf{-} & \iddots &  \color{white} 0 &  \color{white} 0 &  \color{white} 0 \\
      \color{white} 0 &  \color{white} 0 &  \color{white} 0 &  \color{white} 0 & \textbf{+} & \textbf{-} & \textbf{+} &  \color{white} 0 &  \color{white} 0 &  \color{white} 0 &  \color{white} 0 \\
      \color{white} 0 &  \color{white} 0 &  \color{white} 0 &  \color{white} 0 &  \color{white} 0 & \textbf{+} &  \color{white} 0 &  \color{white} 0 &  \color{white} 0 &  \color{white} 0 &  \color{white} 0 \\
      \color{white} 0 &  \color{white} 0 &  \color{white} 0 &  \color{white} 0 &  \color{white} 0 &  \color{white} 0 &  \color{white} 0 &  \color{white} 0 &  \color{white} 0 &  \color{white} 0 & \textbf{+} \\
      \color{white} 0 &  \color{white} 0 &  \color{white} 0 &  \color{white} 0 &  \color{white} 0 &  \color{white} 0 &  \color{white} 0 &  \color{white} 0 &  \color{white} 0 & \textbf{+} &  \color{white} 0}
\hspace{-0.12cm}\nearrow\hspace{-0.12cm}\SmallMatrix{
     \vspace{-0.1cm} \color{white} 0 &  \color{white} 0 &  \color{white} 0 &  \color{white} 0 &  \color{white} 0 &  \color{white} 0 &  \color{white} 0 & \textbf{+} &  \color{white} 0 &  \color{white} 0 &  \color{white} 0 \\
     \vspace{-0.1cm} \color{white} 0 &  \color{white} 0 &  \color{white} 0 &  \color{white} 0 &  \color{white} 0 & \textbf{+} &  \color{white} 0 & \textbf{-} & \ddots &  \color{white} 0 &  \color{white} 0 \\
      \color{white} 0 &  \color{white} 0 &  \color{white} 0 &  \color{white} 0 & \textbf{+} & \textbf{-} & \textbf{+} &  \color{white} 0 & \ddots & \ddots &  \color{white} 0 \\
     \vspace{-0.1cm} \color{white} 0 &  \color{white} 0 &  \color{white} 0 & \textbf{+} & \textbf{-} & \hdots & \textbf{-} & \textbf{+} &  \color{white} 0 & \textbf{-} & \textbf{+} \\
      \color{white} 0 &  \color{white} 0 & \iddots &  \textbf{-} &  \textbf{+} &  \textbf{-} &  \textbf{+} &  \textbf{-} & \ddots &  \color{white} 0 &  \color{white} 0 \\
     \vspace{-0.1cm} \color{white} 0 & \textbf{+} & \textbf{-} &  \textbf{+} &  \textbf{-} & \hdots &  \textbf{-} &  \textbf{+} & \textbf{-} & \textbf{+} &  \color{white} 0 \\
      \color{white} 0 &  \color{white} 0 & \ddots &  \textbf{-} &  \textbf{+} &  \textbf{-} &  \textbf{+} &  \textbf{-} & \iddots &  \color{white} 0 &  \color{white} 0 \\
     \vspace{-0.1cm}\textbf{+} & \textbf{-} &  \color{white} 0 & \textbf{+} & \textbf{-} & \hdots & \textbf{-} & \textbf{+} &  \color{white} 0 &  \color{white} 0 &  \color{white} 0 \\
     \vspace{-0.1cm} \color{white} 0 & \ddots & \ddots &  \color{white} 0 & \textbf{+} & \textbf{-} & \textbf{+} &  \color{white} 0 &  \color{white} 0 &  \color{white} 0 &  \color{white} 0 \\
      \color{white} 0 &  \color{white} 0 & \ddots & \textbf{-} &  \color{white} 0 & \textbf{+} &  \color{white} 0 &  \color{white} 0 &  \color{white} 0 &  \color{white} 0 &  \color{white} 0 \\
      \color{white} 0 &  \color{white} 0 &  \color{white} 0 & \textbf{+} &  \color{white} 0 &  \color{white} 0 &  \color{white} 0 &  \color{white} 0 &  \color{white} 0 &  \color{white} 0 &  \color{white} 0}
\hspace{-0.12cm}\nearrow\hspace{-0.12cm}\SmallMatrix{
      \color{white} 0 &  \color{white} 0 &  \color{white} 0 &  \color{white} 0 &  \color{white} 0 & \textbf{+} &  \color{white} 0 &  \color{white} 0 &  \color{white} 0 &  \color{white} 0 &  \color{white} 0 \\
      \color{white} 0 &  \color{white} 0 &  \color{white} 0 &  \color{white} 0 & \textbf{+} & \textbf{-} & \textbf{+} &  \color{white} 0 &  \color{white} 0 &  \color{white} 0 &  \color{white} 0 \\
     \vspace{-0.1cm} \color{white} 0 &  \color{white} 0 &  \color{white} 0 & \textbf{+} & \textbf{-} & \hdots & \textbf{-} & \textbf{+} &  \color{white} 0 &  \color{white} 0 &  \color{white} 0 \\
      \color{white} 0 &  \color{white} 0 & \iddots &  \textbf{-} &  \textbf{+} &  \textbf{-} &  \textbf{+} &  \textbf{-} & \ddots &  \color{white} 0 &  \color{white} 0 \\
      \color{white} 0 &  \textbf{+} &  \textbf{-} &  \textbf{+} &  \textbf{-} &  \hdots &  \textbf{-} &  \textbf{+} &  \textbf{-} &  \textbf{+} &  \color{white} 0 \\
     \textbf{+} & \textbf{-} &  \textbf{+} &  \textbf{-} &  \textbf{+} & \hdots &  \textbf{+} &  \textbf{-} &  \textbf{+} & \textbf{-} & \textbf{+} \\
     \vspace{-0.1cm} \color{white} 0 &  \textbf{+} &  \textbf{-} &  \textbf{+} &  \textbf{-} &  \hdots &  \textbf{-} &  \textbf{+} &  \textbf{-} &  \textbf{+} &  \color{white} 0 \\
      \color{white} 0 &  \color{white} 0 & \ddots &  \textbf{-} &  \textbf{+} &  \textbf{-} &  \textbf{+} &  \textbf{-} & \iddots &  \color{white} 0 &  \color{white} 0 \\
      \color{white} 0 &  \color{white} 0 &  \color{white} 0 & \textbf{+} & \textbf{-} & \hdots & \textbf{-} & \textbf{+} &  \color{white} 0 &  \color{white} 0 &  \color{white} 0 \\
      \color{white} 0 &  \color{white} 0 &  \color{white} 0 &  \color{white} 0 & \textbf{+} & \textbf{-} & \textbf{+} &  \color{white} 0 &  \color{white} 0 &  \color{white} 0 &  \color{white} 0 \\
      \color{white} 0 &  \color{white} 0 &  \color{white} 0 &  \color{white} 0 &  \color{white} 0 & \textbf{+} &  \color{white} 0 &  \color{white} 0 &  \color{white} 0 &  \color{white} 0 &  \color{white} 0}
\]
\[
\hspace{-0.12cm}\nearrow\hspace{-0.12cm}\SmallMatrix{
     \vspace{-0.1cm} \color{white} 0 &  \color{white} 0 &  \color{white} 0 &  \textbf{+} &  \color{white} 0 &  \color{white} 0 &  \color{white} 0 & \color{white} 0 &  \color{white} 0 &  \color{white} 0 &  \color{white} 0 \\
     \vspace{-0.1cm} \color{white} 0 &  \color{white} 0 &  \iddots &  \textbf{-} &  \color{white} 0 & \textbf{+} &  \color{white} 0 & \color{white} 0 & \color{white} 0 &  \color{white} 0 &  \color{white} 0 \\
      \color{white} 0 &  \iddots &  \iddots &  \color{white} 0 & \textbf{+} & \textbf{-} & \textbf{+} &  \color{white} 0 & \color{white} 0 & \color{white} 0 &  \color{white} 0 \\
     \vspace{-0.1cm}\textbf{+} &  \textbf{-} &  \color{white} 0 & \textbf{+} & \textbf{-} & \hdots & \textbf{-} & \textbf{+} &  \color{white} 0 & \color{white} 0 & \color{white} 0 \\
      \color{white} 0 &  \color{white} 0 & \iddots &  \textbf{-} &  \textbf{+} &  \textbf{-} &  \textbf{+} &  \textbf{-} & \ddots &  \color{white} 0 &  \color{white} 0 \\
     \vspace{-0.1cm} \color{white} 0 & \textbf{+} & \textbf{-} &  \textbf{+} &  \textbf{-} & \hdots &  \textbf{-} &  \textbf{+} & \textbf{-} & \textbf{+} &  \color{white} 0 \\
      \color{white} 0 &  \color{white} 0 & \ddots &  \textbf{-} &  \textbf{+} &  \textbf{-} &  \textbf{+} &  \textbf{-} & \iddots &  \color{white} 0 &  \color{white} 0 \\
     \vspace{-0.1cm} \color{white} 0 & \color{white} 0 &  \color{white} 0 & \textbf{+} & \textbf{-} & \hdots & \textbf{-} & \textbf{+} &  \color{white} 0 &  \textbf{-} &  \textbf{+} \\
     \vspace{-0.1cm} \color{white} 0 & \color{white} 0 & \color{white} 0 &  \color{white} 0 & \textbf{+} & \textbf{-} & \textbf{+} &  \color{white} 0 &  \iddots &  \iddots &  \color{white} 0 \\
      \color{white} 0 &  \color{white} 0 & \color{white} 0 & \color{white} 0 &  \color{white} 0 & \textbf{+} &  \color{white} 0 &  \textbf{-} &  \iddots &  \color{white} 0 &  \color{white} 0 \\
      \color{white} 0 &  \color{white} 0 &  \color{white} 0 & \color{white} 0 &  \color{white} 0 &  \color{white} 0 &  \color{white} 0 &  \textbf{+} &  \color{white} 0 &  \color{white} 0 &  \color{white} 0}
\hspace{-0.12cm}\nearrow\hspace{-0.12cm}\SmallMatrix{
      \color{white} 0 & \color{white} 0 &  \color{white} 0 &  \color{white} 0 &  \color{white} 0 &  \color{white} 0 &  \color{white} 0 &  \color{white} 0 &  \color{white} 0 &  \textbf{+} &  \color{white} 0 \\
     \color{white} 0 &  \color{white} 0 &  \color{white} 0 &  \color{white} 0 &  \color{white} 0 &  \color{white} 0 &  \color{white} 0 &  \color{white} 0 &  \color{white} 0 &  \color{white} 0 &  \textbf{+} \\
      \color{white} 0 &  \color{white} 0 &  \color{white} 0 &  \color{white} 0 &  \color{white} 0 & \textbf{+} &  \color{white} 0 &  \color{white} 0 &  \color{white} 0 &  \color{white} 0 &  \color{white} 0 \\
     \vspace{-0.1cm} \color{white} 0 &  \color{white} 0 &  \color{white} 0 &  \color{white} 0 & \textbf{+} & \textbf{-} & \textbf{+} &  \color{white} 0 &  \color{white} 0 &  \color{white} 0 &  \color{white} 0 \\
      \color{white} 0 &  \color{white} 0 &  \color{white} 0 & \iddots &  \textbf{-} &  \textbf{+} &  \textbf{-} & \ddots &  \color{white} 0 &  \color{white} 0 &  \color{white} 0 \\
     \vspace{-0.1cm} \color{white} 0 &  \color{white} 0 & \textbf{+} & \textbf{-} &  \textbf{+} & \hdots &  \textbf{+} & \textbf{-} & \textbf{+} &  \color{white} 0 &  \color{white} 0 \\
      \color{white} 0 &  \color{white} 0 &  \color{white} 0 & \ddots &  \textbf{-} &  \textbf{+} &  \textbf{-} & \iddots &  \color{white} 0 &  \color{white} 0 &  \color{white} 0 \\
      \color{white} 0 &  \color{white} 0 &  \color{white} 0 &  \color{white} 0 & \textbf{+} & \textbf{-} & \textbf{+} &  \color{white} 0 &  \color{white} 0 &  \color{white} 0 &  \color{white} 0 \\
      \color{white} 0 &  \color{white} 0 &  \color{white} 0 &  \color{white} 0 &  \color{white} 0 & \textbf{+} &  \color{white} 0 &  \color{white} 0 &  \color{white} 0 &  \color{white} 0 &  \color{white} 0 \\
      \textbf{+} &  \color{white} 0 &  \color{white} 0 &  \color{white} 0 &  \color{white} 0 &  \color{white} 0 &  \color{white} 0 &  \color{white} 0 &  \color{white} 0 &  \color{white} 0 & \color{white} 0 \\
      \color{white} 0 &  \textbf{+} &  \color{white} 0 &  \color{white} 0 &  \color{white} 0 &  \color{white} 0 &  \color{white} 0 &  \color{white} 0 &  \color{white} 0 & \color{white} 0 &  \color{white} 0}
\hspace{-0.12cm}\nearrow
\hdots
\nearrow\hspace{-0.12cm}\SmallMatrix{
     \vspace{-0.1cm} \color{white} 0 &  \color{white} 0 & \color{white} 0 &  \color{white} 0 &  \color{white} 0 &  \color{white} 0 &  \color{white} 0 &  \color{white} 0 &  \textbf{+} &  \color{white} 0 &  \color{white} 0 \\
      \color{white} 0 & \color{white} 0 &  \color{white} 0 &  \color{white} 0 &  \color{white} 0 &  \color{white} 0 &  \color{white} 0 &  \color{white} 0 &  \color{white} 0 &  \ddots &  \color{white} 0 \\
      \color{white} 0 &  \color{white} 0 &  \color{white} 0 &  \color{white} 0 &  \color{white} 0 &  \color{white} 0 &  \color{white} 0 &  \color{white} 0 &  \color{white} 0 &  \color{white} 0 &  \textbf{+} \\
      \color{white} 0 &  \color{white} 0 &  \color{white} 0 &  \color{white} 0 &  \color{white} 0 & \textbf{+} &  \color{white} 0 &  \color{white} 0 &  \color{white} 0 &  \color{white} 0 &  \color{white} 0 \\
      \color{white} 0 &  \color{white} 0 &  \color{white} 0 &  \color{white} 0 & \textbf{+} & \textbf{-} & \textbf{+} &  \color{white} 0 &  \color{white} 0 &  \color{white} 0 &  \color{white} 0 \\
      \color{white} 0 &  \color{white} 0 &  \color{white} 0 & \textbf{+} & \textbf{-} & \textbf{+} & \textbf{-} & \textbf{+} &  \color{white} 0 &  \color{white} 0 &  \color{white} 0 \\
      \color{white} 0 &  \color{white} 0 &  \color{white} 0 &  \color{white} 0 & \textbf{+} & \textbf{-} & \textbf{+} &  \color{white} 0 &  \color{white} 0 &  \color{white} 0 &  \color{white} 0 \\
      \color{white} 0 &  \color{white} 0 &  \color{white} 0 &  \color{white} 0 &  \color{white} 0 & \textbf{+} &  \color{white} 0 &  \color{white} 0 &  \color{white} 0 &  \color{white} 0 &  \color{white} 0 \\
     \vspace{-0.1cm} \textbf{+} &  \color{white} 0 &  \color{white} 0 &  \color{white} 0 &  \color{white} 0 &  \color{white} 0 &  \color{white} 0 &  \color{white} 0 &  \color{white} 0 &  \color{white} 0 & \color{white} 0 \\
      \color{white} 0 & \ddots &  \color{white} 0 &  \color{white} 0 &  \color{white} 0 &  \color{white} 0 &  \color{white} 0 &  \color{white} 0 &  \color{white} 0 & \color{white} 0 &  \color{white} 0 \\
      \color{white} 0 &  \color{white} 0 &  \textbf{+} &  \color{white} 0 &  \color{white} 0 &  \color{white} 0 &  \color{white} 0 &  \color{white} 0 & \color{white} 0 &  \color{white} 0 &  \color{white} 0}
\hspace{-0.12cm}\nearrow\hspace{-0.12cm}\SmallMatrix{
     \color{white} 0 &  \color{white} 0 &  \color{white} 0 &  \color{white} 0 &  \color{white} 0 &  \color{white} 0 &  \color{white} 0 &  \color{white} 0 &  \color{white} 0 &  \color{white} 0 &  \textbf{+} \\
     \vspace{-0.1cm} \color{white} 0 &  \color{white} 0 &  \color{white} 0 & \color{white} 0 &  \color{white} 0 &  \color{white} 0 &  \color{white} 0 &  \textbf{+} &  \color{white} 0 &  \color{white} 0 &  \color{white} 0 \\
      \color{white} 0 &  \color{white} 0 & \color{white} 0 &  \color{white} 0 &  \color{white} 0 &  \color{white} 0 &  \color{white} 0 &  \color{white} 0 &  \ddots &  \color{white} 0 &  \color{white} 0 \\
      \color{white} 0 & \color{white} 0 &  \color{white} 0 &  \color{white} 0 &  \color{white} 0 &  \color{white} 0 &  \color{white} 0 &  \color{white} 0 &  \color{white} 0 &  \textbf{+} &  \color{white} 0 \\
      \color{white} 0 &  \color{white} 0 &  \color{white} 0 &  \color{white} 0 &  \color{white} 0 & \textbf{+} &  \color{white} 0 &  \color{white} 0 &  \color{white} 0 &  \color{white} 0 &  \color{white} 0 \\
      \color{white} 0 &  \color{white} 0 &  \color{white} 0 &  \color{white} 0 & \textbf{+} & \textbf{-} & \textbf{+} &  \color{white} 0 &  \color{white} 0 &  \color{white} 0 &  \color{white} 0 \\
      \color{white} 0 &  \color{white} 0 &  \color{white} 0 &  \color{white} 0 &  \color{white} 0 & \textbf{+} &  \color{white} 0 &  \color{white} 0 &  \color{white} 0 &  \color{white} 0 &  \color{white} 0 \\
     \vspace{-0.1cm} \color{white} 0 &  \textbf{+} &  \color{white} 0 &  \color{white} 0 &  \color{white} 0 &  \color{white} 0 &  \color{white} 0 &  \color{white} 0 &  \color{white} 0 & \color{white} 0 &  \color{white} 0 \\
      \color{white} 0 &  \color{white} 0 &  \ddots &  \color{white} 0 &  \color{white} 0 &  \color{white} 0 &  \color{white} 0 &  \color{white} 0 & \color{white} 0 &  \color{white} 0 &  \color{white} 0 \\
      \color{white} 0 &  \color{white} 0 &  \color{white} 0 &  \textbf{+} &  \color{white} 0 &  \color{white} 0 &  \color{white} 0 & \color{white} 0 &  \color{white} 0 &  \color{white} 0 &  \color{white} 0 \\
      \textbf{+} &  \color{white} 0 &  \color{white} 0 &  \color{white} 0 &  \color{white} 0 &  \color{white} 0 &  \color{white} 0 &  \color{white} 0 &  \color{white} 0 &  \color{white} 0 & \color{white} 0}
\hspace{-0.12cm}\nearrow\hspace{-0.12cm}\SmallMatrix{
      \color{white} 0 &  \color{white} 0 &  \color{white} 0 &  \color{white} 0 &  \textbf{+} &  \color{white} 0 & \color{white} 0 &  \color{white} 0 &  \color{white} 0 &  \color{white} 0 &  \color{white} 0 \\
     \vspace{-0.1cm} \color{white} 0 &  \color{white} 0 &  \color{white} 0 &  \textbf{+} &  \color{white} 0 &  \color{white} 0 &  \color{white} 0 & \color{white} 0 &  \color{white} 0 &  \color{white} 0 &  \color{white} 0 \\
      \color{white} 0 &  \color{white} 0 &  \iddots &  \color{white} 0 &  \color{white} 0 &  \color{white} 0 &  \color{white} 0 &  \color{white} 0 & \color{white} 0 &  \color{white} 0 &  \color{white} 0 \\
      \color{white} 0 &  \textbf{+} &  \color{white} 0 &  \color{white} 0 &  \color{white} 0 &  \color{white} 0 &  \color{white} 0 &  \color{white} 0 &  \color{white} 0 & \color{white} 0 &  \color{white} 0 \\
     \textbf{+} &  \color{white} 0 &  \color{white} 0 &  \color{white} 0 &  \color{white} 0 &  \color{white} 0 &  \color{white} 0 &  \color{white} 0 &  \color{white} 0 &  \color{white} 0 & \color{white} 0 \\
      \color{white} 0 &  \color{white} 0 &  \color{white} 0 &  \color{white} 0 &  \color{white} 0 & \textbf{+} &  \color{white} 0 &  \color{white} 0 &  \color{white} 0 &  \color{white} 0 &  \color{white} 0 \\
      \color{white} 0 &  \color{white} 0 &  \color{white} 0 &  \color{white} 0 &  \color{white} 0 &  \color{white} 0 &  \color{white} 0 &  \color{white} 0 &  \color{white} 0 &  \color{white} 0 &  \textbf{+} \\
     \vspace{-0.1cm} \color{white} 0 & \color{white} 0 &  \color{white} 0 &  \color{white} 0 &  \color{white} 0 &  \color{white} 0 &  \color{white} 0 &  \color{white} 0 &  \color{white} 0 &  \textbf{+} &  \color{white} 0 \\
      \color{white} 0 &  \color{white} 0 & \color{white} 0 &  \color{white} 0 &  \color{white} 0 &  \color{white} 0 &  \color{white} 0 &  \color{white} 0 &  \iddots &  \color{white} 0 &  \color{white} 0 \\
      \color{white} 0 &  \color{white} 0 &  \color{white} 0 & \color{white} 0 &  \color{white} 0 &  \color{white} 0 &  \color{white} 0 &  \textbf{+} &  \color{white} 0 &  \color{white} 0 &  \color{white} 0 \\
      \color{white} 0 &  \color{white} 0 &  \color{white} 0 &  \color{white} 0 & \color{white} 0 &  \color{white} 0 &  \textbf{+} &  \color{white} 0 &  \color{white} 0 &  \color{white} 0 &  \color{white} 0}
\]

It can be easily seen that each plane of this hypermatrix is an ASM. All vertical lines of $A$ corresponding to diamond positions of $L$ clearly have the alternating property. All vertical lines corresponding to the diagonal from $(2,p+2)$ to $(p-2,n-1)$, the diagonal from $(p+2,2)$ to $(n-1,p-2)$, the anti-diagonal from $(2,p-2)$ to $(p-2,2)$, and the anti-diagonal from $(p+2,n-1)$ to $(n-1, p+2)$ have exactly three non-zero entries, which alternate $+, - , +$. All other vertical lines contain exactly one non-zero entry, namely one $+$ entry. Therefore this is an ASHM.

\*

The $p$ entries occur in the following positions of the corresponding ASHL.
\[\Matrix{
    \color{white} 0 & \color{white} 0 & \color{white} 0 & \color{white} 0 & \color{white} 0\color{white} 0 & p & \color{white} 0 & \color{white} 0 & \color{white} 0 & \color{white} 0 & \color{white} 00 \\
    \vspace{-0.1cm}\color{white} 0 & \color{white} 0 & \color{white} 0 & p & p & p & p & p & \color{white} 0 & \color{white} 0 & \color{white} 0 \\
    \vspace{-0.1cm}\color{white} 0 & \color{white} 0 & \iddots & \iddots &  p & \vdots & p & \ddots & \ddots & \color{white} 0 & \color{white} 0 \\
    \color{white} 0& p & \iddots  & \iddots &\color{white}  p &\color{white}  p & \color{white} p &\ddots & \ddots & p & \color{white} 0 \\
    \color{white} 0\color{white} 0 & p &  p &\color{white}  p & \color{white} p &\color{white}  p &\color{white}  p &\color{white}  p &  p & p & \color{white} 0 \\
    p & p & \hdots & \color{white} p & \color{white} p & \color{white} p & \color{white} p & \color{white} p & \hdots & p & p \\
    \vspace{-0.1cm}\color{white} 0 & p &  p & \color{white} p & \color{white} p & \color{white} p & \color{white} p & \color{white} p & p & p & \color{white} 0\color{white} 0 \\
    \vspace{-0.1cm}\color{white} 0 & p & \ddots & \ddots & \color{white} p & \color{white} p & \color{white} p & \iddots & \iddots & p & \color{white} 0 \\
    \color{white} 0 & \color{white} 0 & \ddots & \ddots & p & \vdots & p & \iddots & \iddots & \color{white} 0 & \color{white} 0 \\
    \color{white} 0 & \color{white} 0 & \color{white} 0 & p & p & p & p & p & \color{white} 0 & \color{white} 0 & \color{white} 0 \\
    \color{white} 00 & \color{white} 0 & \color{white} 0 & \color{white} 0 & \color{white} 0 & p & \color{white} 0\color{white} 0 & \color{white} 0 & \color{white} 0 & \color{white} 0 & \color{white} 0}
\]

As outlined above, $p$ occurs as an entry in the diamond positions of $L$, and also occurs as every entry in the diagonal from $L_{2,p+2}$ to $L_{p-2,n-1}$, the diagonal from $L_{p+2,2}$ to $L_{n-1,p-2}$, the anti-diagonal from $L_{2,p-2}$ to $L_{p-2,2}$, and the anti-diagonal from $L_{p+2,n-1}$ to $L_{n-1, p+2}$.

\*

Therefore $p$ occurs as an entry of $L$ a total of $\frac{n^2+4n-19}{2}$ times:
\[(1 + 3 + \dots + n-2 + n + n-2 + \dots + 3 +1) + 4\Big(\frac{n+1}{2}-3\Big) = \Big(\frac{n+1}{2}\Big)^2 + \Big(\frac{n-1}{2}\Big)^2 + (2n -10) = \frac{n^2+4n-19}{2}\]

\underline{The even case:}
\[
\hspace{-0.4cm}\SmallMatrix{
     \color{white} 0 & \color{white} 0 & \color{white} 0 & \color{white} 0 & \color{white} 0 & \color{white} 0 & \color{white} 0 & \textbf{+} & \color{white} 0 & \color{white} 0 & \color{white} 0 & \color{white} 0 \\
     \vspace{-0.1cm}\color{white} 0 & \color{white}0 & \color{white} 0 & \color{white} 0 & \color{white} 0 & \color{white} 0 & \color{white} 0 & \color{white} 0 & \textbf{+} & \color{white} 0 & \color{white} 0 & \color{white} 0 \\
     \color{white} 0 & \color{white} 0 & \color{white} 0 & \color{white} 0 & \color{white} 0 & \color{white} 0 & \color{white} 0 & \color{white} 0 & \color{white} 0 & \ddots & \color{white} 0 & \color{white} 0 \\
     \color{white} 0 & \color{white} 0 & \color{white} 0 & \color{white} 0 & \color{white} 0 & \color{white} 0 & \color{white} 0 & \color{white} 0 & \color{white} 0 & \color{white} 0 & \textbf{+} & \color{white} 0 \\
     \color{white} 0 & \color{white} 0 & \color{white} 0 & \color{white} 0 & \color{white} 0 & \color{white} 0 & \color{white} 0 & \color{white} 0 & \color{white} 0 & \color{white} 0 & \color{white} 0 & \textbf{+} \\
     \color{white} 0 & \color{white} 0 & \color{white} 0 & \color{white} 0 & \color{white} 0 & \color{white} 0 & \textbf{+} & \color{white} 0 & \color{white} 0 & \color{white} 0 & \color{white} 0 & \color{white} 0 \\
     \color{white} 0 & \color{white} 0 & \color{white} 0 & \color{white} 0 & \color{white} 0 & \textbf{+} & \color{white} 0 & \color{white} 0 & \color{white} 0 & \color{white} 0 & \color{white} 0 & \color{white} 0 \\
     \textbf{+} & \color{white} 0 & \color{white} 0 & \color{white} 0 & \color{white} 0 & \color{white} 0 & \color{white} 0 & \color{white} 0 & \color{white} 0 & \color{white} 0 & \color{white} 0 & \color{white} 0 \\
     \vspace{-0.1cm}\color{white}0 & \textbf{+} & \color{white} 0 & \color{white} 0 & \color{white} 0 & \color{white} 0 & \color{white} 0 & \color{white} 0 & \color{white} 0 & \color{white} 0 & \color{white} 0 & \color{white} 0 \\
     \color{white} 0 & \color{white} 0 & \ddots & \color{white} 0 & \color{white} 0 & \color{white} 0 & \color{white} 0 & \color{white} 0 & \color{white} 0 & \color{white} 0 & \color{white} 0 & \color{white} 0 \\
     \color{white} 0 & \color{white} 0 & \color{white} 0 & \textbf{+} & \color{white} 0 & \color{white} 0 & \color{white} 0 & \color{white} 0 & \color{white} 0 & \color{white} 0 & \color{white} 0 & \color{white} 0 \\
     \color{white} 0 & \color{white} 0 & \color{white} 0 & \color{white} 0 & \textbf{+} & \color{white} 0 & \color{white} 0 & \color{white} 0 & \color{white} 0 & \color{white} 0 &\color{white} 0 &  \color{white} 0}
\hspace{-0.12cm}\nearrow\hspace{-0.12cm}\SmallMatrix{
     \textbf{+} & \color{white} 0 & \color{white} 0 & \color{white} 0 & \color{white} 0 & \color{white} 0 & \color{white} 0 & \color{white} 0 & \color{white} 0 & \color{white} 0 & \color{white} 0 & \color{white} 0 \\
     \vspace{-0.1cm}\color{white}0 & \color{white} 0 & \color{white} 0 & \textbf{+} & \color{white} 0 & \color{white} 0 & \color{white} 0 & \color{white} 0 & \color{white} 0 & \color{white} 0 & \color{white} 0 & \color{white} 0 \\
     \color{white} 0 & \color{white} 0 & \iddots & \color{white} 0 & \color{white} 0 & \color{white} 0 & \color{white} 0 & \color{white} 0 & \color{white} 0 & \color{white} 0 & \color{white} 0 & \color{white} 0 \\
     \color{white} 0 & \textbf{+} & \color{white} 0 & \color{white} 0 & \color{white} 0 & \color{white} 0 & \color{white} 0 & \color{white} 0 & \color{white} 0 & \color{white} 0 & \color{white} 0 & \color{white} 0 \\
     \color{white} 0 & \color{white} 0 & \color{white} 0 & \color{white} 0 & \color{white} 0 & \color{white} 0 & \textbf{+} & \color{white} 0 & \color{white} 0 & \color{white} 0 & \color{white} 0 & \color{white} 0 \\
     \color{white} 0 & \color{white} 0 & \color{white} 0 & \color{white} 0 & \color{white} 0 & \textbf{+} & \textbf{-} & \textbf{+} & \color{white} 0 & \color{white} 0 & \color{white} 0 & \color{white} 0 \\
     \color{white} 0 & \color{white} 0 & \color{white} 0 & \color{white} 0 & \textbf{+} & \textbf{-} & \textbf{+} & \color{white} 0 & \color{white} 0 & \color{white} 0 & \color{white} 0 & \color{white} 0 \\
     \color{white} 0 & \color{white} 0 & \color{white} 0 & \color{white} 0 & \color{white} 0 & \textbf{+} & \color{white} 0 & \color{white} 0 & \color{white} 0 & \color{white} 0 & \color{white} 0 & \color{white} 0 \\
     \vspace{-0.1cm}\color{white}0 & \color{white} 0 & \color{white} 0 & \color{white} 0 & \color{white} 0 & \color{white} 0 & \color{white} 0 & \color{white} 0 & \color{white} 0 & \color{white} 0 & \textbf{+} & \color{white} 0 \\
     \color{white} 0 & \color{white} 0 & \color{white} 0 & \color{white} 0 & \color{white} 0 & \color{white} 0 & \color{white} 0 & \color{white} 0 & \color{white} 0 & \iddots & \color{white} 0 & \color{white} 0 \\
     \color{white} 0 & \color{white} 0 & \color{white} 0 & \color{white} 0 & \color{white} 0 & \color{white} 0 & \color{white} 0 & \color{white} 0 & \textbf{+} & \color{white} 0 & \color{white} 0 & \color{white} 0 \\
     \color{white} 0 & \color{white} 0 & \color{white} 0 & \color{white} 0 & \color{white} 0 & \color{white} 0 & \color{white} 0 & \color{white} 0 & \color{white} 0 & \color{white} 0 & \color{white} 0 & \textbf{+}}
\hspace{-0.12cm}\nearrow\hspace{-0.12cm}\SmallMatrix{
     \vspace{-0.1cm}\color{white}0 & \color{white} 0 & \textbf{+} & \color{white} 0 & \color{white} 0 & \color{white} 0 & \color{white} 0 & \color{white} 0 & \color{white} 0 & \color{white} 0 & \color{white} 0 & \color{white} 0 \\
     \color{white} 0 & \iddots & \color{white} 0 & \color{white} 0 & \color{white} 0 & \color{white} 0 & \color{white} 0 & \color{white} 0 & \color{white} 0 & \color{white} 0 & \color{white} 0 & \color{white} 0 \\
     \textbf{+} & \color{white} 0 & \color{white} 0 & \color{white} 0 & \color{white} 0 & \color{white} 0 & \color{white} 0 & \color{white} 0 & \color{white} 0 & \color{white} 0 & \color{white} 0 & \color{white} 0 \\
     \color{white} 0 & \color{white} 0 & \color{white} 0 & \color{white} 0 & \color{white} 0 & \color{white} 0 & \textbf{+} & \color{white} 0 & \color{white} 0 & \color{white} 0 & \color{white} 0 & \color{white} 0 \\
     \color{white} 0 & \color{white} 0 & \color{white} 0 & \color{white} 0 & \color{white} 0 & \textbf{+} & \textbf{-} & \textbf{+} & \color{white} 0 & \color{white} 0 & \color{white} 0 & \color{white} 0 \\
     \color{white} 0 & \color{white} 0 & \color{white} 0 & \color{white} 0 & \textbf{+} & \textbf{-} & \textbf{+} & \textbf{-} & \textbf{+} & \color{white} 0 & \color{white} 0 & \color{white} 0 \\
     \color{white} 0 & \color{white} 0 & \color{white} 0 & \textbf{+} & \textbf{-} & \textbf{+} & \textbf{-} & \textbf{+} & \color{white} 0 & \color{white} 0 & \color{white} 0 & \color{white} 0 \\
     \color{white} 0 & \color{white} 0 & \color{white} 0 & \color{white} 0 & \textbf{+} & \textbf{-} & \textbf{+} & \color{white} 0 & \color{white} 0 & \color{white} 0 & \color{white} 0 & \color{white} 0 \\
     \color{white} 0 & \color{white} 0 & \color{white} 0 & \color{white} 0 & \color{white} 0 & \textbf{+} & \color{white} 0 & \color{white} 0 & \color{white} 0 & \color{white} 0 & \color{white} 0 & \color{white} 0 \\
     \vspace{-0.1cm}\color{white} 0 & \color{white}0 & \color{white} 0 & \color{white} 0 & \color{white} 0 & \color{white} 0 & \color{white} 0 & \color{white} 0 & \color{white} 0 & \color{white} 0 & \color{white} 0 & \textbf{+} \\
     \color{white} 0 & \color{white} 0 & \color{white} 0 & \color{white} 0 & \color{white} 0 & \color{white} 0 & \color{white} 0 & \color{white} 0 & \color{white} 0 & \color{white} 0 & \iddots & \color{white} 0 \\
     \color{white} 0 & \color{white} 0 & \color{white} 0 & \color{white} 0 & \color{white} 0 & \color{white} 0 & \color{white} 0 & \color{white} 0 & \color{white} 0 & \textbf{+} & \color{white} 0 & \color{white} 0}
\hspace{-0.12cm}\nearrow
\hdots
\nearrow\hspace{-0.12cm}\SmallMatrix{
     \color{white} 0 & \textbf{+} & \color{white} 0 & \color{white} 0 & \color{white} 0 & \color{white} 0 & \color{white} 0 & \color{white} 0 & \color{white} 0 & \color{white} 0 & \color{white} 0 & \color{white} 0 \\
     \textbf{+} & \color{white} 0 & \color{white} 0 & \color{white} 0 & \color{white} 0 & \color{white} 0 & \color{white} 0 & \color{white} 0 & \color{white} 0 & \color{white} 0 & \color{white} 0 & \color{white} 0 \\
     \color{white} 0 & \color{white} 0 & \color{white} 0 & \color{white} 0 & \color{white} 0 & \color{white} 0 & \textbf{+} & \color{white} 0 & \color{white} 0 & \color{white} 0 & \color{white} 0 & \color{white} 0 \\
     \vspace{-0.1cm}\color{white} 0 & \color{white}0 & \color{white} 0 & \color{white} 0 & \color{white} 0 & \textbf{+} & \textbf{-} & \textbf{+} & \color{white} 0 & \color{white} 0 & \color{white} 0 & \color{white} 0 \\
     \color{white} 0 & \color{white} 0 & \color{white} 0 & \color{white} 0 & \iddots & \color{white} \textbf{-} & \color{white} \textbf{+} & \color{white} \textbf{-} & \ddots & \color{white} 0 & \color{white} 0 & \color{white} 0 \\
     \vspace{-0.1cm}\color{white} 0 & \color{white}0 & \color{white} 0 & \textbf{+} & \textbf{-} & \color{white} \textbf{+} & \hdots & \color{white} \textbf{+} & \textbf{-} & \textbf{+} & \color{white} 0 & \color{white} 0 \\
     \vspace{-0.1cm}\color{white}0 & \color{white} 0 & \textbf{+} & \textbf{-} & \color{white} \textbf{+} & \hdots & \color{white} \textbf{+} & \textbf{-} & \textbf{+} & \color{white} 0 & \color{white} 0 & \color{white} 0 \\
     \color{white} 0 & \color{white} 0 & \color{white} 0 & \ddots & \color{white} \textbf{-} & \color{white} \textbf{+} & \color{white} \textbf{-} & \iddots & \color{white} 0 & \color{white} 0 & \color{white} 0 & \color{white} 0 \\
     \color{white} 0 & \color{white} 0 & \color{white} 0 & \color{white} 0 & \textbf{+} & \textbf{-} & \textbf{+} & \color{white} 0 & \color{white} 0 & \color{white} 0 & \color{white} 0 & \color{white} 0 \\
     \color{white} 0 & \color{white} 0 & \color{white} 0 & \color{white} 0 & \color{white} 0 & \textbf{+} & \color{white} 0 & \color{white} 0 & \color{white} 0 & \color{white} 0 & \color{white} 0 & \color{white} 0 \\
     \color{white} 0 & \color{white} 0 & \color{white} 0 & \color{white} 0 & \color{white} 0 & \color{white} 0 & \color{white} 0 & \color{white} 0 & \color{white} 0 & \color{white} 0 & \color{white} 0 & \textbf{+} \\
     \color{white} 0 & \color{white} 0 & \color{white} 0 & \color{white} 0 & \color{white} 0 & \color{white} 0 & \color{white} 0 & \color{white} 0 & \color{white} 0 & \color{white} 0 & \textbf{+} & \color{white} 0}
\hspace{-0.12cm}\nearrow\hspace{-0.12cm}\SmallMatrix{
     \vspace{-0.1cm}\color{white} 0 & \color{white}0 & \color{white} 0 & \color{white} 0 & \color{white} 0 & \color{white} 0 & \color{white} 0 & \color{white} 0 & \textbf{+} & \color{white} 0 & \color{white} 0 & \color{white} 0 \\
     \vspace{-0.1cm}\color{white} 0 & \color{white}0 & \color{white} 0 & \color{white} 0 & \color{white} 0 & \color{white} 0 & \textbf{+} & \color{white} 0 & \textbf{-} & \ddots & \color{white} 0 & \color{white} 0 \\
     \color{white} 0 & \color{white} 0 & \color{white} 0 & \color{white} 0 & \color{white} 0 & \textbf{+} & \textbf{-} & \textbf{+} & \color{white} 0 & \ddots & \ddots & \color{white} 0 \\
     \vspace{-0.1cm}\color{white} 0 & \color{white}0 & \color{white} 0 & \color{white} 0 & \textbf{+} & \textbf{-} & \hdots & \textbf{-} & \textbf{+} & \color{white} 0 & \textbf{-} & \textbf{+} \\
     \color{white} 0 & \color{white} 0 & \color{white} 0 & \iddots & \color{white} \textbf{-} & \color{white} \textbf{+} & \color{white} \textbf{-} & \color{white} \textbf{+} & \color{white} \textbf{-} & \ddots & \color{white} 0 & \color{white} 0 \\
     \vspace{-0.1cm}\color{white} 0 & \color{white}0 & \textbf{+} & \textbf{-} & \color{white} \textbf{+} & \color{white} \textbf{-} & \hdots & \color{white} \textbf{-} & \color{white} \textbf{+} & \textbf{-} & \textbf{+} & \color{white} 0 \\
     \vspace{-0.1cm}\color{white}0 & \textbf{+} & \textbf{-} & \color{white} \textbf{+} & \color{white} \textbf{-} & \hdots & \color{white} \textbf{-} & \color{white} \textbf{+} & \textbf{-} & \textbf{+} & \color{white} 0 & \color{white} 0 \\
     \color{white} 0 & \color{white} 0 & \ddots & \color{white} \textbf{-} & \color{white} \textbf{+} &\color{white}  \textbf{-} & \color{white} \textbf{+} & \color{white} \textbf{-} & \iddots & \color{white} 0 & \color{white} 0 & \color{white} 0 \\
     \vspace{-0.1cm}\textbf{+} & \textbf{-} & \color{white} 0 & \textbf{+} & \textbf{-} & \hdots & \textbf{-} & \textbf{+} & \color{white} 0 & \color{white} 0 & \color{white} 0 & \color{white} 0 \\
     \vspace{-0.1cm}\color{white}0 & \ddots & \ddots & \color{white} 0 & \textbf{+} & \textbf{-} & \textbf{+} & \color{white} 0 & \color{white} 0 & \color{white} 0 & \color{white} 0 & \color{white} 0 \\
     \color{white} 0 & \color{white} 0 & \ddots & \textbf{-} & \color{white} 0 & \textbf{+} & \color{white} 0 & \color{white} 0 & \color{white} 0 & \color{white} 0 & \color{white} 0 & \color{white} 0 \\
     \color{white} 0 & \color{white} 0 & \color{white} 0 & \textbf{+} & \color{white} 0 & \color{white} 0 & \color{white} 0 & \color{white} 0 & \color{white} 0 & \color{white} 0 & \color{white} 0 & \color{white} 0}
\hspace{-0.12cm}\nearrow\hspace{-0.12cm}\SmallMatrix{
     \color{white} 0 & \color{white} 0 & \color{white} 0 & \color{white} 0 & \color{white} 0 & \color{white} 0 & \textbf{+} & \color{white} 0 & \color{white} 0 & \color{white} 0 & \color{white} 0 & \color{white} 0 \\
     \color{white} 0 & \color{white} 0 & \color{white} 0 & \color{white} 0 & \color{white} 0 & \textbf{+} & \textbf{-} & \textbf{+} & \color{white} 0 & \color{white} 0 & \color{white} 0 & \color{white} 0 \\
     \vspace{-0.1cm}\color{white} 0 & \color{white}0 & \color{white} 0 & \color{white} 0 & \textbf{+} & \textbf{-} & \hdots & \textbf{-} & \textbf{+} & \color{white} 0 & \color{white} 0 & \color{white} 0 \\
     \color{white} 0 & \color{white} 0 & \color{white} 0 & \iddots & \color{white} \textbf{-} & \color{white} \textbf{+} & \color{white} \textbf{-} & \color{white} \textbf{+} & \color{white} \textbf{-} & \ddots & \color{white} 0 & \color{white} 0 \\
     \color{white} 0 & \color{white} 0 & \color{white} \textbf{+} & \color{white} \textbf{-} & \color{white} \textbf{+} & \color{white} \textbf{-} & \color{white} \hdots &\color{white}  \textbf{-} &\color{white}  \textbf{+} & \color{white} \textbf{-} & \color{white} \textbf{+} & \color{white} 0 \\
     \color{white} 0 & \textbf{+} & \textbf{-} & \color{white} \textbf{+} & \color{white} \textbf{-} & \color{white} \textbf{+} & \hdots & \color{white} \textbf{+} &\color{white}  \textbf{-} &\color{white}  \textbf{+} & \textbf{-} & \textbf{+} \\
     \textbf{+} & \textbf{-} & \color{white} \textbf{+} & \color{white} \textbf{-} & \color{white} \textbf{+} & \hdots & \color{white} \textbf{+} &\color{white}  \textbf{-} &\color{white}  \textbf{+} & \textbf{-} & \textbf{+}& \color{white} 0 \\
     \vspace{-0.1cm}\color{white}0 & \color{white} \textbf{+} &\color{white}  \textbf{-} & \color{white} \textbf{+} &\color{white}  \textbf{-} & \color{white} \hdots &\color{white}  \textbf{-} & \color{white} \textbf{+} & \color{white} \textbf{-} & \color{white} \textbf{+} & \color{white} 0 & \color{white} 0 \\
     \color{white} 0 & \color{white} 0 & \ddots &\color{white}  \textbf{-} & \color{white} \textbf{+} &\color{white}  \textbf{-} & \color{white} \textbf{+} & \color{white} \textbf{-} & \iddots & \color{white} 0 & \color{white} 0 & \color{white} 0 \\
     \color{white} 0 & \color{white} 0 & \color{white} 0 & \textbf{+} & \textbf{-} & \hdots & \textbf{-} & \textbf{+} & \color{white} 0 & \color{white} 0 & \color{white} 0 & \color{white} 0 \\
     \color{white} 0 & \color{white} 0 & \color{white} 0 & \color{white} 0 & \textbf{+} & \textbf{-} & \textbf{+} & \color{white} 0 & \color{white} 0 & \color{white} 0 & \color{white} 0 & \color{white} 0 \\
     \color{white} 0 & \color{white} 0 & \color{white} 0 & \color{white} 0 & \color{white} 0 & \textbf{+} & \color{white} 0 & \color{white} 0 & \color{white} 0 & \color{white} 0 & \color{white} 0 & \color{white} 0}
\]
\[
\hspace{-0.12cm}\nearrow\hspace{-0.12cm}\SmallMatrix{
     \vspace{-0.1cm}\color{white} 0 & \color{white}0 & \color{white} 0 & \textbf{+} & \color{white} 0 & \color{white} 0 & \color{white} 0 & \color{white} 0 & \color{white} 0 & \color{white} 0 & \color{white} 0 & \color{white} 0 \\
     \vspace{-0.1cm}\color{white} 0 & \color{white}0 & \iddots & \textbf{-} & \color{white} 0 & \color{white} 0 & \textbf{+} & \color{white} 0 & \color{white} 0 & \color{white} 0 & \color{white} 0 & \color{white} 0 \\
     \color{white} 0 & \iddots & \iddots & \color{white} 0 & \color{white} 0 & \textbf{+} & \textbf{-} & \textbf{+} & \color{white} 0 & \color{white} 0 & \color{white} 0 & \color{white} 0 \\
     \vspace{-0.1cm}\textbf{+} & \textbf{-} & \color{white} 0 & \color{white} 0 & \textbf{+} & \textbf{-} & \hdots & \textbf{-} & \textbf{+} & \color{white} 0 & \color{white} 0 & \color{white} 0 \\
     \color{white} 0 & \color{white} 0 & \color{white} 0 & \iddots & \color{white} \textbf{-} & \color{white} \textbf{+} & \color{white} \textbf{-} & \color{white} \textbf{+} & \color{white} \textbf{-} & \ddots & \color{white} 0 & \color{white} 0 \\
     \vspace{-0.1cm}\color{white} 0 & \color{white}0 & \textbf{+} & \textbf{-} & \color{white} \textbf{+} & \color{white} \textbf{-} & \hdots & \color{white} \textbf{-} & \color{white} \textbf{+} & \textbf{-} & \textbf{+} & \color{white} 0 \\
     \vspace{-0.1cm}\color{white}0 & \textbf{+} & \textbf{-} & \color{white} \textbf{+} & \color{white} \textbf{-} & \hdots & \color{white} \textbf{-} & \color{white} \textbf{+} & \textbf{-} & \textbf{+} & \color{white} 0 & \color{white} 0 \\
     \color{white} 0 & \color{white} 0 & \ddots & \color{white} \textbf{-} & \color{white} \textbf{+} &\color{white}  \textbf{-} & \color{white} \textbf{+} & \color{white} \textbf{-} & \iddots & \color{white} 0 & \color{white} 0 & \color{white} 0 \\
     \vspace{-0.1cm}\color{white} 0 & \color{white} 0 & \color{white} 0 & \textbf{+} & \textbf{-} & \hdots & \textbf{-} & \textbf{+} & \color{white} 0 & \color{white} 0 & \textbf{-} & \textbf{+} \\
     \vspace{-0.1cm}\color{white} 0 & \color{white} 0 & \color{white} 0 & \color{white} 0 & \textbf{+} & \textbf{-} & \textbf{+} & \color{white} 0 & \color{white} 0 & \iddots & \iddots & \color{white} 0 \\
     \color{white} 0 & \color{white} 0 & \color{white} 0 & \color{white} 0 & \color{white} 0 & \textbf{+} & \color{white} 0 & \color{white} 0 & \textbf{-} & \iddots & \color{white} 0 & \color{white} 0 \\
     \color{white} 0 & \color{white} 0 & \color{white} 0 & \color{white} 0 & \color{white} 0 & \color{white} 0 & \color{white} 0 & \color{white} 0 & \textbf{+} & \color{white} 0 & \color{white} 0 & \color{white} 0}
\hspace{-0.12cm}\nearrow\hspace{-0.12cm}\SmallMatrix{
     \color{white} 0 & \color{white} 0 & \color{white} 0 & \color{white} 0 & \color{white} 0 & \color{white} 0 & \color{white} 0 & \color{white} 0 & \color{white} 0 & \color{white} 0 & \textbf{+} & \color{white} 0 \\
     \color{white} 0 & \color{white} 0 & \color{white} 0 & \color{white} 0 & \color{white} 0 & \color{white} 0 & \color{white} 0 & \color{white} 0 & \color{white} 0 & \color{white} 0 & \color{white} 0 & \textbf{+} \\
     \color{white} 0 & \color{white} 0 & \color{white} 0 & \color{white} 0 & \color{white} 0 & \color{white} 0 & \textbf{+} & \color{white} 0 & \color{white} 0 & \color{white} 0 & \color{white} 0 & \color{white} 0 \\
     \vspace{-0.1cm}\color{white} 0 & \color{white}0 & \color{white} 0 & \color{white} 0 & \color{white} 0 & \textbf{+} & \textbf{-} & \textbf{+} & \color{white} 0 & \color{white} 0 & \color{white} 0 & \color{white} 0 \\
     \color{white} 0 & \color{white} 0 & \color{white} 0 & \color{white} 0 & \iddots & \color{white} \textbf{-} & \color{white} \textbf{+} & \color{white} \textbf{-} & \ddots & \color{white} 0 & \color{white} 0 & \color{white} 0 \\
     \vspace{-0.1cm}\color{white} 0 & \color{white}0 & \color{white} 0 & \textbf{+} & \textbf{-} & \color{white} \textbf{+} & \hdots & \color{white} \textbf{+} & \textbf{-} & \textbf{+} & \color{white} 0 & \color{white} 0 \\
     \vspace{-0.1cm}\color{white}0 & \color{white} 0 & \textbf{+} & \textbf{-} & \color{white} \textbf{+} & \hdots & \color{white} \textbf{+} & \textbf{-} & \textbf{+} & \color{white} 0 & \color{white} 0 & \color{white} 0 \\
     \color{white} 0 & \color{white} 0 & \color{white} 0 & \ddots & \color{white} \textbf{-} & \color{white} \textbf{+} & \color{white} \textbf{-} & \iddots & \color{white} 0 & \color{white} 0 & \color{white} 0 & \color{white} 0 \\
     \color{white} 0 & \color{white} 0 & \color{white} 0 & \color{white} 0 & \textbf{+} & \textbf{-} & \textbf{+} & \color{white} 0 & \color{white} 0 & \color{white} 0 & \color{white} 0 & \color{white} 0 \\
     \color{white} 0 & \color{white} 0 & \color{white} 0 & \color{white} 0 & \color{white} 0 & \textbf{+} & \color{white} 0 & \color{white} 0 & \color{white} 0 & \color{white} 0 & \color{white} 0 & \color{white} 0 \\
     \textbf{+} & \color{white} 0 & \color{white} 0 & \color{white} 0 & \color{white} 0 & \color{white} 0 & \color{white} 0 & \color{white} 0 & \color{white} 0 & \color{white} 0 & \color{white} 0 & \color{white} 0 \\
     \color{white} 0 & \textbf{+} & \color{white} 0 & \color{white} 0 & \color{white} 0 & \color{white} 0 & \color{white} 0 & \color{white} 0 & \color{white} 0 & \color{white} 0 & \color{white} 0 & \color{white} 0}
\hspace{-0.12cm}\nearrow
\hdots
\nearrow\hspace{-0.12cm}\SmallMatrix{
     \vspace{-0.1cm}\color{white}0 & \color{white} 0 & \color{white} 0 & \color{white} 0 & \color{white} 0 & \color{white} 0 & \color{white} 0 & \color{white} 0 & \color{white} 0 & \textbf{+} & \color{white} 0 & \color{white} 0 \\
     \color{white} 0 & \color{white} 0 & \color{white} 0 & \color{white} 0 & \color{white} 0 & \color{white} 0 & \color{white} 0 & \color{white} 0 & \color{white} 0 & \color{white} 0 & \ddots & \color{white} 0 \\
     \color{white} 0 & \color{white} 0 & \color{white} 0 & \color{white} 0 & \color{white} 0 & \color{white} 0 & \color{white} 0 & \color{white} 0 & \color{white} 0 & \color{white} 0 & \color{white} 0 & \textbf{+} \\
     \color{white} 0 & \color{white} 0 & \color{white} 0 & \color{white} 0 & \color{white} 0 & \color{white} 0 & \textbf{+} & \color{white} 0 & \color{white} 0 & \color{white} 0 & \color{white} 0 & \color{white} 0 \\
     \color{white} 0 & \color{white} 0 & \color{white} 0 & \color{white} 0 & \color{white} 0 & \textbf{+} & \textbf{-} & \textbf{+} & \color{white} 0 & \color{white} 0 & \color{white} 0 & \color{white} 0 \\
     \color{white} 0 & \color{white} 0 & \color{white} 0 & \color{white} 0 & \textbf{+} & \textbf{-} & \textbf{+} & \textbf{-} & \textbf{+} & \color{white} 0 & \color{white} 0 & \color{white} 0 \\
     \color{white} 0 & \color{white} 0 & \color{white} 0 & \textbf{+} & \textbf{-} & \textbf{+} & \textbf{-} & \textbf{+} & \color{white} 0 & \color{white} 0 & \color{white} 0 & \color{white} 0 \\
     \color{white} 0 & \color{white} 0 & \color{white} 0 & \color{white} 0 & \textbf{+} & \textbf{-} & \textbf{+} & \color{white} 0 & \color{white} 0 & \color{white} 0 & \color{white} 0 & \color{white} 0 \\
     \color{white} 0 & \color{white} 0 & \color{white} 0 & \color{white} 0 & \color{white} 0 & \textbf{+} & \color{white} 0 & \color{white} 0 & \color{white} 0 & \color{white} 0 & \color{white} 0 & \color{white} 0 \\
     \vspace{-0.1cm}\textbf{+} & \color{white}0 & \color{white} 0 & \color{white} 0 & \color{white} 0 & \color{white} 0 & \color{white} 0 & \color{white} 0 & \color{white} 0 & \color{white} 0 & \color{white} 0 & \color{white} 0 \\
     \color{white} 0 & \ddots & \color{white} 0 & \color{white} 0 & \color{white} 0 & \color{white} 0 & \color{white} 0 & \color{white} 0 & \color{white} 0 & \color{white} 0 & \color{white} 0 & \color{white} 0 \\
     \color{white} 0 & \color{white} 0 & \textbf{+} & \color{white} 0 & \color{white} 0 & \color{white} 0 & \color{white} 0 & \color{white} 0 & \color{white} 0 & \color{white} 0 & \color{white} 0 & \color{white} 0}
\hspace{-0.12cm}\nearrow\hspace{-0.12cm}\SmallMatrix{
     \color{white} 0 & \color{white} 0 & \color{white} 0 & \color{white} 0 & \color{white} 0 & \color{white} 0 & \color{white} 0 & \color{white} 0 & \color{white} 0 & \color{white} 0 & \color{white} 0 & \textbf{+} \\
     \vspace{-0.1cm}\color{white}0 & \color{white} 0 & \color{white} 0 & \color{white} 0 & \color{white} 0 & \color{white} 0 & \color{white} 0 & \color{white} 0 & \textbf{+} & \color{white} 0 & \color{white} 0 & \color{white} 0 \\
     \color{white} 0 & \color{white} 0 & \color{white} 0 & \color{white} 0 & \color{white} 0 & \color{white} 0 & \color{white} 0 & \color{white} 0 & \color{white} 0 & \ddots & \color{white} 0 & \color{white} 0 \\
     \color{white} 0 & \color{white} 0 & \color{white} 0 & \color{white} 0 & \color{white} 0 & \color{white} 0 & \color{white} 0 & \color{white} 0 & \color{white} 0 & \color{white} 0 & \textbf{+} & \color{white} 0 \\
     \color{white} 0 & \color{white} 0 & \color{white} 0 & \color{white} 0 & \color{white} 0 & \color{white} 0 & \textbf{+} & \color{white} 0 & \color{white} 0 & \color{white} 0 & \color{white} 0 & \color{white} 0 \\
     \color{white} 0 & \color{white} 0 & \color{white} 0 & \color{white} 0 & \color{white} 0 & \textbf{+} & \textbf{-} & \textbf{+} & \color{white} 0 & \color{white} 0 & \color{white} 0 & \color{white} 0 \\
     \color{white} 0 & \color{white} 0 & \color{white} 0 & \color{white} 0 & \textbf{+} & \textbf{-} & \textbf{+} & \color{white} 0 & \color{white} 0 & \color{white} 0 & \color{white} 0 & \color{white} 0 \\
     \color{white} 0 & \color{white} 0 & \color{white} 0 & \color{white} 0 & \color{white} 0 & \textbf{+} & \color{white} 0 & \color{white} 0 & \color{white} 0 & \color{white} 0 & \color{white} 0 & \color{white} 0 \\
     \vspace{-0.1cm}\color{white}0 & \textbf{+} & \color{white} 0 & \color{white} 0 & \color{white} 0 & \color{white} 0 & \color{white} 0 & \color{white} 0 & \color{white} 0 & \color{white} 0 & \color{white} 0 & \color{white} 0 \\
     \color{white} 0 & \color{white} 0 & \ddots & \color{white} 0 & \color{white} 0 & \color{white} 0 & \color{white} 0 & \color{white} 0 & \color{white} 0 & \color{white} 0 & \color{white} 0 & \color{white} 0 \\
     \color{white} 0 & \color{white} 0 & \color{white} 0 & \textbf{+} & \color{white} 0 & \color{white} 0 & \color{white} 0 & \color{white} 0 & \color{white} 0 & \color{white} 0 & \color{white} 0 & \color{white} 0 \\
     \textbf{+} & \color{white} 0 & \color{white} 0 & \color{white} 0 & \color{white} 0 & \color{white} 0 & \color{white} 0 & \color{white} 0 & \color{white} 0 & \color{white} 0 & \color{white} 0 & \color{white} 0}
\hspace{-0.12cm}\nearrow\hspace{-0.12cm}\SmallMatrix{
     \color{white} 0 & \color{white} 0 & \color{white} 0 & \color{white} 0 & \textbf{+} & \color{white} 0 & \color{white} 0 & \color{white} 0 & \color{white} 0 & \color{white} 0 & \color{white} 0 & \color{white} 0 \\
     \vspace{-0.1cm}\color{white} 0 & \color{white}0 & \color{white} 0 & \textbf{+} & \color{white} 0 & \color{white} 0 & \color{white} 0 & \color{white} 0 & \color{white} 0 & \color{white} 0 & \color{white} 0 & \color{white} 0 \\
     \color{white} 0 & \color{white} 0 & \iddots & \color{white} 0 & \color{white} 0 & \color{white} 0 & \color{white} 0 & \color{white} 0 & \color{white} 0 & \color{white} 0 & \color{white} 0 & \color{white} 0 \\
     \color{white} 0 & \textbf{+} & \color{white} 0 & \color{white} 0 & \color{white} 0 & \color{white} 0 & \color{white} 0 & \color{white} 0 & \color{white} 0 & \color{white} 0 & \color{white} 0 & \color{white} 0 \\
     \textbf{+} & \color{white} 0 & \color{white} 0 & \color{white} 0 & \color{white} 0 & \color{white} 0 & \color{white} 0 & \color{white} 0 & \color{white} 0 & \color{white} 0 & \color{white} 0 & \color{white} 0 \\
     \color{white} 0 & \color{white} 0 & \color{white} 0 & \color{white} 0 & \color{white} 0 & \color{white} 0 & \textbf{+} & \color{white} 0 & \color{white} 0 & \color{white} 0 & \color{white} 0 & \color{white} 0 \\
     \color{white} 0 & \color{white} 0 & \color{white} 0 & \color{white} 0 & \color{white} 0 & \textbf{+} & \color{white} 0 & \color{white} 0 & \color{white} 0 & \color{white} 0 & \color{white} 0 & \color{white} 0 \\
     \color{white} 0 & \color{white} 0 & \color{white} 0 & \color{white} 0 & \color{white} 0 & \color{white} 0 & \color{white} 0 & \color{white} 0 & \color{white} 0 & \color{white} 0 & \color{white} 0 & \textbf{+} \\
     \vspace{-0.1cm}\color{white}0 & \color{white} 0 & \color{white} 0 & \color{white} 0 & \color{white} 0 & \color{white} 0 & \color{white} 0 & \color{white} 0 & \color{white} 0 & \color{white} 0 & \textbf{+} & \color{white} 0 \\
     \color{white} 0 & \color{white} 0 & \color{white} 0 & \color{white} 0 & \color{white} 0 & \color{white} 0 & \color{white} 0 & \color{white} 0 & \color{white} 0 & \iddots & \color{white} 0 & \color{white} 0 \\
     \color{white} 0 & \color{white} 0 & \color{white} 0 & \color{white} 0 & \color{white} 0 & \color{white} 0 & \color{white} 0 & \color{white} 0 & \textbf{+} & \color{white} 0 & \color{white} 0 & \color{white} 0 \\
     \color{white} 0 & \color{white} 0 & \color{white} 0 & \color{white} 0 & \color{white} 0 & \color{white} 0 & \color{white} 0 & \textbf{+} & \color{white} 0 & \color{white} 0 &\color{white} 0 &  \color{white} 0}
\hspace{-0.12cm}\nearrow\hspace{-0.12cm}\SmallMatrix{
     \color{white} 0 & \color{white} 0 & \color{white} 0 & \color{white} 0 & \color{white} 0 & \textbf{+} & \color{white} 0 & \color{white} 0 & \color{white} 0 & \color{white} 0 & \color{white} 0 & \color{white} 0 \\
     \color{white} 0 & \color{white} 0 & \color{white} 0 & \color{white} 0 & \textbf{+} & \color{white} 0 & \color{white} 0 & \color{white} 0 & \color{white} 0 & \color{white} 0 & \color{white} 0 & \color{white} 0 \\
     \color{white} 0 & \color{white} 0 & \color{white} 0 & \textbf{+} & \color{white} 0 & \color{white} 0 & \color{white} 0 & \color{white} 0 & \color{white} 0 & \color{white} 0 & \color{white} 0 & \color{white} 0 \\
     \color{white} 0 & \color{white} 0 & \textbf{+} & \color{white} 0 & \color{white} 0 & \color{white} 0 & \color{white} 0 & \color{white} 0 & \color{white} 0 & \color{white} 0 & \color{white} 0 & \color{white} 0 \\
     \color{white} 0 & \textbf{+} & \color{white} 0 & \color{white} 0 & \color{white} 0 & \color{white} 0 & \color{white} 0 & \color{white} 0 & \color{white} 0 & \color{white} 0 & \color{white} 0 & \color{white} 0 \\
     \textbf{+} & \color{white} 0 & \color{white} 0 & \color{white} 0 & \color{white} 0 & \color{white} 0 & \color{white} 0 & \color{white} 0 & \color{white} 0 & \color{white} 0 & \color{white} 0 & \color{white} 0 \\
     \color{white} 0 & \color{white} 0 & \color{white} 0 & \color{white} 0 & \color{white} 0 & \color{white} 0 & \color{white} 0 & \color{white} 0 & \color{white} 0 & \color{white} 0 & \color{white} 0 & \textbf{+} \\
     \color{white} 0 & \color{white} 0 & \color{white} 0 & \color{white} 0 & \color{white} 0 & \color{white} 0 & \color{white} 0 & \color{white} 0 & \color{white} 0 & \color{white} 0 & \textbf{+} & \color{white} 0 \\
     \color{white} 0 & \color{white} 0 & \color{white} 0 & \color{white} 0 & \color{white} 0 & \color{white} 0 & \color{white} 0 & \color{white} 0 & \color{white} 0 & \textbf{+} & \color{white} 0 & \color{white} 0 \\
     \color{white} 0 & \color{white} 0 & \color{white} 0 & \color{white} 0 & \color{white} 0 & \color{white} 0 & \color{white} 0 & \color{white} 0 & \textbf{+} & \color{white} 0 & \color{white} 0 & \color{white} 0 \\
     \color{white} 0 & \color{white} 0 & \color{white} 0 & \color{white} 0 & \color{white} 0 & \color{white} 0 & \color{white} 0 & \textbf{+} & \color{white} 0 & \color{white} 0 & \color{white} 0 & \color{white} 0 \\
     \color{white} 0 & \color{white} 0 & \color{white} 0 & \color{white} 0 & \color{white} 0 & \color{white} 0 & \textbf{+} & \color{white} 0 & \color{white} 0 & \color{white} 0 & \color{white} 0 & \color{white} 0}
\]

It can be easily seen that each plane of this hypermatrix is an ASM. All vertical lines of $A$ corresponding to diamond positions of $L$ clearly have the alternating property. All vertical lines corresponding to the diagonal from $(2,m+2)$ to $(p-2,n-1)$, the diagonal from $(m+2,2)$ to $(n-1,p-2)$, the anti-diagonal from $(2,p-2)$ to $(p-2,2)$, and the anti-diagonal from $(m+2,n-1)$ to $(n-1, m+2)$ have exactly three non-zero entries, which alternate $+, - , +$. All other vertical lines contain exactly one non-zero entry, namely one $+$ entry. Therefore this is an ASHM.

\*

The $p$ entries occur in the following positions of the corresponding ASHL.
\[\Matrix{
    \color{white} 0 & \color{white} 0 & \color{white} 0 & \color{white} 0 & \color{white} 0 & \color{white} 0\color{white} 0 & p & \color{white} 0 & \color{white} 0 & \color{white} 0 & \color{white} 0 & \color{white} 00 \\
    \vspace{-0.1cm}\color{white} 0 & \color{white} 0 & \color{white} 0 & p & \color{white} 0 & p & p & p & p & \color{white} 0 & \color{white} 0 & \color{white} 0 \\
    \vspace{-0.1cm}\color{white} 0 & \color{white} 0 & \iddots & \color{white} 0 & \iddots &  p & \vdots & p & \ddots & \ddots & \color{white} 0 & \color{white} 0 \\
    \color{white} 0& p & \color{white} 0 & \iddots  & \iddots &\color{white}  p &\color{white}  p & \color{white} p &\ddots & \ddots & p & \color{white} 0 \\
    \color{white} 0\color{white} 0 & \color{white} 0 & p &  p &\color{white}  p & \color{white} p &\color{white}  p &\color{white}  p &\color{white}  p &  p & p & \color{white} 0 \\
    \color{white} 0 & p & p & \hdots & \color{white} p & \color{white} p & \color{white} p & \color{white} p & \color{white} p & \hdots & p & p \\
    p & p & \hdots & \color{white} p & \color{white} p & \color{white} p & \color{white} p & \color{white} p & \hdots & p & p & \color{white} 0  \\
    \vspace{-0.1cm}\color{white} 0 & p &  p & \color{white} p & \color{white} p & \color{white} p & \color{white} p & \color{white} p & p & p & \color{white} 0 & \color{white} 0 \\
    \vspace{-0.1cm}\color{white} 0 & p & \ddots & \ddots & \color{white} p & \color{white} p & \color{white} p & \iddots & \iddots & \color{white} 0 & p & \color{white} 0 \\
    \color{white} 0 & \color{white} 0 & \ddots & \ddots & p & \vdots & p & \iddots &\color{white} 0 &  \iddots & \color{white} 0 & \color{white} 0 \\
    \color{white} 0 & \color{white} 0 & \color{white} 0 & p & p & p & p & \color{white} 0 & p & \color{white} 0 & \color{white} 0 & \color{white} 0 \\
    \color{white} 00 & \color{white} 0 & \color{white} 0 & \color{white} 0 & \color{white} 0 & p & \color{white} 0\color{white} 0 & \color{white} 0 & \color{white} 0 & \color{white} 0 & \color{white} 0 & \color{white} 0}
\]

As outlined above, $p$ occurs as an entry in the diamond positions of $L$, and also occurs as every entry in the diagonal from $L_{2,m+2}$ to $L_{p-2,n-1}$, the diagonal from $L_{m+2,2}$ to $L_{n-1,p-2}$, the anti-diagonal from $L_{2,p-2}$ to $L_{p-2,2}$, and the anti-diagonal from $L_{m+2,n-1}$ to $L_{n-1, m+2}$.

\*

Therefore $p$ occurs as an entry of $L$ a total of $\frac{n^2+4n-20}{2}$ times:
\[2(1 + 3 + \dots + n-1) + 4\Big(\frac{n+1}{2}-3\Big) = 2\Big(\frac{n}{2}\Big)^2 + (2n -10) = \frac{n^2+4n-20}{2}\]
\end{proof}

Note that this bound is not tight. This is currently our best general construction, but we have constructed specific examples narrowly exceeding this bound.

\begin{example}
In the $n=11$ case, the construction outlined in the proof of Theorem \ref{max_entries} gives an $n \times n \times n$ ASHM $A$ with ASHL $L(A)$ containing the same entry $73$ times. The ASHM $B$, with $L(B)$ containing the same entry $77$ times, is obtained by the addition of T-blocks to $A$ as follows.
\[B = A + T_{2,1,3:\,3,2,4} + T_{9,10,3:\,10,11,4} - T_{2,2,4:\,11,10,8} + T_{9,1,8:\,10,2,9} + T_{2,10,8:\,3,11,9}\]
Explicitly, $B$ is the following ASHM.
\[B = \SmallMatrix{
     \color{white} 0 & \color{white} 0 & \color{white} 0 & \color{white} 0 & \color{white} 0 & \color{white} 0 & \textbf{+} & \color{white} 0 & \color{white} 0 & \color{white} 0 & \color{white} 0 \\
     \color{white} 0 & \color{white} 0 & \color{white} 0 & \color{white} 0 & \color{white} 0 & \color{white} 0 & \color{white} 0 & \textbf{+} & \color{white} 0 & \color{white} 0 & \color{white} 0 \\
     \color{white} 0 & \color{white} 0 & \color{white} 0 & \color{white} 0 & \color{white} 0 & \color{white} 0 & \color{white} 0 & \color{white} 0 & \textbf{+} & \color{white} 0 & \color{white} 0 \\
     \color{white} 0 & \color{white} 0 & \color{white} 0 & \color{white} 0 & \color{white} 0 & \color{white} 0 & \color{white} 0 & \color{white} 0 & \color{white} 0 & \textbf{+} & \color{white} 0 \\
     \color{white} 0 & \color{white} 0 & \color{white} 0 & \color{white} 0 & \color{white} 0 & \color{white} 0 & \color{white} 0 & \color{white} 0 & \color{white} 0 & \color{white} 0 & \textbf{+} \\
     \color{white} 0 & \color{white} 0 & \color{white} 0 & \color{white} 0 & \color{white} 0 & \textbf{+} & \color{white} 0 & \color{white} 0 & \color{white} 0 & \color{white} 0 & \color{white} 0 \\
     \textbf{+} & \color{white} 0 & \color{white} 0 & \color{white} 0 & \color{white} 0 & \color{white} 0 & \color{white} 0 & \color{white} 0 & \color{white} 0 & \color{white} 0 & \color{white} 0 \\
     \color{white} 0 & \textbf{+} & \color{white} 0 & \color{white} 0 & \color{white} 0 & \color{white} 0 & \color{white} 0 & \color{white} 0 & \color{white} 0 & \color{white} 0 & \color{white} 0 \\
     \color{white} 0 & \color{white} 0 & \textbf{+} & \color{white} 0 & \color{white} 0 & \color{white} 0 & \color{white} 0 & \color{white} 0 & \color{white} 0 & \color{white} 0 & \color{white} 0 \\
     \color{white} 0 & \color{white} 0 & \color{white} 0 & \textbf{+} & \color{white} 0 & \color{white} 0 & \color{white} 0 & \color{white} 0 & \color{white} 0 & \color{white} 0 & \color{white} 0 \\
     \color{white} 0 & \color{white} 0 & \color{white} 0 & \color{white} 0 & \textbf{+} & \color{white} 0 & \color{white} 0 & \color{white} 0 & \color{white} 0 & \color{white} 0 & \color{white} 0}
\hspace{-0.12cm}\nearrow\hspace{-0.12cm}\SmallMatrix{
     \textbf{+} & \color{white} 0 & \color{white} 0 & \color{white} 0 & \color{white} 0 & \color{white} 0 & \color{white} 0 & \color{white} 0 & \color{white} 0 & \color{white} 0 & \color{white} 0 \\
     \color{white} 0 & \color{white} 0 & \color{white} 0 & \textbf{+} & \color{white} 0 & \color{white} 0 & \color{white} 0 & \color{white} 0 & \color{white} 0 & \color{white} 0 & \color{white} 0 \\
     \color{white} 0 & \color{white} 0 & \textbf{+} & \color{white} 0 & \color{white} 0 & \color{white} 0 & \color{white} 0 & \color{white} 0 & \color{white} 0 & \color{white} 0 & \color{white} 0 \\
     \color{white} 0 & \textbf{+} & \color{white} 0 & \color{white} 0 & \color{white} 0 & \color{white} 0 & \color{white} 0 & \color{white} 0 & \color{white} 0 & \color{white} 0 & \color{white} 0 \\
     \color{white} 0 & \color{white} 0 & \color{white} 0 & \color{white} 0 & \color{white} 0 & \textbf{+} & \color{white} 0 & \color{white} 0 & \color{white} 0 & \color{white} 0 & \color{white} 0 \\
     \color{white} 0 & \color{white} 0 & \color{white} 0 & \color{white} 0 & \textbf{+} & \textbf{-} & \textbf{+} & \color{white} 0 & \color{white} 0 & \color{white} 0 & \color{white} 0 \\
     \color{white} 0 & \color{white} 0 & \color{white} 0 & \color{white} 0 & \color{white} 0 & \textbf{+} & \color{white} 0 & \color{white} 0 & \color{white} 0 & \color{white} 0 & \color{white} 0 \\
     \color{white} 0 & \color{white} 0 & \color{white} 0 & \color{white} 0 & \color{white} 0 & \color{white} 0 & \color{white} 0 & \color{white} 0 & \color{white} 0 & \textbf{+} & \color{white} 0 \\
     \color{white} 0 & \color{white} 0 & \color{white} 0 & \color{white} 0 & \color{white} 0 & \color{white} 0 & \color{white} 0 & \color{white} 0 & \textbf{+} & \color{white} 0 & \color{white} 0 \\
     \color{white} 0 & \color{white} 0 & \color{white} 0 & \color{white} 0 & \color{white} 0 & \color{white} 0 & \color{white} 0 & \textbf{+} & \color{white} 0 & \color{white} 0 & \color{white} 0 \\
     \color{white} 0 & \color{white} 0 & \color{white} 0 & \color{white} 0 & \color{white} 0 & \color{white} 0 & \color{white} 0 & \color{white} 0 & \color{white} 0 & \color{white} 0 & \textbf{+}}
\hspace{-0.12cm}\nearrow\hspace{-0.12cm}\SmallMatrix{
     \color{white} 0 & \color{white} 0 & \textbf{+} & \color{white} 0 & \color{white} 0 & \color{white} 0 & \color{white} 0 & \color{white} 0 & \color{white} 0 & \color{white} 0 & \color{white} 0 \\
     \textbf{+} & \color{white} 0 & \color{white} 0 & \color{white} 0 & \color{white} 0 & \color{white} 0 & \color{white} 0 & \color{white} 0 & \color{white} 0 & \color{white} 0 & \color{white} 0 \\
     \color{white} 0 & \textbf{+} & \color{white} 0 & \color{white} 0 & \color{white} 0 & \color{white} 0 & \color{white} 0 & \color{white} 0 & \color{white} 0 & \color{white} 0 & \color{white} 0 \\
     \color{white} 0 & \color{white} 0 & \color{white} 0 & \color{white} 0 & \color{white} 0 & \textbf{+} & \color{white} 0 & \color{white} 0 & \color{white} 0 & \color{white} 0 & \color{white} 0 \\
     \color{white} 0 & \color{white} 0 & \color{white} 0 & \color{white} 0 & \textbf{+} & \textbf{-} & \textbf{+} & \color{white} 0 & \color{white} 0 & \color{white} 0 & \color{white} 0 \\
     \color{white} 0 & \color{white} 0 & \color{white} 0 & \textbf{+} & \textbf{-} & \textbf{+} & \textbf{-} & \textbf{+} & \color{white} 0 & \color{white} 0 & \color{white} 0 \\
     \color{white} 0 & \color{white} 0 & \color{white} 0 & \color{white} 0 & \textbf{+} & \textbf{-} & \textbf{+} & \color{white} 0 & \color{white} 0 & \color{white} 0 & \color{white} 0 \\
     \color{white} 0 & \color{white} 0 & \color{white} 0 & \color{white} 0 & \color{white} 0 & \textbf{+} & \color{white} 0 & \color{white} 0 & \color{white} 0 & \color{white} 0 & \color{white} 0 \\
     \color{white} 0 & \color{white} 0 & \color{white} 0 & \color{white} 0 & \color{white} 0 & \color{white} 0 & \color{white} 0 & \color{white} 0 & \color{white} 0 & \textbf{+} & \color{white} 0 \\
     \color{white} 0 & \color{white} 0 & \color{white} 0 & \color{white} 0 & \color{white} 0 & \color{white} 0 & \color{white} 0 & \color{white} 0 & \color{white} 0 & \color{white} 0 & \textbf{+} \\
     \color{white} 0 & \color{white} 0 & \color{white} 0 & \color{white} 0 & \color{white} 0 & \color{white} 0 & \color{white} 0 & \color{white} 0 & \textbf{+} & \color{white} 0 & \color{white} 0}
\hspace{-0.12cm}\nearrow\hspace{-0.12cm}\SmallMatrix{
     \color{white} 0 & \textbf{+} & \color{white} 0 & \color{white} 0 & \color{white} 0 & \color{white} 0 & \color{white} 0 & \color{white} 0 & \color{white} 0 & \color{white} 0 & \color{white} 0 \\
     \color{white} 0 & \color{white} 0 & \color{white} 0 & \color{white} 0 & \color{white} 0 & \color{white} 0 & \color{white} 0 & \color{white} 0 & \color{white} 0 & \textbf{+} & \color{white} 0 \\
     \textbf{+} & \textbf{-} & \color{white} 0 & \color{white} 0 & \color{white} 0 & \textbf{+} & \color{white} 0 & \color{white} 0 & \color{white} 0 & \color{white} 0 & \color{white} 0 \\
     \color{white} 0 & \color{white} 0 & \color{white} 0 & \color{white} 0 & \textbf{+} & \textbf{-} & \textbf{+} & \color{white} 0 & \color{white} 0 & \color{white} 0 & \color{white} 0 \\
     \color{white} 0 & \color{white} 0 & \color{white} 0 & \textbf{+} & \textbf{-} & \textbf{+} & \textbf{-} & \textbf{+} & \color{white} 0 & \color{white} 0 & \color{white} 0 \\
     \color{white} 0 & \color{white} 0 & \textbf{+} & \textbf{-} & \textbf{+} & \textbf{-} & \textbf{+} & \textbf{-} & \textbf{+} & \color{white} 0 & \color{white} 0 \\
     \color{white} 0 & \color{white} 0 & \color{white} 0 & \textbf{+} & \textbf{-} & \textbf{+} & \textbf{-} & \textbf{+} & \color{white} 0 & \color{white} 0 & \color{white} 0 \\
     \color{white} 0 & \color{white} 0 & \color{white} 0 & \color{white} 0 & \textbf{+} & \textbf{-} & \textbf{+} & \color{white} 0 & \color{white} 0 & \color{white} 0 & \color{white} 0 \\
     \color{white} 0 & \color{white} 0 & \color{white} 0 & \color{white} 0 & \color{white} 0 & \textbf{+} & \color{white} 0 & \color{white} 0 & \color{white} 0 & \textbf{-} & \textbf{+} \\
     \color{white} 0 & \color{white} 0 & \color{white} 0 & \color{white} 0 & \color{white} 0 & \color{white} 0 & \color{white} 0 & \color{white} 0 & \color{white} 0 & \textbf{+} & \color{white} 0 \\
     \color{white} 0 & \textbf{+} & \color{white} 0 & \color{white} 0 & \color{white} 0 & \color{white} 0 & \color{white} 0 & \color{white} 0 & \color{white} 0 & \color{white} 0 & \color{white} 0}
\hspace{-0.12cm}\nearrow\hspace{-0.12cm}\SmallMatrix{
     \color{white} 0 & \color{white} 0 & \color{white} 0 & \color{white} 0 & \color{white} 0 & \color{white} 0 & \color{white} 0 & \textbf{+} & \color{white} 0 & \color{white} 0 & \color{white} 0 \\
     \color{white} 0 & \color{white} 0 & \color{white} 0 & \color{white} 0 & \color{white} 0 & \textbf{+} & \color{white} 0 & \textbf{-} & \textbf{+} & \color{white} 0 & \color{white} 0 \\
     \color{white} 0 & \color{white} 0 & \color{white} 0 & \color{white} 0 & \textbf{+} & \textbf{-} & \textbf{+} & \color{white} 0 & \textbf{-} & \textbf{+} & \color{white} 0 \\
     \color{white} 0 & \color{white} 0 & \color{white} 0 & \textbf{+} & \textbf{-} & \textbf{+} & \textbf{-} & \textbf{+} & \color{white} 0 & \textbf{-} & \textbf{+} \\
     \color{white} 0 & \color{white} 0 & \textbf{+} & \textbf{-} & \textbf{+} & \textbf{-} & \textbf{+} & \textbf{-} & \textbf{+} & \color{white} 0 & \color{white} 0 \\
     \color{white} 0 & \textbf{+} & \textbf{-} & \textbf{+} & \textbf{-} & \textbf{+} & \textbf{-} & \textbf{+} & \textbf{-} & \textbf{+} & \color{white} 0 \\
     \color{white} 0 & \color{white} 0 & \textbf{+} & \textbf{-} & \textbf{+} & \textbf{-} & \textbf{+} & \textbf{-} & \textbf{+} & \color{white} 0 & \color{white} 0 \\
     \textbf{+} & \textbf{-} & \color{white} 0 & \textbf{+} & \textbf{-} & \textbf{+} & \textbf{-} & \textbf{+} & \color{white} 0 & \color{white} 0 & \color{white} 0 \\
     \color{white} 0 & \textbf{+} & \textbf{-} & \color{white} 0 & \textbf{+} & \textbf{-} & \textbf{+} & \color{white} 0 & \color{white} 0 & \color{white} 0 & \color{white} 0 \\
     \color{white} 0 & \color{white} 0 & \textbf{+} & \textbf{-} & \color{white} 0 & \textbf{+} & \color{white} 0 & \color{white} 0 & \color{white} 0 & \color{white} 0 & \color{white} 0 \\
     \color{white} 0 & \color{white} 0 & \color{white} 0 & \textbf{+} & \color{white} 0 & \color{white} 0 & \color{white} 0 & \color{white} 0 & \color{white} 0 & \color{white} 0 & \color{white} 0}
\hspace{-0.12cm}\nearrow\hspace{-0.12cm}\SmallMatrix{
     \color{white} 0 & \color{white} 0 & \color{white} 0 & \color{white} 0 & \color{white} 0 & \textbf{+} & \color{white} 0 & \color{white} 0 & \color{white} 0 & \color{white} 0 & \color{white} 0 \\
     \color{white} 0 & \color{white} 0 & \color{white} 0 & \color{white} 0 & \textbf{+} & \textbf{-} & \textbf{+} & \color{white} 0 & \color{white} 0 & \color{white} 0 & \color{white} 0 \\
     \color{white} 0 & \color{white} 0 & \color{white} 0 & \textbf{+} & \textbf{-} & \textbf{+} & \textbf{-} & \textbf{+} & \color{white} 0 & \color{white} 0 & \color{white} 0 \\
     \color{white} 0 & \color{white} 0 & \textbf{+} & \textbf{-} & \textbf{+} & \textbf{-} & \textbf{+} & \textbf{-} & \textbf{+} & \color{white} 0 & \color{white} 0 \\
     \color{white} 0 & \textbf{+} & \textbf{-} & \textbf{+} & \textbf{-} & \textbf{+} & \textbf{-} & \textbf{+} & \textbf{-} & \textbf{+} & \color{white} 0 \\
     \textbf{+} & \textbf{-} & \textbf{+} & \textbf{-} & \textbf{+} & \textbf{-} & \textbf{+} & \textbf{-} & \textbf{+} & \textbf{-} & \textbf{+} \\
     \color{white} 0 & \textbf{+} & \textbf{-} & \textbf{+} & \textbf{-} & \textbf{+} & \textbf{-} & \textbf{+} & \textbf{-} & \textbf{+} & \color{white} 0 \\
     \color{white} 0 & \color{white} 0 & \textbf{+} & \textbf{-} & \textbf{+} & \textbf{-} & \textbf{+} & \textbf{-} & \textbf{+} & \color{white} 0 & \color{white} 0 \\
     \color{white} 0 & \color{white} 0 & \color{white} 0 & \textbf{+} & \textbf{-} & \textbf{+} & \textbf{-} & \textbf{+} & \color{white} 0 & \color{white} 0 & \color{white} 0 \\
     \color{white} 0 & \color{white} 0 & \color{white} 0 & \color{white} 0 & \textbf{+} & \textbf{-} & \textbf{+} & \color{white} 0 & \color{white} 0 & \color{white} 0 & \color{white} 0 \\
     \color{white} 0 & \color{white} 0 & \color{white} 0 & \color{white} 0 & \color{white} 0 & \textbf{+} & \color{white} 0 & \color{white} 0 & \color{white} 0 & \color{white} 0 & \color{white} 0}
\]
\[
\hspace{-0.12cm}\nearrow\hspace{-0.12cm}\SmallMatrix{
     \color{white} 0 & \color{white} 0 & \color{white} 0 & \textbf{+} & \color{white} 0 & \color{white} 0 & \color{white} 0 & \color{white} 0 & \color{white} 0 & \color{white} 0 & \color{white} 0 \\
     \color{white} 0 & \color{white} 0 & \textbf{+} & \textbf{-} & \color{white} 0 & \textbf{+} & \color{white} 0 & \color{white} 0 & \color{white} 0 & \color{white} 0 & \color{white} 0 \\
     \color{white} 0 & \textbf{+} & \textbf{-} & \color{white} 0 & \textbf{+} & \textbf{-} & \textbf{+} & \color{white} 0 & \color{white} 0 & \color{white} 0 & \color{white} 0 \\
     \textbf{+} & \textbf{-} & \color{white} 0 & \textbf{+} & \textbf{-} & \textbf{+} & \textbf{-} & \textbf{+} & \color{white} 0 & \color{white} 0 & \color{white} 0 \\
     \color{white} 0 & \color{white} 0 & \textbf{+} & \textbf{-} & \textbf{+} & \textbf{-} & \textbf{+} & \textbf{-} & \textbf{+} & \color{white} 0 & \color{white} 0 \\
     \color{white} 0 & \textbf{+} & \textbf{-} & \textbf{+} & \textbf{-} & \textbf{+} & \textbf{-} & \textbf{+} & \textbf{-} & \textbf{+} & \color{white} 0 \\
     \color{white} 0 & \color{white} 0 & \textbf{+} & \textbf{-} & \textbf{+} & \textbf{-} & \textbf{+} & \textbf{-} & \textbf{+} & \color{white} 0 & \color{white} 0 \\
     \color{white} 0 & \color{white} 0 & \color{white} 0 & \textbf{+} & \textbf{-} & \textbf{+} & \textbf{-} & \textbf{+} & \color{white} 0 & \textbf{-} & \textbf{+} \\
     \color{white} 0 & \color{white} 0 & \color{white} 0 & \color{white} 0 & \textbf{+} & \textbf{-} & \textbf{+} & \color{white} 0 & \textbf{-} & \textbf{+} & \color{white} 0 \\
     \color{white} 0 & \color{white} 0 & \color{white} 0 & \color{white} 0 & \color{white} 0 & \textbf{+} & \color{white} 0 & \textbf{-} & \textbf{+} & \color{white} 0 & \color{white} 0 \\
     \color{white} 0 & \color{white} 0 & \color{white} 0 & \color{white} 0 & \color{white} 0 & \color{white} 0 & \color{white} 0 & \textbf{+} & \color{white} 0 & \color{white} 0 & \color{white} 0}
\hspace{-0.12cm}\nearrow\hspace{-0.12cm}\SmallMatrix{
     \color{white} 0 & \color{white} 0 & \color{white} 0 & \color{white} 0 & \color{white} 0 & \color{white} 0 & \color{white} 0 & \color{white} 0 & \color{white} 0 & \textbf{+} & \color{white} 0 \\
     \color{white} 0 & \textbf{+} & \color{white} 0 & \color{white} 0 & \color{white} 0 & \color{white} 0 & \color{white} 0 & \color{white} 0 & \color{white} 0 & \color{white} 0 & \color{white} 0 \\
     \color{white} 0 & \color{white} 0 & \color{white} 0 & \color{white} 0 & \color{white} 0 & \textbf{+} & \color{white} 0 & \color{white} 0 & \color{white} 0 & \textbf{-} & \textbf{+} \\
     \color{white} 0 & \color{white} 0 & \color{white} 0 & \color{white} 0 & \textbf{+} & \textbf{-} & \textbf{+} & \color{white} 0 & \color{white} 0 & \color{white} 0 & \color{white} 0 \\
     \color{white} 0 & \color{white} 0 & \color{white} 0 & \textbf{+} & \textbf{-} & \textbf{+} & \textbf{-} & \textbf{+} & \color{white} 0 & \color{white} 0 & \color{white} 0 \\
     \color{white} 0 & \color{white} 0 & \textbf{+} & \textbf{-} & \textbf{+} & \textbf{-} & \textbf{+} & \textbf{-} & \textbf{+} & \color{white} 0 & \color{white} 0 \\
     \color{white} 0 & \color{white} 0 & \color{white} 0 & \textbf{+} & \textbf{-} & \textbf{+} & \textbf{-} & \textbf{+} & \color{white} 0 & \color{white} 0 & \color{white} 0 \\
     \color{white} 0 & \color{white} 0 & \color{white} 0 & \color{white} 0 & \textbf{+} & \textbf{-} & \textbf{+} & \color{white} 0 & \color{white} 0 & \color{white} 0 & \color{white} 0 \\
     \textbf{+} & \textbf{-} & \color{white} 0 & \color{white} 0 & \color{white} 0 & \textbf{+} & \color{white} 0 & \color{white} 0 & \color{white} 0 & \color{white} 0 & \color{white} 0 \\
     \color{white} 0 & \textbf{+} & \color{white} 0 & \color{white} 0 & \color{white} 0 & \color{white} 0 & \color{white} 0 & \color{white} 0 & \color{white} 0 & \color{white} 0 & \color{white} 0 \\
     \color{white} 0 & \color{white} 0 & \color{white} 0 & \color{white} 0 & \color{white} 0 & \color{white} 0 & \color{white} 0 & \color{white} 0 & \color{white} 0 & \textbf{+} & \color{white} 0}
\hspace{-0.12cm}\nearrow\hspace{-0.12cm}\SmallMatrix{
     \color{white} 0 & \color{white} 0 & \color{white} 0 & \color{white} 0 & \color{white} 0 & \color{white} 0 & \color{white} 0 & \color{white} 0 & \textbf{+} & \color{white} 0 & \color{white} 0 \\
     \color{white} 0 & \color{white} 0 & \color{white} 0 & \color{white} 0 & \color{white} 0 & \color{white} 0 & \color{white} 0 & \color{white} 0 & \color{white} 0 & \color{white} 0 & \textbf{+} \\
     \color{white} 0 & \color{white} 0 & \color{white} 0 & \color{white} 0 & \color{white} 0 & \color{white} 0 & \color{white} 0 & \color{white} 0 & \color{white} 0 & \textbf{+} & \color{white} 0 \\
     \color{white} 0 & \color{white} 0 & \color{white} 0 & \color{white} 0 & \color{white} 0 & \textbf{+} & \color{white} 0 & \color{white} 0 & \color{white} 0 & \color{white} 0 & \color{white} 0 \\
     \color{white} 0 & \color{white} 0 & \color{white} 0 & \color{white} 0 & \textbf{+} & \textbf{-} & \textbf{+} & \color{white} 0 & \color{white} 0 & \color{white} 0 & \color{white} 0 \\
     \color{white} 0 & \color{white} 0 & \color{white} 0 & \textbf{+} & \textbf{-} & \textbf{+} & \textbf{-} & \textbf{+} & \color{white} 0 & \color{white} 0 & \color{white} 0 \\
     \color{white} 0 & \color{white} 0 & \color{white} 0 & \color{white} 0 & \textbf{+} & \textbf{-} & \textbf{+} & \color{white} 0 & \color{white} 0 & \color{white} 0 & \color{white} 0 \\
     \color{white} 0 & \color{white} 0 & \color{white} 0 & \color{white} 0 & \color{white} 0 & \textbf{+} & \color{white} 0 & \color{white} 0 & \color{white} 0 & \color{white} 0 & \color{white} 0 \\
     \color{white} 0 & \textbf{+} & \color{white} 0 & \color{white} 0 & \color{white} 0 & \color{white} 0 & \color{white} 0 & \color{white} 0 & \color{white} 0 & \color{white} 0 & \color{white} 0 \\
     \textbf{+} & \color{white} 0 & \color{white} 0 & \color{white} 0 & \color{white} 0 & \color{white} 0 & \color{white} 0 & \color{white} 0 & \color{white} 0 & \color{white} 0 & \color{white} 0 \\
     \color{white} 0 & \color{white} 0 & \textbf{+} & \color{white} 0 & \color{white} 0 & \color{white} 0 & \color{white} 0 & \color{white} 0 & \color{white} 0 & \color{white} 0 & \color{white} 0}
\hspace{-0.12cm}\nearrow\hspace{-0.12cm}\SmallMatrix{
     \color{white} 0 & \color{white} 0 & \color{white} 0 & \color{white} 0 & \color{white} 0 & \color{white} 0 & \color{white} 0 & \color{white} 0 & \color{white} 0 & \color{white} 0 & \textbf{+} \\
     \color{white} 0 & \color{white} 0 & \color{white} 0 & \color{white} 0 & \color{white} 0 & \color{white} 0 & \color{white} 0 & \textbf{+} & \color{white} 0 & \color{white} 0 & \color{white} 0 \\
     \color{white} 0 & \color{white} 0 & \color{white} 0 & \color{white} 0 & \color{white} 0 & \color{white} 0 & \color{white} 0 & \color{white} 0 & \textbf{+} & \color{white} 0 & \color{white} 0 \\
     \color{white} 0 & \color{white} 0 & \color{white} 0 & \color{white} 0 & \color{white} 0 & \color{white} 0 & \color{white} 0 & \color{white} 0 & \color{white} 0 & \textbf{+} & \color{white} 0 \\
     \color{white} 0 & \color{white} 0 & \color{white} 0 & \color{white} 0 & \color{white} 0 & \textbf{+} & \color{white} 0 & \color{white} 0 & \color{white} 0 & \color{white} 0 & \color{white} 0 \\
     \color{white} 0 & \color{white} 0 & \color{white} 0 & \color{white} 0 & \textbf{+} & \textbf{-} & \textbf{+} & \color{white} 0 & \color{white} 0 & \color{white} 0 & \color{white} 0 \\
     \color{white} 0 & \color{white} 0 & \color{white} 0 & \color{white} 0 & \color{white} 0 & \textbf{+} & \color{white} 0 & \color{white} 0 & \color{white} 0 & \color{white} 0 & \color{white} 0 \\
     \color{white} 0 & \textbf{+} & \color{white} 0 & \color{white} 0 & \color{white} 0 & \color{white} 0 & \color{white} 0 & \color{white} 0 & \color{white} 0 & \color{white} 0 & \color{white} 0 \\
     \color{white} 0 & \color{white} 0 & \textbf{+} & \color{white} 0 & \color{white} 0 & \color{white} 0 & \color{white} 0 & \color{white} 0 & \color{white} 0 & \color{white} 0 & \color{white} 0 \\
     \color{white} 0 & \color{white} 0 & \color{white} 0 & \textbf{+} & \color{white} 0 & \color{white} 0 & \color{white} 0 & \color{white} 0 & \color{white} 0 & \color{white} 0 & \color{white} 0 \\
     \textbf{+} & \color{white} 0 & \color{white} 0 & \color{white} 0 & \color{white} 0 & \color{white} 0 & \color{white} 0 & \color{white} 0 & \color{white} 0 & \color{white} 0 & \color{white} 0}
\hspace{-0.12cm}\nearrow\hspace{-0.12cm}\SmallMatrix{
     \color{white} 0 & \color{white} 0 & \color{white} 0 & \color{white} 0 & \textbf{+} & \color{white} 0 & \color{white} 0 & \color{white} 0 & \color{white} 0 & \color{white} 0 & \color{white} 0 \\
     \color{white} 0 & \color{white} 0 & \color{white} 0 & \textbf{+} & \color{white} 0 & \color{white} 0 & \color{white} 0 & \color{white} 0 & \color{white} 0 & \color{white} 0 & \color{white} 0 \\
     \color{white} 0 & \color{white} 0 & \textbf{+} & \color{white} 0 & \color{white} 0 & \color{white} 0 & \color{white} 0 & \color{white} 0 & \color{white} 0 & \color{white} 0 & \color{white} 0 \\
     \color{white} 0 & \textbf{+} & \color{white} 0 & \color{white} 0 & \color{white} 0 & \color{white} 0 & \color{white} 0 & \color{white} 0 & \color{white} 0 & \color{white} 0 & \color{white} 0 \\
     \textbf{+} & \color{white} 0 & \color{white} 0 & \color{white} 0 & \color{white} 0 & \color{white} 0 & \color{white} 0 & \color{white} 0 & \color{white} 0 & \color{white} 0 & \color{white} 0 \\
     \color{white} 0 & \color{white} 0 & \color{white} 0 & \color{white} 0 & \color{white} 0 & \textbf{+} & \color{white} 0 & \color{white} 0 & \color{white} 0 & \color{white} 0 & \color{white} 0 \\
     \color{white} 0 & \color{white} 0 & \color{white} 0 & \color{white} 0 & \color{white} 0 & \color{white} 0 & \color{white} 0 & \color{white} 0 & \color{white} 0 & \color{white} 0 & \textbf{+} \\
     \color{white} 0 & \color{white} 0 & \color{white} 0 & \color{white} 0 & \color{white} 0 & \color{white} 0 & \color{white} 0 & \color{white} 0 & \color{white} 0 & \textbf{+} & \color{white} 0 \\
     \color{white} 0 & \color{white} 0 & \color{white} 0 & \color{white} 0 & \color{white} 0 & \color{white} 0 & \color{white} 0 & \color{white} 0 & \textbf{+} & \color{white} 0 & \color{white} 0 \\
     \color{white} 0 & \color{white} 0 & \color{white} 0 & \color{white} 0 & \color{white} 0 & \color{white} 0 & \color{white} 0 & \textbf{+} & \color{white} 0 & \color{white} 0 & \color{white} 0 \\
     \color{white} 0 & \color{white} 0 & \color{white} 0 & \color{white} 0 & \color{white} 0 & \color{white} 0 & \textbf{+} & \color{white} 0 & \color{white} 0 & \color{white} 0 & \color{white} 0}
\]

Which has the following corresponding ASHL.
\[L(B) = 
\Matrix{
     2 & 4 & 3 & 7 & 11 & 6 & 1 & 5 & 9 & 8 & 10 \\
     3 & 8 & 7 & 6 & 6 & 6 & 6 & 6 & 5 & 4 & 9 \\
     4 & 6 & 6 & 6 & 6 & 6 & 6 & 6 & 6 & 6 & 8 \\
     7 & 6 & 6 & 6 & 6 & 6 & 6 & 6 & 6 & 6 & 5 \\
     11 & 6 & 6 & 6 & 6 & 6 & 6 & 6 & 6 & 6 & 1 \\
     6 & 6 & 6 & 6 & 6 & 6 & 6 & 6 & 6 & 6 & 6 \\
     1 & 6 & 6 & 6 & 6 & 6 & 6 & 6 & 6 & 6 & 11 \\
     5 & 6 & 6 & 6 & 6 & 6 & 6 & 6 & 6 & 6 & 7 \\
     8 & 6 & 6 & 6 & 6 & 6 & 6 & 6 & 6 & 6 & 4 \\
     9 & 8 & 5 & 6 & 6 & 6 & 6 & 6 & 7 & 4 & 3 \\
     10 & 4 & 9 & 5 & 1 & 6 & 11 & 7 & 3 & 8 & 2}\]
\end{example}

\begin{example}
In the $n=13$ case, the construction outlined in the proof of Theorem \ref{max_entries} gives an $n \times n \times n$ ASHM $A$ with ASHL $L(A)$ containing the same entry $101$ times. The following ASHM $B$ exceeds this, with $L(B)$ containing the same entry $103$ times.
\[B = 
\SmallMatrix{
     \color{white} 0 & \color{white} 0 & \color{white} 0 & \color{white} 0 & \color{white} 0 & \textbf{+} & \color{white} 0 & \color{white} 0 & \color{white} 0 & \color{white} 0 & \color{white} 0 & \color{white} 0 & \color{white} 0 \\
     \color{white} 0 & \color{white} 0 & \color{white} 0 & \color{white} 0 & \textbf{+} & \color{white} 0 & \color{white} 0 & \color{white} 0 & \color{white} 0 & \color{white} 0 & \color{white} 0 & \color{white} 0 & \color{white} 0 \\
     \color{white} 0 & \color{white} 0 & \color{white} 0 & \textbf{+} & \color{white} 0 & \color{white} 0 & \color{white} 0 & \color{white} 0 & \color{white} 0 & \color{white} 0 & \color{white} 0 & \color{white} 0 & \color{white} 0 \\
     \color{white} 0 & \color{white} 0 & \textbf{+} & \color{white} 0 & \color{white} 0 & \color{white} 0 & \color{white} 0 & \color{white} 0 & \color{white} 0 & \color{white} 0 & \color{white} 0 & \color{white} 0 & \color{white} 0 \\
     \color{white} 0 & \textbf{+} & \color{white} 0 & \color{white} 0 & \color{white} 0 & \color{white} 0 & \color{white} 0 & \color{white} 0 & \color{white} 0 & \color{white} 0 & \color{white} 0 & \color{white} 0 & \color{white} 0 \\
     \textbf{+} & \color{white} 0 & \color{white} 0 & \color{white} 0 & \color{white} 0 & \color{white} 0 & \color{white} 0 & \color{white} 0 & \color{white} 0 & \color{white} 0 & \color{white} 0 & \color{white} 0 & \color{white} 0 \\
     \color{white} 0 & \color{white} 0 & \color{white} 0 & \color{white} 0 & \color{white} 0 & \color{white} 0 & \textbf{+} & \color{white} 0 & \color{white} 0 & \color{white} 0 & \color{white} 0 & \color{white} 0 & \color{white} 0 \\
     \color{white} 0 & \color{white} 0 & \color{white} 0 & \color{white} 0 & \color{white} 0 & \color{white} 0 & \color{white} 0 & \color{white} 0 & \color{white} 0 & \color{white} 0 & \color{white} 0 & \color{white} 0 & \textbf{+} \\
     \color{white} 0 & \color{white} 0 & \color{white} 0 & \color{white} 0 & \color{white} 0 & \color{white} 0 & \color{white} 0 & \color{white} 0 & \color{white} 0 & \color{white} 0 & \color{white} 0 & \textbf{+} & \color{white} 0 \\
     \color{white} 0 & \color{white} 0 & \color{white} 0 & \color{white} 0 & \color{white} 0 & \color{white} 0 & \color{white} 0 & \color{white} 0 & \color{white} 0 & \color{white} 0 & \textbf{+} & \color{white} 0 & \color{white} 0 \\
     \color{white} 0 & \color{white} 0 & \color{white} 0 & \color{white} 0 & \color{white} 0 & \color{white} 0 & \color{white} 0 & \color{white} 0 & \color{white} 0 & \textbf{+} & \color{white} 0 & \color{white} 0 & \color{white} 0 \\
     \color{white} 0 & \color{white} 0 & \color{white} 0 & \color{white} 0 & \color{white} 0 & \color{white} 0 & \color{white} 0 & \color{white} 0 & \textbf{+} & \color{white} 0 & \color{white} 0 & \color{white} 0 & \color{white} 0 \\
     \color{white} 0 & \color{white} 0 & \color{white} 0 & \color{white} 0 & \color{white} 0 & \color{white} 0 & \color{white} 0 & \textbf{+} & \color{white} 0 & \color{white} 0 & \color{white} 0 & \color{white} 0 & \color{white} 0}
\hspace{-0.2cm}\nearrow\hspace{-0.2cm}\SmallMatrix{
     \color{white} 0 & \color{white} 0 & \color{white} 0 & \color{white} 0 & \color{white} 0 & \color{white} 0 & \color{white} 0 & \color{white} 0 & \color{white} 0 & \textbf{+} & \color{white} 0 & \color{white} 0 & \color{white} 0 \\
     \color{white} 0 & \color{white} 0 & \textbf{+} & \color{white} 0 & \color{white} 0 & \color{white} 0 & \color{white} 0 & \color{white} 0 & \color{white} 0 & \color{white} 0 & \color{white} 0 & \color{white} 0 & \color{white} 0 \\
     \color{white} 0 & \textbf{+} & \color{white} 0 & \color{white} 0 & \color{white} 0 & \color{white} 0 & \color{white} 0 & \color{white} 0 & \color{white} 0 & \color{white} 0 & \color{white} 0 & \color{white} 0 & \color{white} 0 \\
     \color{white} 0 & \color{white} 0 & \color{white} 0 & \color{white} 0 & \color{white} 0 & \color{white} 0 & \color{white} 0 & \color{white} 0 & \color{white} 0 & \color{white} 0 & \color{white} 0 & \color{white} 0 & \textbf{+} \\
     \color{white} 0 & \color{white} 0 & \color{white} 0 & \color{white} 0 & \color{white} 0 & \color{white} 0 & \color{white} 0 & \color{white} 0 & \textbf{+} & \color{white} 0 & \color{white} 0 & \color{white} 0 & \color{white} 0 \\
     \color{white} 0 & \color{white} 0 & \color{white} 0 & \color{white} 0 & \color{white} 0 & \color{white} 0 & \textbf{+} & \color{white} 0 & \color{white} 0 & \color{white} 0 & \color{white} 0 & \color{white} 0 & \color{white} 0 \\
     \color{white} 0 & \color{white} 0 & \color{white} 0 & \color{white} 0 & \color{white} 0 & \textbf{+} & \textbf{-} & \textbf{+} & \color{white} 0 & \color{white} 0 & \color{white} 0 & \color{white} 0 & \color{white} 0 \\
     \color{white} 0 & \color{white} 0 & \color{white} 0 & \color{white} 0 & \color{white} 0 & \color{white} 0 & \textbf{+} & \color{white} 0 & \color{white} 0 & \color{white} 0 & \color{white} 0 & \color{white} 0 & \color{white} 0 \\
     \color{white} 0 & \color{white} 0 & \color{white} 0 & \color{white} 0 & \textbf{+} & \color{white} 0 & \color{white} 0 & \color{white} 0 & \color{white} 0 & \color{white} 0 & \color{white} 0 & \color{white} 0 & \color{white} 0 \\
     \textbf{+} & \color{white} 0 & \color{white} 0 & \color{white} 0 & \color{white} 0 & \color{white} 0 & \color{white} 0 & \color{white} 0 & \color{white} 0 & \color{white} 0 & \color{white} 0 & \color{white} 0 & \color{white} 0 \\
     \color{white} 0 & \color{white} 0 & \color{white} 0 & \color{white} 0 & \color{white} 0 & \color{white} 0 & \color{white} 0 & \color{white} 0 & \color{white} 0 & \color{white} 0 & \color{white} 0 & \textbf{+} & \color{white} 0 \\
     \color{white} 0 & \color{white} 0 & \color{white} 0 & \color{white} 0 & \color{white} 0 & \color{white} 0 & \color{white} 0 & \color{white} 0 & \color{white} 0 & \color{white} 0 & \textbf{+} & \color{white} 0 & \color{white} 0 \\
     \color{white} 0 & \color{white} 0 & \color{white} 0 & \textbf{+} & \color{white} 0 & \color{white} 0 & \color{white} 0 & \color{white} 0 & \color{white} 0 & \color{white} 0 & \color{white} 0 & \color{white} 0 & \color{white} 0}
\hspace{-0.2cm}\nearrow\hspace{-0.2cm}\SmallMatrix{
     \color{white} 0 & \color{white} 0 & \color{white} 0 & \color{white} 0 & \color{white} 0 & \color{white} 0 & \color{white} 0 & \color{white} 0 & \color{white} 0 & \color{white} 0 & \textbf{+} & \color{white} 0 & \color{white} 0 \\
     \color{white} 0 & \color{white} 0 & \color{white} 0 & \color{white} 0 & \color{white} 0 & \color{white} 0 & \color{white} 0 & \color{white} 0 & \textbf{+} & \color{white} 0 & \color{white} 0 & \color{white} 0 & \color{white} 0 \\
     \color{white} 0 & \color{white} 0 & \color{white} 0 & \color{white} 0 & \color{white} 0 & \color{white} 0 & \color{white} 0 & \color{white} 0 & \color{white} 0 & \color{white} 0 & \color{white} 0 & \color{white} 0 & \textbf{+} \\
     \color{white} 0 & \color{white} 0 & \color{white} 0 & \color{white} 0 & \color{white} 0 & \color{white} 0 & \color{white} 0 & \color{white} 0 & \color{white} 0 & \textbf{+} & \color{white} 0 & \color{white} 0 & \color{white} 0 \\
     \color{white} 0 & \color{white} 0 & \color{white} 0 & \color{white} 0 & \color{white} 0 & \color{white} 0 & \textbf{+} & \color{white} 0 & \textbf{-} & \color{white} 0 & \color{white} 0 & \textbf{+} & \color{white} 0 \\
     \color{white} 0 & \color{white} 0 & \color{white} 0 & \color{white} 0 & \color{white} 0 & \textbf{+} & \textbf{-} & \textbf{+} & \color{white} 0 & \color{white} 0 & \color{white} 0 & \color{white} 0 & \color{white} 0 \\
     \color{white} 0 & \color{white} 0 & \color{white} 0 & \color{white} 0 & \textbf{+} & \textbf{-} & \textbf{+} & \textbf{-} & \textbf{+} & \color{white} 0 & \color{white} 0 & \color{white} 0 & \color{white} 0 \\
     \color{white} 0 & \color{white} 0 & \color{white} 0 & \color{white} 0 & \color{white} 0 & \textbf{+} & \textbf{-} & \textbf{+} & \color{white} 0 & \color{white} 0 & \color{white} 0 & \color{white} 0 & \color{white} 0 \\
     \color{white} 0 & \textbf{+} & \color{white} 0 & \color{white} 0 & \textbf{-} & \color{white} 0 & \textbf{+} & \color{white} 0 & \color{white} 0 & \color{white} 0 & \color{white} 0 & \color{white} 0 & \color{white} 0 \\
     \color{white} 0 & \color{white} 0 & \color{white} 0 & \textbf{+} & \color{white} 0 & \color{white} 0 & \color{white} 0 & \color{white} 0 & \color{white} 0 & \color{white} 0 & \color{white} 0 & \color{white} 0 & \color{white} 0 \\
     \textbf{+} & \color{white} 0 & \color{white} 0 & \color{white} 0 & \color{white} 0 & \color{white} 0 & \color{white} 0 & \color{white} 0 & \color{white} 0 & \color{white} 0 & \color{white} 0 & \color{white} 0 & \color{white} 0 \\
     \color{white} 0 & \color{white} 0 & \color{white} 0 & \color{white} 0 & \textbf{+} & \color{white} 0 & \color{white} 0 & \color{white} 0 & \color{white} 0 & \color{white} 0 & \color{white} 0 & \color{white} 0 & \color{white} 0 \\
     \color{white} 0 & \color{white} 0 & \textbf{+} & \color{white} 0 & \color{white} 0 & \color{white} 0 & \color{white} 0 & \color{white} 0 & \color{white} 0 & \color{white} 0 & \color{white} 0 & \color{white} 0 & \color{white} 0}
\hspace{-0.2cm}\nearrow\hspace{-0.2cm}\SmallMatrix{
     \textbf{+} & \color{white} 0 & \color{white} 0 & \color{white} 0 & \color{white} 0 & \color{white} 0 & \color{white} 0 & \color{white} 0 & \color{white} 0 & \color{white} 0 & \color{white} 0 & \color{white} 0 & \color{white} 0 \\
     \color{white} 0 & \color{white} 0 & \color{white} 0 & \color{white} 0 & \color{white} 0 & \color{white} 0 & \color{white} 0 & \color{white} 0 & \color{white} 0 & \textbf{+} & \color{white} 0 & \color{white} 0 & \color{white} 0 \\
     \color{white} 0 & \color{white} 0 & \color{white} 0 & \color{white} 0 & \color{white} 0 & \color{white} 0 & \color{white} 0 & \color{white} 0 & \color{white} 0 & \color{white} 0 & \textbf{+} & \color{white} 0 & \color{white} 0 \\
     \color{white} 0 & \color{white} 0 & \color{white} 0 & \color{white} 0 & \color{white} 0 & \color{white} 0 & \textbf{+} & \color{white} 0 & \color{white} 0 & \textbf{-} & \color{white} 0 & \textbf{+} & \color{white} 0 \\
     \color{white} 0 & \color{white} 0 & \color{white} 0 & \color{white} 0 & \color{white} 0 & \textbf{+} & \textbf{-} & \textbf{+} & \color{white} 0 & \color{white} 0 & \color{white} 0 & \color{white} 0 & \color{white} 0 \\
     \color{white} 0 & \color{white} 0 & \color{white} 0 & \color{white} 0 & \textbf{+} & \textbf{-} & \textbf{+} & \textbf{-} & \textbf{+} & \color{white} 0 & \color{white} 0 & \color{white} 0 & \color{white} 0 \\
     \color{white} 0 & \color{white} 0 & \color{white} 0 & \textbf{+} & \textbf{-} & \textbf{+} & \textbf{-} & \textbf{+} & \textbf{-} & \textbf{+} & \color{white} 0 & \color{white} 0 & \color{white} 0 \\
     \color{white} 0 & \color{white} 0 & \color{white} 0 & \color{white} 0 & \textbf{+} & \textbf{-} & \textbf{+} & \textbf{-} & \textbf{+} & \color{white} 0 & \color{white} 0 & \color{white} 0 & \color{white} 0 \\
     \color{white} 0 & \color{white} 0 & \color{white} 0 & \color{white} 0 & \color{white} 0 & \textbf{+} & \textbf{-} & \textbf{+} & \color{white} 0 & \color{white} 0 & \color{white} 0 & \color{white} 0 & \color{white} 0 \\
     \color{white} 0 & \textbf{+} & \color{white} 0 & \textbf{-} & \color{white} 0 & \color{white} 0 & \textbf{+} & \color{white} 0 & \color{white} 0 & \color{white} 0 & \color{white} 0 & \color{white} 0 & \color{white} 0 \\
     \color{white} 0 & \color{white} 0 & \textbf{+} & \color{white} 0 & \color{white} 0 & \color{white} 0 & \color{white} 0 & \color{white} 0 & \color{white} 0 & \color{white} 0 & \color{white} 0 & \color{white} 0 & \color{white} 0 \\
     \color{white} 0 & \color{white} 0 & \color{white} 0 & \textbf{+} & \color{white} 0 & \color{white} 0 & \color{white} 0 & \color{white} 0 & \color{white} 0 & \color{white} 0 & \color{white} 0 & \color{white} 0 & \color{white} 0 \\
     \color{white} 0 & \color{white} 0 & \color{white} 0 & \color{white} 0 & \color{white} 0 & \color{white} 0 & \color{white} 0 & \color{white} 0 & \color{white} 0 & \color{white} 0 & \color{white} 0 & \color{white} 0 & \textbf{+}}
\hspace{-0.2cm}\nearrow\hspace{-0.2cm}\SmallMatrix{
     \color{white} 0 & \color{white} 0 & \textbf{+} & \color{white} 0 & \color{white} 0 & \color{white} 0 & \color{white} 0 & \color{white} 0 & \color{white} 0 & \color{white} 0 & \color{white} 0 & \color{white} 0 & \color{white} 0 \\
     \color{white} 0 & \textbf{+} & \textbf{-} & \color{white} 0 & \color{white} 0 & \color{white} 0 & \color{white} 0 & \color{white} 0 & \color{white} 0 & \color{white} 0 & \color{white} 0 & \textbf{+} & \color{white} 0 \\
     \textbf{+} & \textbf{-} & \color{white} 0 & \color{white} 0 & \color{white} 0 & \color{white} 0 & \textbf{+} & \color{white} 0 & \color{white} 0 & \color{white} 0 & \color{white} 0 & \color{white} 0 & \color{white} 0 \\
     \color{white} 0 & \color{white} 0 & \color{white} 0 & \color{white} 0 & \color{white} 0 & \textbf{+} & \textbf{-} & \textbf{+} & \color{white} 0 & \color{white} 0 & \color{white} 0 & \color{white} 0 & \color{white} 0 \\
     \color{white} 0 & \color{white} 0 & \color{white} 0 & \color{white} 0 & \textbf{+} & \textbf{-} & \textbf{+} & \textbf{-} & \textbf{+} & \color{white} 0 & \color{white} 0 & \color{white} 0 & \color{white} 0 \\
     \color{white} 0 & \color{white} 0 & \color{white} 0 & \textbf{+} & \textbf{-} & \textbf{+} & \textbf{-} & \textbf{+} & \textbf{-} & \textbf{+} & \color{white} 0 & \color{white} 0 & \color{white} 0 \\
     \color{white} 0 & \color{white} 0 & \textbf{+} & \textbf{-} & \textbf{+} & \textbf{-} & \textbf{+} & \textbf{-} & \textbf{+} & \textbf{-} & \textbf{+} & \color{white} 0 & \color{white} 0 \\
     \color{white} 0 & \color{white} 0 & \color{white} 0 & \textbf{+} & \textbf{-} & \textbf{+} & \textbf{-} & \textbf{+} & \textbf{-} & \textbf{+} & \color{white} 0 & \color{white} 0 & \color{white} 0 \\
     \color{white} 0 & \color{white} 0 & \color{white} 0 & \color{white} 0 & \textbf{+} & \textbf{-} & \textbf{+} & \textbf{-} & \textbf{+} & \color{white} 0 & \color{white} 0 & \color{white} 0 & \color{white} 0 \\
     \color{white} 0 & \color{white} 0 & \color{white} 0 & \color{white} 0 & \color{white} 0 & \textbf{+} & \textbf{-} & \textbf{+} & \color{white} 0 & \color{white} 0 & \color{white} 0 & \color{white} 0 & \color{white} 0 \\
     \color{white} 0 & \color{white} 0 & \color{white} 0 & \color{white} 0 & \color{white} 0 & \color{white} 0 & \textbf{+} & \color{white} 0 & \color{white} 0 & \color{white} 0 & \color{white} 0 & \textbf{-} & \textbf{+} \\
     \color{white} 0 & \textbf{+} & \color{white} 0 & \color{white} 0 & \color{white} 0 & \color{white} 0 & \color{white} 0 & \color{white} 0 & \color{white} 0 & \color{white} 0 & \textbf{-} & \textbf{+} & \color{white} 0 \\
     \color{white} 0 & \color{white} 0 & \color{white} 0 & \color{white} 0 & \color{white} 0 & \color{white} 0 & \color{white} 0 & \color{white} 0 & \color{white} 0 & \color{white} 0 & \textbf{+} & \color{white} 0 & \color{white} 0}
\hspace{-0.2cm}\nearrow\hspace{-0.2cm}\SmallMatrix{
     \color{white} 0 & \color{white} 0 & \color{white} 0 & \color{white} 0 & \textbf{+} & \color{white} 0 & \color{white} 0 & \color{white} 0 & \color{white} 0 & \color{white} 0 & \color{white} 0 & \color{white} 0 & \color{white} 0 \\
     \color{white} 0 & \color{white} 0 & \color{white} 0 & \textbf{+} & \textbf{-} & \color{white} 0 & \textbf{+} & \color{white} 0 & \color{white} 0 & \color{white} 0 & \color{white} 0 & \color{white} 0 & \color{white} 0 \\
     \color{white} 0 & \color{white} 0 & \textbf{+} & \textbf{-} & \color{white} 0 & \textbf{+} & \textbf{-} & \textbf{+} & \color{white} 0 & \color{white} 0 & \color{white} 0 & \color{white} 0 & \color{white} 0 \\
     \color{white} 0 & \textbf{+} & \textbf{-} & \color{white} 0 & \textbf{+} & \textbf{-} & \textbf{+} & \textbf{-} & \textbf{+} & \color{white} 0 & \color{white} 0 & \color{white} 0 & \color{white} 0 \\
     \textbf{+} & \textbf{-} & \color{white} 0 & \textbf{+} & \textbf{-} & \textbf{+} & \textbf{-} & \textbf{+} & \textbf{-} & \textbf{+} & \color{white} 0 & \color{white} 0 & \color{white} 0 \\
     \color{white} 0 & \color{white} 0 & \textbf{+} & \textbf{-} & \textbf{+} & \textbf{-} & \textbf{+} & \textbf{-} & \textbf{+} & \textbf{-} & \textbf{+} & \color{white} 0 & \color{white} 0 \\
     \color{white} 0 & \textbf{+} & \textbf{-} & \textbf{+} & \textbf{-} & \textbf{+} & \textbf{-} & \textbf{+} & \textbf{-} & \textbf{+} & \textbf{-} & \textbf{+} & \color{white} 0 \\
     \color{white} 0 & \color{white} 0 & \textbf{+} & \textbf{-} & \textbf{+} & \textbf{-} & \textbf{+} & \textbf{-} & \textbf{+} & \textbf{-} & \textbf{+} & \color{white} 0 & \color{white} 0 \\
     \color{white} 0 & \color{white} 0 & \color{white} 0 & \textbf{+} & \textbf{-} & \textbf{+} & \textbf{-} & \textbf{+} & \textbf{-} & \textbf{+} & \color{white} 0 & \textbf{-} & \textbf{+} \\
     \color{white} 0 & \color{white} 0 & \color{white} 0 & \color{white} 0 & \textbf{+} & \textbf{-} & \textbf{+} & \textbf{-} & \textbf{+} & \color{white} 0 & \textbf{-} & \textbf{+} & \color{white} 0 \\
     \color{white} 0 & \color{white} 0 & \color{white} 0 & \color{white} 0 & \color{white} 0 & \textbf{+} & \textbf{-} & \textbf{+} & \color{white} 0 & \textbf{-} & \textbf{+} & \color{white} 0 & \color{white} 0 \\
     \color{white} 0 & \color{white} 0 & \color{white} 0 & \color{white} 0 & \color{white} 0 & \color{white} 0 & \textbf{+} & \color{white} 0 & \textbf{-} & \textbf{+} & \color{white} 0 & \color{white} 0 & \color{white} 0 \\
     \color{white} 0 & \color{white} 0 & \color{white} 0 & \color{white} 0 & \color{white} 0 & \color{white} 0 & \color{white} 0 & \color{white} 0 & \textbf{+} & \color{white} 0 & \color{white} 0 & \color{white} 0 & \color{white} 0}
\hspace{-0.2cm}\nearrow\]
\[\hspace{-0.2cm}\SmallMatrix{
     \color{white} 0 & \color{white} 0 & \color{white} 0 & \color{white} 0 & \color{white} 0 & \color{white} 0 & \textbf{+} & \color{white} 0 & \color{white} 0 & \color{white} 0 & \color{white} 0 & \color{white} 0 & \color{white} 0 \\
     \color{white} 0 & \color{white} 0 & \color{white} 0 & \color{white} 0 & \color{white} 0 & \textbf{+} & \textbf{-} & \textbf{+} & \color{white} 0 & \color{white} 0 & \color{white} 0 & \color{white} 0 & \color{white} 0 \\
     \color{white} 0 & \color{white} 0 & \color{white} 0 & \color{white} 0 & \textbf{+} & \textbf{-} & \textbf{+} & \textbf{-} & \textbf{+} & \color{white} 0 & \color{white} 0 & \color{white} 0 & \color{white} 0 \\
     \color{white} 0 & \color{white} 0 & \color{white} 0 & \textbf{+} & \textbf{-} & \textbf{+} & \textbf{-} & \textbf{+} & \textbf{-} & \textbf{+} & \color{white} 0 & \color{white} 0 & \color{white} 0 \\
     \color{white} 0 & \color{white} 0 & \textbf{+} & \textbf{-} & \textbf{+} & \textbf{-} & \textbf{+} & \textbf{-} & \textbf{+} & \textbf{-} & \textbf{+} & \color{white} 0 & \color{white} 0 \\
     \color{white} 0 & \textbf{+} & \textbf{-} & \textbf{+} & \textbf{-} & \textbf{+} & \textbf{-} & \textbf{+} & \textbf{-} & \textbf{+} & \textbf{-} & \textbf{+} & \color{white} 0 \\
     \textbf{+} & \textbf{-} & \textbf{+} & \textbf{-} & \textbf{+} & \textbf{-} & \textbf{+} & \textbf{-} & \textbf{+} & \textbf{-} & \textbf{+} & \textbf{-} & \textbf{+} \\
     \color{white} 0 & \textbf{+} & \textbf{-} & \textbf{+} & \textbf{-} & \textbf{+} & \textbf{-} & \textbf{+} & \textbf{-} & \textbf{+} & \textbf{-} & \textbf{+} & \color{white} 0 \\
     \color{white} 0 & \color{white} 0 & \textbf{+} & \textbf{-} & \textbf{+} & \textbf{-} & \textbf{+} & \textbf{-} & \textbf{+} & \textbf{-} & \textbf{+} & \color{white} 0 & \color{white} 0 \\
     \color{white} 0 & \color{white} 0 & \color{white} 0 & \textbf{+} & \textbf{-} & \textbf{+} & \textbf{-} & \textbf{+} & \textbf{-} & \textbf{+} & \color{white} 0 & \color{white} 0 & \color{white} 0 \\
     \color{white} 0 & \color{white} 0 & \color{white} 0 & \color{white} 0 & \textbf{+} & \textbf{-} & \textbf{+} & \textbf{-} & \textbf{+} & \color{white} 0 & \color{white} 0 & \color{white} 0 & \color{white} 0 \\
     \color{white} 0 & \color{white} 0 & \color{white} 0 & \color{white} 0 & \color{white} 0 & \textbf{+} & \textbf{-} & \textbf{+} & \color{white} 0 & \color{white} 0 & \color{white} 0 & \color{white} 0 & \color{white} 0 \\
     \color{white} 0 & \color{white} 0 & \color{white} 0 & \color{white} 0 & \color{white} 0 & \color{white} 0 & \textbf{+} & \color{white} 0 & \color{white} 0 & \color{white} 0 & \color{white} 0 & \color{white} 0 & \color{white} 0}
\nearrow\hspace{-0.2cm}\SmallMatrix{
     \color{white} 0 & \color{white} 0 & \color{white} 0 & \color{white} 0 & \color{white} 0 & \color{white} 0 & \color{white} 0 & \color{white} 0 & \color{white} 0 & \color{white} 0 & \color{white} 0 & \color{white} 0 & \textbf{+} \\
     \color{white} 0 & \color{white} 0 & \color{white} 0 & \color{white} 0 & \color{white} 0 & \color{white} 0 & \textbf{+} & \color{white} 0 & \color{white} 0 & \color{white} 0 & \color{white} 0 & \color{white} 0 & \color{white} 0 \\
     \color{white} 0 & \color{white} 0 & \color{white} 0 & \color{white} 0 & \color{white} 0 & \textbf{+} & \textbf{-} & \textbf{+} & \color{white} 0 & \color{white} 0 & \color{white} 0 & \color{white} 0 & \color{white} 0 \\
     \color{white} 0 & \color{white} 0 & \color{white} 0 & \color{white} 0 & \textbf{+} & \textbf{-} & \textbf{+} & \textbf{-} & \textbf{+} & \color{white} 0 & \color{white} 0 & \color{white} 0 & \color{white} 0 \\
     \color{white} 0 & \color{white} 0 & \color{white} 0 & \textbf{+} & \textbf{-} & \textbf{+} & \textbf{-} & \textbf{+} & \textbf{-} & \textbf{+} & \color{white} 0 & \color{white} 0 & \color{white} 0 \\
     \color{white} 0 & \color{white} 0 & \textbf{+} & \textbf{-} & \textbf{+} & \textbf{-} & \textbf{+} & \textbf{-} & \textbf{+} & \textbf{-} & \textbf{+} & \color{white} 0 & \color{white} 0 \\
     \color{white} 0 & \textbf{+} & \textbf{-} & \textbf{+} & \textbf{-} & \textbf{+} & \textbf{-} & \textbf{+} & \textbf{-} & \textbf{+} & \textbf{-} & \textbf{+} & \color{white} 0 \\
     \color{white} 0 & \color{white} 0 & \textbf{+} & \textbf{-} & \textbf{+} & \textbf{-} & \textbf{+} & \textbf{-} & \textbf{+} & \textbf{-} & \textbf{+} & \color{white} 0 & \color{white} 0 \\
     \color{white} 0 & \color{white} 0 & \color{white} 0 & \textbf{+} & \textbf{-} & \textbf{+} & \textbf{-} & \textbf{+} & \textbf{-} & \textbf{+} & \color{white} 0 & \color{white} 0 & \color{white} 0 \\
     \color{white} 0 & \color{white} 0 & \color{white} 0 & \color{white} 0 & \textbf{+} & \textbf{-} & \textbf{+} & \textbf{-} & \textbf{+} & \color{white} 0 & \color{white} 0 & \color{white} 0 & \color{white} 0 \\
     \color{white} 0 & \color{white} 0 & \color{white} 0 & \color{white} 0 & \color{white} 0 & \textbf{+} & \textbf{-} & \textbf{+} & \color{white} 0 & \color{white} 0 & \color{white} 0 & \color{white} 0 & \color{white} 0 \\
     \color{white} 0 & \color{white} 0 & \color{white} 0 & \color{white} 0 & \color{white} 0 & \color{white} 0 & \textbf{+} & \color{white} 0 & \color{white} 0 & \color{white} 0 & \color{white} 0 & \color{white} 0 & \color{white} 0 \\
     \textbf{+} & \color{white} 0 & \color{white} 0 & \color{white} 0 & \color{white} 0 & \color{white} 0 & \color{white} 0 & \color{white} 0 & \color{white} 0 & \color{white} 0 & \color{white} 0 & \color{white} 0 & \color{white} 0}
\hspace{-0.2cm}\nearrow\hspace{-0.2cm}\SmallMatrix{
     \color{white} 0 & \color{white} 0 & \color{white} 0 & \color{white} 0 & \color{white} 0 & \color{white} 0 & \color{white} 0 & \color{white} 0 & \color{white} 0 & \color{white} 0 & \color{white} 0 & \textbf{+} & \color{white} 0 \\
     \color{white} 0 & \color{white} 0 & \color{white} 0 & \color{white} 0 & \color{white} 0 & \color{white} 0 & \color{white} 0 & \color{white} 0 & \color{white} 0 & \color{white} 0 & \color{white} 0 & \color{white} 0 & \textbf{+} \\
     \color{white} 0 & \color{white} 0 & \color{white} 0 & \color{white} 0 & \color{white} 0 & \color{white} 0 & \textbf{+} & \color{white} 0 & \color{white} 0 & \color{white} 0 & \color{white} 0 & \color{white} 0 & \color{white} 0 \\
     \color{white} 0 & \color{white} 0 & \color{white} 0 & \color{white} 0 & \color{white} 0 & \textbf{+} & \textbf{-} & \textbf{+} & \color{white} 0 & \color{white} 0 & \color{white} 0 & \color{white} 0 & \color{white} 0 \\
     \color{white} 0 & \color{white} 0 & \color{white} 0 & \color{white} 0 & \textbf{+} & \textbf{-} & \textbf{+} & \textbf{-} & \textbf{+} & \color{white} 0 & \color{white} 0 & \color{white} 0 & \color{white} 0 \\
     \color{white} 0 & \color{white} 0 & \color{white} 0 & \textbf{+} & \textbf{-} & \textbf{+} & \textbf{-} & \textbf{+} & \textbf{-} & \textbf{+} & \color{white} 0 & \color{white} 0 & \color{white} 0 \\
     \color{white} 0 & \color{white} 0 & \textbf{+} & \textbf{-} & \textbf{+} & \textbf{-} & \textbf{+} & \textbf{-} & \textbf{+} & \textbf{-} & \textbf{+} & \color{white} 0 & \color{white} 0 \\
     \color{white} 0 & \color{white} 0 & \color{white} 0 & \textbf{+} & \textbf{-} & \textbf{+} & \textbf{-} & \textbf{+} & \textbf{-} & \textbf{+} & \color{white} 0 & \color{white} 0 & \color{white} 0 \\
     \color{white} 0 & \color{white} 0 & \color{white} 0 & \color{white} 0 & \textbf{+} & \textbf{-} & \textbf{+} & \textbf{-} & \textbf{+} & \color{white} 0 & \color{white} 0 & \color{white} 0 & \color{white} 0 \\
     \color{white} 0 & \color{white} 0 & \color{white} 0 & \color{white} 0 & \color{white} 0 & \textbf{+} & \textbf{-} & \textbf{+} & \color{white} 0 & \color{white} 0 & \color{white} 0 & \color{white} 0 & \color{white} 0 \\
     \color{white} 0 & \color{white} 0 & \color{white} 0 & \color{white} 0 & \color{white} 0 & \color{white} 0 & \textbf{+} & \color{white} 0 & \color{white} 0 & \color{white} 0 & \color{white} 0 & \color{white} 0 & \color{white} 0 \\
     \textbf{+} & \color{white} 0 & \color{white} 0 & \color{white} 0 & \color{white} 0 & \color{white} 0 & \color{white} 0 & \color{white} 0 & \color{white} 0 & \color{white} 0 & \color{white} 0 & \color{white} 0 & \color{white} 0 \\
     \color{white} 0 & \textbf{+} & \color{white} 0 & \color{white} 0 & \color{white} 0 & \color{white} 0 & \color{white} 0 & \color{white} 0 & \color{white} 0 & \color{white} 0 & \color{white} 0 & \color{white} 0 & \color{white} 0}
\hspace{-0.2cm}\nearrow\hspace{-0.2cm}\SmallMatrix{
     \color{white} 0 & \color{white} 0 & \color{white} 0 & \textbf{+} & \color{white} 0 & \color{white} 0 & \color{white} 0 & \color{white} 0 & \color{white} 0 & \color{white} 0 & \color{white} 0 & \color{white} 0 & \color{white} 0 \\
     \color{white} 0 & \color{white} 0 & \textbf{+} & \textbf{-} & \color{white} 0 & \color{white} 0 & \color{white} 0 & \color{white} 0 & \color{white} 0 & \color{white} 0 & \textbf{+} & \color{white} 0 & \color{white} 0 \\
     \color{white} 0 & \textbf{+} & \textbf{-} & \color{white} 0 & \color{white} 0 & \color{white} 0 & \color{white} 0 & \color{white} 0 & \color{white} 0 & \color{white} 0 & \color{white} 0 & \textbf{+} & \color{white} 0 \\
     \textbf{+} & \textbf{-} & \color{white} 0 & \color{white} 0 & \color{white} 0 & \color{white} 0 & \textbf{+} & \color{white} 0 & \color{white} 0 & \color{white} 0 & \color{white} 0 & \color{white} 0 & \color{white} 0 \\
     \color{white} 0 & \color{white} 0 & \color{white} 0 & \color{white} 0 & \color{white} 0 & \textbf{+} & \textbf{-} & \textbf{+} & \color{white} 0 & \color{white} 0 & \color{white} 0 & \color{white} 0 & \color{white} 0 \\
     \color{white} 0 & \color{white} 0 & \color{white} 0 & \color{white} 0 & \textbf{+} & \textbf{-} & \textbf{+} & \textbf{-} & \textbf{+} & \color{white} 0 & \color{white} 0 & \color{white} 0 & \color{white} 0 \\
     \color{white} 0 & \color{white} 0 & \color{white} 0 & \textbf{+} & \textbf{-} & \textbf{+} & \textbf{-} & \textbf{+} & \textbf{-} & \textbf{+} & \color{white} 0 & \color{white} 0 & \color{white} 0 \\
     \color{white} 0 & \color{white} 0 & \color{white} 0 & \color{white} 0 & \textbf{+} & \textbf{-} & \textbf{+} & \textbf{-} & \textbf{+} & \color{white} 0 & \color{white} 0 & \color{white} 0 & \color{white} 0 \\
     \color{white} 0 & \color{white} 0 & \color{white} 0 & \color{white} 0 & \color{white} 0 & \textbf{+} & \textbf{-} & \textbf{+} & \color{white} 0 & \color{white} 0 & \color{white} 0 & \color{white} 0 & \color{white} 0 \\
     \color{white} 0 & \color{white} 0 & \color{white} 0 & \color{white} 0 & \color{white} 0 & \color{white} 0 & \textbf{+} & \color{white} 0 & \color{white} 0 & \color{white} 0 & \color{white} 0 & \textbf{-} & \textbf{+} \\
     \color{white} 0 & \textbf{+} & \color{white} 0 & \color{white} 0 & \color{white} 0 & \color{white} 0 & \color{white} 0 & \color{white} 0 & \color{white} 0 & \color{white} 0 & \textbf{-} & \textbf{+} & \color{white} 0 \\
     \color{white} 0 & \color{white} 0 & \textbf{+} & \color{white} 0 & \color{white} 0 & \color{white} 0 & \color{white} 0 & \color{white} 0 & \color{white} 0 & \textbf{-} & \textbf{+} & \color{white} 0 & \color{white} 0 \\
     \color{white} 0 & \color{white} 0 & \color{white} 0 & \color{white} 0 & \color{white} 0 & \color{white} 0 & \color{white} 0 & \color{white} 0 & \color{white} 0 & \textbf{+} & \color{white} 0 & \color{white} 0 & \color{white} 0}
\hspace{-0.2cm}\nearrow\hspace{-0.2cm}\SmallMatrix{
     \color{white} 0 & \textbf{+} & \color{white} 0 & \color{white} 0 & \color{white} 0 & \color{white} 0 & \color{white} 0 & \color{white} 0 & \color{white} 0 & \color{white} 0 & \color{white} 0 & \color{white} 0 & \color{white} 0 \\
     \textbf{+} & \textbf{-} & \color{white} 0 & \textbf{+} & \color{white} 0 & \color{white} 0 & \color{white} 0 & \color{white} 0 & \color{white} 0 & \color{white} 0 & \color{white} 0 & \color{white} 0 & \color{white} 0 \\
     \color{white} 0 & \color{white} 0 & \textbf{+} & \color{white} 0 & \color{white} 0 & \color{white} 0 & \color{white} 0 & \color{white} 0 & \color{white} 0 & \color{white} 0 & \color{white} 0 & \color{white} 0 & \color{white} 0 \\
     \color{white} 0 & \textbf{+} & \color{white} 0 & \color{white} 0 & \color{white} 0 & \color{white} 0 & \color{white} 0 & \color{white} 0 & \color{white} 0 & \color{white} 0 & \color{white} 0 & \color{white} 0 & \color{white} 0 \\
     \color{white} 0 & \color{white} 0 & \color{white} 0 & \color{white} 0 & \color{white} 0 & \color{white} 0 & \textbf{+} & \color{white} 0 & \color{white} 0 & \color{white} 0 & \color{white} 0 & \color{white} 0 & \color{white} 0 \\
     \color{white} 0 & \color{white} 0 & \color{white} 0 & \color{white} 0 & \color{white} 0 & \textbf{+} & \textbf{-} & \textbf{+} & \color{white} 0 & \color{white} 0 & \color{white} 0 & \color{white} 0 & \color{white} 0 \\
     \color{white} 0 & \color{white} 0 & \color{white} 0 & \color{white} 0 & \textbf{+} & \textbf{-} & \textbf{+} & \textbf{-} & \textbf{+} & \color{white} 0 & \color{white} 0 & \color{white} 0 & \color{white} 0 \\
     \color{white} 0 & \color{white} 0 & \color{white} 0 & \color{white} 0 & \color{white} 0 & \textbf{+} & \textbf{-} & \textbf{+} & \color{white} 0 & \color{white} 0 & \color{white} 0 & \color{white} 0 & \color{white} 0 \\
     \color{white} 0 & \color{white} 0 & \color{white} 0 & \color{white} 0 & \color{white} 0 & \color{white} 0 & \textbf{+} & \color{white} 0 & \color{white} 0 & \color{white} 0 & \color{white} 0 & \color{white} 0 & \color{white} 0 \\
     \color{white} 0 & \color{white} 0 & \color{white} 0 & \color{white} 0 & \color{white} 0 & \color{white} 0 & \color{white} 0 & \color{white} 0 & \color{white} 0 & \color{white} 0 & \color{white} 0 & \textbf{+} & \color{white} 0 \\
     \color{white} 0 & \color{white} 0 & \color{white} 0 & \color{white} 0 & \color{white} 0 & \color{white} 0 & \color{white} 0 & \color{white} 0 & \color{white} 0 & \color{white} 0 & \textbf{+} & \color{white} 0 & \color{white} 0 \\
     \color{white} 0 & \color{white} 0 & \color{white} 0 & \color{white} 0 & \color{white} 0 & \color{white} 0 & \color{white} 0 & \color{white} 0 & \color{white} 0 & \textbf{+} & \color{white} 0 & \textbf{-} & \textbf{+} \\
     \color{white} 0 & \color{white} 0 & \color{white} 0 & \color{white} 0 & \color{white} 0 & \color{white} 0 & \color{white} 0 & \color{white} 0 & \color{white} 0 & \color{white} 0 & \color{white} 0 & \textbf{+} & \color{white} 0}
\hspace{-0.2cm}\nearrow\hspace{-0.2cm}\SmallMatrix{
     \color{white} 0 & \color{white} 0 & \color{white} 0 & \color{white} 0 & \color{white} 0 & \color{white} 0 & \color{white} 0 & \color{white} 0 & \textbf{+} & \color{white} 0 & \color{white} 0 & \color{white} 0 & \color{white} 0 \\
     \color{white} 0 & \color{white} 0 & \color{white} 0 & \color{white} 0 & \textbf{+} & \color{white} 0 & \color{white} 0 & \color{white} 0 & \color{white} 0 & \color{white} 0 & \color{white} 0 & \color{white} 0 & \color{white} 0 \\
     \color{white} 0 & \color{white} 0 & \color{white} 0 & \textbf{+} & \color{white} 0 & \color{white} 0 & \color{white} 0 & \color{white} 0 & \color{white} 0 & \color{white} 0 & \color{white} 0 & \color{white} 0 & \color{white} 0 \\
     \color{white} 0 & \color{white} 0 & \textbf{+} & \color{white} 0 & \color{white} 0 & \color{white} 0 & \color{white} 0 & \color{white} 0 & \color{white} 0 & \color{white} 0 & \color{white} 0 & \color{white} 0 & \color{white} 0 \\
     \color{white} 0 & \textbf{+} & \color{white} 0 & \color{white} 0 & \color{white} 0 & \color{white} 0 & \color{white} 0 & \color{white} 0 & \textbf{-} & \color{white} 0 & \color{white} 0 & \color{white} 0 & \textbf{+} \\
     \color{white} 0 & \color{white} 0 & \color{white} 0 & \color{white} 0 & \color{white} 0 & \color{white} 0 & \textbf{+} & \color{white} 0 & \color{white} 0 & \color{white} 0 & \color{white} 0 & \color{white} 0 & \color{white} 0 \\
     \color{white} 0 & \color{white} 0 & \color{white} 0 & \color{white} 0 & \color{white} 0 & \textbf{+} & \textbf{-} & \textbf{+} & \color{white} 0 & \color{white} 0 & \color{white} 0 & \color{white} 0 & \color{white} 0 \\
     \color{white} 0 & \color{white} 0 & \color{white} 0 & \color{white} 0 & \color{white} 0 & \color{white} 0 & \textbf{+} & \color{white} 0 & \color{white} 0 & \color{white} 0 & \color{white} 0 & \color{white} 0 & \color{white} 0 \\
     \textbf{+} & \color{white} 0 & \color{white} 0 & \color{white} 0 & \textbf{-} & \color{white} 0 & \color{white} 0 & \color{white} 0 & \color{white} 0 & \color{white} 0 & \color{white} 0 & \textbf{+} & \color{white} 0 \\
     \color{white} 0 & \color{white} 0 & \color{white} 0 & \color{white} 0 & \color{white} 0 & \color{white} 0 & \color{white} 0 & \color{white} 0 & \color{white} 0 & \color{white} 0 & \textbf{+} & \color{white} 0 & \color{white} 0 \\
     \color{white} 0 & \color{white} 0 & \color{white} 0 & \color{white} 0 & \color{white} 0 & \color{white} 0 & \color{white} 0 & \color{white} 0 & \color{white} 0 & \textbf{+} & \color{white} 0 & \color{white} 0 & \color{white} 0 \\
     \color{white} 0 & \color{white} 0 & \color{white} 0 & \color{white} 0 & \color{white} 0 & \color{white} 0 & \color{white} 0 & \color{white} 0 & \textbf{+} & \color{white} 0 & \color{white} 0 & \color{white} 0 & \color{white} 0 \\
     \color{white} 0 & \color{white} 0 & \color{white} 0 & \color{white} 0 & \textbf{+} & \color{white} 0 & \color{white} 0 & \color{white} 0 & \color{white} 0 & \color{white} 0 & \color{white} 0 & \color{white} 0 & \color{white} 0}
\hspace{-0.2cm}\nearrow\hspace{-0.2cm}\SmallMatrix{
     \color{white} 0 & \color{white} 0 & \color{white} 0 & \color{white} 0 & \color{white} 0 & \color{white} 0 & \color{white} 0 & \textbf{+} & \color{white} 0 & \color{white} 0 & \color{white} 0 & \color{white} 0 & \color{white} 0 \\
     \color{white} 0 & \textbf{+} & \color{white} 0 & \color{white} 0 & \color{white} 0 & \color{white} 0 & \color{white} 0 & \color{white} 0 & \color{white} 0 & \color{white} 0 & \color{white} 0 & \color{white} 0 & \color{white} 0 \\
     \color{white} 0 & \color{white} 0 & \color{white} 0 & \color{white} 0 & \color{white} 0 & \color{white} 0 & \color{white} 0 & \color{white} 0 & \color{white} 0 & \textbf{+} & \color{white} 0 & \color{white} 0 & \color{white} 0 \\
     \color{white} 0 & \color{white} 0 & \color{white} 0 & \color{white} 0 & \color{white} 0 & \color{white} 0 & \color{white} 0 & \color{white} 0 & \color{white} 0 & \color{white} 0 & \textbf{+} & \color{white} 0 & \color{white} 0 \\
     \color{white} 0 & \color{white} 0 & \color{white} 0 & \color{white} 0 & \color{white} 0 & \color{white} 0 & \color{white} 0 & \color{white} 0 & \textbf{+} & \color{white} 0 & \color{white} 0 & \color{white} 0 & \color{white} 0 \\
     \color{white} 0 & \color{white} 0 & \color{white} 0 & \color{white} 0 & \color{white} 0 & \color{white} 0 & \color{white} 0 & \color{white} 0 & \color{white} 0 & \color{white} 0 & \color{white} 0 & \color{white} 0 & \textbf{+} \\
     \color{white} 0 & \color{white} 0 & \color{white} 0 & \color{white} 0 & \color{white} 0 & \color{white} 0 & \textbf{+} & \color{white} 0 & \color{white} 0 & \color{white} 0 & \color{white} 0 & \color{white} 0 & \color{white} 0 \\
     \textbf{+} & \color{white} 0 & \color{white} 0 & \color{white} 0 & \color{white} 0 & \color{white} 0 & \color{white} 0 & \color{white} 0 & \color{white} 0 & \color{white} 0 & \color{white} 0 & \color{white} 0 & \color{white} 0 \\
     \color{white} 0 & \color{white} 0 & \color{white} 0 & \color{white} 0 & \textbf{+} & \color{white} 0 & \color{white} 0 & \color{white} 0 & \color{white} 0 & \color{white} 0 & \color{white} 0 & \color{white} 0 & \color{white} 0 \\
     \color{white} 0 & \color{white} 0 & \textbf{+} & \color{white} 0 & \color{white} 0 & \color{white} 0 & \color{white} 0 & \color{white} 0 & \color{white} 0 & \color{white} 0 & \color{white} 0 & \color{white} 0 & \color{white} 0 \\
     \color{white} 0 & \color{white} 0 & \color{white} 0 & \textbf{+} & \color{white} 0 & \color{white} 0 & \color{white} 0 & \color{white} 0 & \color{white} 0 & \color{white} 0 & \color{white} 0 & \color{white} 0 & \color{white} 0 \\
     \color{white} 0 & \color{white} 0 & \color{white} 0 & \color{white} 0 & \color{white} 0 & \color{white} 0 & \color{white} 0 & \color{white} 0 & \color{white} 0 & \color{white} 0 & \color{white} 0 & \textbf{+} & \color{white} 0 \\
     \color{white} 0 & \color{white} 0 & \color{white} 0 & \color{white} 0 & \color{white} 0 & \textbf{+} & \color{white} 0 & \color{white} 0 & \color{white} 0 & \color{white} 0 & \color{white} 0 & \color{white} 0 & \color{white} 0}
\]

\[L(B) = 
\Matrix{
    4 & 11 & 5 & 10 & 6 & 1 & 7 & 13 & 12 & 2 & 3 & 9 & 8 \\
    11 & 7 & 7 & 7 & 7 & 7 & 7 & 7 & 3 & 4 & 10 & 5 & 9 \\
    5 & 7 & 7 & 7 & 7 & 7 & 7 & 7 & 7 & 13 & 4 & 10 & 3 \\
    10 & 7 & 7 & 7 & 7 & 7 & 7 & 7 & 7 & 6 & 13 & 4 & 2 \\
    6 & 7 & 7 & 7 & 7 & 7 & 7 & 7 & 7 & 7 & 7 & 3 & 12 \\
    1 & 7 & 7 & 7 & 7 & 7 & 7 & 7 & 7 & 7 & 7 & 7 & 13 \\
    7 & 7 & 7 & 7 & 7 & 7 & 7 & 7 & 7 & 7 & 7 & 7 & 7 \\
    13 & 7 & 7 & 7 & 7 & 7 & 7 & 7 & 7 & 7 & 7 & 7 & 1 \\
    12 & 3 & 7 & 7 & 7 & 7 & 7 & 7 & 7 & 7 & 7 & 7 & 6 \\
    2 & 4 & 13 & 6 & 7 & 7 & 7 & 7 & 7 & 7 & 7 & 7 & 10 \\
    3 & 10 & 4 & 13 & 7 & 7 & 7 & 7 & 7 & 7 & 7 & 7 & 5 \\
    9 & 5 & 10 & 4 & 3 & 7 & 7 & 7 & 7 & 7 & 7 & 7 & 11 \\
    8 & 9 & 3 & 2 & 12 & 13 & 7 & 1 & 6 & 10 & 5 & 11 & 4}
\]
\end{example}

\*

We conclude by posing the following problem.

\begin{problem}Is it possible to construct an $n \times n \times n$ ASHM $A$ for which $(n-2)^2 +4$ entries of $L(A)$ are equal, for $n > 7$?\end{problem}

\section*{Acknowledgements}

I would firstly like to thank Dr Trent Marbach, who first brought \cite{ashmbib} to my attention. I would also like to thank my PhD supervisors, Dr Rachel Quinlan and Dr Kevin Jennings, for supplying guidance on the presentation of this paper. This research was funded by a Hardiman Research Scholarship.


\begin{thebibliography}{9}
\bibitem{ashmbib} 
R. Brualdi, G. Dahl.
\textit{Alternating Sign Matrices and Hypermatrices, and a Generalization of Latin Squares}.
 Advances in Applied Mathematics, \textbf{95}(10): 1016, 2018.
 
\bibitem{lattice}
A. Lascoux, M.-P. Schützenberger
\textit{Treillis et bases des groupes de Coxeter}
Electronic Journal of Combinatorics, \textbf{3}(2 R): 1-35, 1996.

\bibitem{brualdibib} 
R. Brualdi, K. Kiernan, S. Meyer, M. Schroeder.
\textit{Patterns of Alternating Sign Matrices}.
Linear Algebra and its Applications, \textbf{438}(10): 3967-3990, 2013.
\end{thebibliography}
\end{document}